\documentclass{article}

\usepackage{verbatim}

\usepackage{amsfonts}

\usepackage{amsthm}

\usepackage{amssymb}

\usepackage{amsmath}

\usepackage{mathabx}

\usepackage{mathtools}

\usepackage{tikz}
\usetikzlibrary{decorations.pathreplacing}

\usepackage{enumerate}

\usepackage[all]{xy}

\usepackage{graphicx}

\usepackage{centernot}

\usepackage{stmaryrd}

\usepackage[pagebackref,colorlinks,linkcolor=blue,citecolor=blue,urlcolor=blue,hypertexnames=true]{hyperref}

\usepackage[pagebackref,hypertexnames=true]{hyperref}

\newcommand{\N}{\mathbb{N}}

\newcommand{\Z}{\mathbb{Z}}

\newcommand{\R}{\mathbb{R}}

\newcommand{\Q}{\mathbb{Q}}

\newcommand{\C}{\mathbb{C}}

\newcommand{\cind}{\mathrm{CInd}}

\newcommand{\Hawaii}{Hawai\kern.05em`\kern.05em\relax i}

\theoremstyle{plain}
\newtheorem{theorem}{Theorem}[section]
\newtheorem{lemma}[theorem]{Lemma}
\newtheorem{corollary}[theorem]{Corollary}
\newtheorem{proposition}[theorem]{Proposition}

\newtheorem{definition-theorem}[theorem]{Definition / Theorem}

\newtheorem{question}[theorem]{Question}

\newtheorem*{conjecture*}{Conjecture}
\newtheorem*{theorem*}{Theorem}

\theoremstyle{definition}
\newtheorem{definition}[theorem]{Definition}
\newtheorem{example}[theorem]{Example}

\newtheorem*{acks*}{Acknowledgements}

\theoremstyle{remark}
\newtheorem{remark}[theorem]{Remark}

\newtheorem*{example*}{Example}  
\newtheorem*{remark*}{Remark}

\begin{document}
\title{The Local Lifting Property, Property FD, and stability of approximate representations}
\author{Francesco Fournier-Facio and Rufus Willett}
\date{\today}

\maketitle

\begin{abstract}
    We establish Kirchberg's Local Lifting Property and Lubotzky--Shalom's Property FD for classes of finitely generated groups of central importance in geometric and combinatorial group theory: $3$-manifold groups, limit groups, and certain one-relator groups and right-angled Artin groups. We deduce that such groups are very flexibly stable, with respect to normalized unitarily invariant norms. In the appendix, we show that these groups also have Kechris's property (E)MD, and hence are stable in finite actions, in the selse of Gohla--Thom. The exposition is made accessible to operator algebraists and group theorists alike.
\end{abstract}

\tableofcontents

\section{Introduction}

Let $\Gamma$ be a group.  Its group $C^*$-algebra $C^*\Gamma$ is a completion of the complex group algebra $\C[\Gamma]$, determined by the following property: the unitary representation theory of $\Gamma$ is exactly the same as the $*$-representation theory of $C^*\Gamma$.  This correspondence allows one to use analytic properties of $C^*\Gamma$ to study the group $\Gamma$.

In this paper we look at two such properties: Kirchberg's \emph{local lifting property (LLP)} \cite{Kirchberg:1993aa} as well as the stronger \emph{lifting property (LP)}, and \emph{property FD} of Lubotzky--Shalom \cite{Lubotzky:2004xw} as well as the weaker \emph{property RFD}.  These properties are quite different, but have both recently come into prominence for applications to the approximate unitary representations of the group, and the subject of representation stability.  They are also of interest to operator algebraists in their own right.

The LLP and property FD are well-understood for free and amenable groups, but few examples were known outside these cases.  The first main purpose of this paper is to show that some of the most important classes of groups from geometric and combinatorial group theory -- $3$-manifold groups, limit groups, and many one-relator and right-angled Artin groups -- have the LP and property FD.  Although the LP and property FD are in principle quite different, the techniques we use for both are similar: the key point is to establish and use various permanence properties such as closure under appropriate amalgamated products, and appropriate extensions.

Our second goal is to give applications to representation stability for these classes of groups.  We use the LLP and property FD to show versions of `very flexible stability' for the classes of groups we cover in our first goal. 

A third goal -- which to some extent accounts for the length of the paper\footnote{And to a great extent accounts for the length of the bibliography.} -- is to make the operator algebraic techniques we use accessible to geometric group theorists, and vice versa; we thus hope to encourage more interactions between these two fields.  As such, we also include some `well-known' and `folklore' results from both geometric group theory and operator algebras that are difficult to extract from the literature for outsiders to these fields. 

We now discuss the motivations from representation stability.

\subsubsection*{Property RFD and the LLP in representation stability}

For the discussion below, we will focus on the LLP, and a weaker version of property FD called RFD.  Rather than give the definitions of the LLP and RFD formally here, we will instead give a rough idea of how they are used in the subject of representation stability and postpone the formal definitions to the body of the paper: see Definitions \ref{llp def} (for the LP and LLP), \ref{fd def} (for FD) and \ref{rfd def} (for RFD).  Let $\Gamma$ be a countable group; for simplicity we will say that $\Gamma$ is LLP or RFD if its $C^*$-algebra $C^*\Gamma$ has those properties.

Let $U_n$ denote the unitary group of $\C^n$, equipped with the operator norm.
An \emph{asymptotic representation} of $\Gamma$ is a sequence $(\phi_n \colon \Gamma\to U_{k_n})_{n=1}^\infty$ of maps such that $\|\phi_n(gh)-\phi_n(g)\phi_n(h)\|\xrightarrow{n\to\infty} 0$.  The subject of representation stability, which has bloomed in recent years (see for example the introduction of \cite{ultrametric:stability}), asks to what extent $(\phi_n)_n$ can be approximated by a sequence of honest representations.

The LLP implies that one can do this, in a rather weak sense.  Precisely, for each $n$ there exists a Hilbert space $H_n$, a unitary representation $\pi_n \colon \Gamma\to U(H_n)$ and an isometric inclusion $v_n \colon \C^{k_n}\to H_n$ such that 
$$
\|v_n^*\pi_n(g)v_n-\phi_n(g)\|\xrightarrow{n\to\infty}0
$$
for all $g\in \Gamma$.  The operator $v_n^*\pi_n(g)v_n$ is the compression of $\pi_n(g)$ onto the range of $v_n$, transported back to an operator on $\C^{k_n}$.  Thus, one can approximate the asymptotic representation $(\phi_n)_n$ by a sequence of `corners' of honest (but possibly infinite-dimensional) representations.

On the other hand, RFD says that any unitary representation of $\Gamma$ can be approximated in an appropriate sense by finite-dimensional representations.  It (together with the LLP) implies in particular that we can take all the Hilbert spaces $H_n$ appearing in the above discussion to be finite-dimensional. 

Putting this discussion together and making it precise, one has the following result.

\begin{theorem}\label{vfs intro}
Let $\Gamma$ be LLP and RFD.  Let $(\phi_n \colon \Gamma\to U_{k_n})$ be an asymptotic representation of $\Gamma$.

Then there exists a sequence of finite-dimensional Hilbert spaces $H_n$, representations $\pi_n \colon \Gamma\to U(H_n)$, and isometric inclusions $v_n \colon \C^{k_n}\to H_n$ such that $\|v_n^*\pi_n(g)v_n-\phi_n(g)\|\xrightarrow{n\to\infty}0$.
\end{theorem}    

In the language of representation stability, this precisely says that if $\Gamma$ is LLP and RFD, then it is \emph{very flexibly stable} with respect to the operator norm: compare Definition \ref{stab def}.  Versions of this had previously appeared in the literature, see for example \cite[Theorem 6.4]{eckhardt:shulman}.  Our set-up is more flexible in several ways:
\begin{itemize}
\item We can allow for other norms on the unitary groups satisfying reasonable conditions: for example, Theorem \ref{vfs intro} applies equally well to approximate representations with respect to the Hilbert--Schmidt norm, and the normalized Schatten $p$-norms.
\item We can allow for other families of representations in place of the finite-dimensional ones: for example, if $\Gamma$ has property FD in the sense of Lubotzky--Shalom, we can assume that the $\pi_n$ are finite-dimensional, and also factor through a finite quotient of $\Gamma$.
\item Under appropriate $K$-theoretic (or cohomological) assumptions, we can force the `complementary part' of $\pi_n$ to the corner $\phi_n$ to be an honest representation.
\end{itemize}

See Section \ref{s:stability} below for details, particularly Theorem \ref{stability thm}, and Corollaries \ref{examples very flexibly} and \ref{examples alg comp}.

\subsubsection*{New examples with the (L)LP and property (R)FD}

As discussed above, our first goal in this paper is to expand the class of groups with the LLP and RFD.  Actually, we typically prove something stronger: the LP rather than the LLP, and property FD rather than RFD.  

For amenable groups, the LLP has been known for fifty years, thanks to the Choi--Effros lifting theorem \cite{Choi:1976aa}.  Property FD for amenable groups is equivalent to residual finiteness, so also quite well-understood.

Both properties were also known for free groups, and known to be closed under some important operations such as free products (for the (L)LP and RFD, and we prove this also for FD) and certain semi-direct products with amenable groups: see Sections \ref{s:plp} and \ref{s:fd} for detailed background and references.  Our main results substantially increase the known class of examples as follows.

\begin{theorem}
\label{intro fd llp}
If $\Gamma$ is finitely generated and satisfies one of the following properties, then its $C^*$-algebra has the LP and it has property FD.
    \begin{enumerate}[(i)]
        \item\label{intro fd llp manifold} $\Gamma = \pi_1(M)$, where $M$ is a connected manifold of dimension at most $3$;
        \item\label{intro fd llp vfbc} $\Gamma$ is virtually free-by-cyclic;
        \item\label{intro fd llp 1rel} $\Gamma$ is a one-relator group, and $\Gamma$ has either torsion, negative immersions, non-trivial center or a small cancellation relation;
        \item\label{intro fd llp limit} $\Gamma$ is a limit group\footnote{Goldbring--Seward--Tucker-Drob \cite[Theorem 4.4]{treeable} previously showed that limit groups have ``property MD'', which implies FD: we discuss property MD more below.};
        \item\label{intro fd llp raag} $\Gamma$ is a right-angled Artin group on a chordal graph.
    \end{enumerate}
In particular, all the groups above are very flexibly stable with respect to the operator norm, the Hilbert--Schmidt norm, or the normalized Schatten $p$-norm.
\end{theorem}

\begin{proof}
    \eqref{intro fd llp manifold} is Example \ref{manifold lp} and Theorem \ref{manifold fd}.
    \eqref{intro fd llp vfbc} is Corollary \ref{vfbc} and Proposition \ref{vfbc fd}.
    \eqref{intro fd llp 1rel} follows from \eqref{intro fd llp vfbc} and Examples \ref{or tor}, \ref{or ni}, \ref{or center} and \ref{or sc}.
    \eqref{intro fd llp limit} is Example \ref{limit lp} and Proposition \ref{limit fd}.
    \eqref{intro fd llp raag} is Example \ref{raag lp} and Proposition \ref{raag fd}.
\end{proof}

We also give some new examples with the LP, but that do not have property FD (as they are not even residually finite); most prominent amongst these are the Baumslag--Solitar groups $BS(n,m)$ (Example \ref{GBS}) and the Baumslag--Gersten group (Proposition \ref{prop:BG}).

The proof of Theorem \ref{intro fd llp} relies on deep results in geometric and combinatorial group theory, as well as some new permanence properties that we establish in this paper.  Most prominently amongst our new permanence properties, we show that the LLP is preserved by extensions by amenable groups, and that property FD is preserved by certain amalgamated free products.

\subsubsection*{Property MD}

The representation-theoretic property FD has a measurable analogue, Kechris's \emph{property MD} \cite{MD:kechris}, see also the survey \cite{MD:burton:kechris}. Property MD in turn implies a stability property for approximate representations in permutations, called \emph{stability in finite actions} \cite{gohla:thom}. It turns out that all of our proofs of property FD can be adapted to this setting.

\begin{theorem}
\label{intro thm md}
    The groups from Theorem \ref{intro fd llp} also have property (E)MD, and hence are stable in finite actions.
\end{theorem}

Some of the permanence properties for FD that we use to prove Theorem \ref{intro fd llp} have MD analogs in the literature \cite{MD:bowen:tuckerdrob, treeable}, but some do not, so to prove Theorem \ref{intro thm md} we also need to prove new permanence properties. Since these results and the methods of proof are quite different from the rest of the paper, we leave them, as well as a more in-depth discussion and motivation of this property, to Appendix \ref{appendix md}. Theorem \ref{intro thm md} is proved as Theorem \ref{MD list}.

\subsubsection*{Open questions}

Let us give a brief summary some questions that are left open by our work, and seem interesting. 

\begin{question}
\label{q:or:llp}
Do all one relator groups have the (L)LP?
\end{question}

See Remark \ref{LP:soficity} for a large class of examples for which this question is open.

There are non-(R)FD one-relator groups due to the existence of non residually finite one-relator groups, but it is reasonable to ask the following.

\begin{question}
\label{q:or:fd}
Is every residually finite one-relator group (R)FD?
\end{question}

In another direction, the following seems to be a particularly challenging group for which the (L)LP is open.

\begin{question}
\label{q:ftimesf}
Does $F_2\times F_2$ have the (L)LP?
\end{question}

The group $F_2\times F_2$ is known not to have (R)FD thanks to the negative solution of the Connes embedding problem \cite{Salle:2023aa}: see \cite{Ozawa:2004ab}.  However, the proof is very difficult and indirect, and it would be good to have a more concrete understanding.

The presence of $F_2 \times F_2$ subgroups is therefore a significant obstacle for generalizing the (L)LP to right-angled Artin groups beyond the chordal case, and a genuine obstruction for property (R)FD. Excluding this case the problem might be approachable.

\begin{question}
\label{q:raag}
Let $\Gamma$ be a right-angled Artin group whose defining graph has no induced squares. Is $\Gamma$ (L)LP? Is it (R)FD?
\end{question}

The first non-chordal example is the pentagon, for which we already do not know the answer (see \cite[Problem 35]{sri:problems} for a related question).

\begin{question}
Is there an example of an a-T-menable group without the (L)LP?
\end{question}

Note that all the examples from Theorem \ref{intro fd llp} are a-T-menable: see the discussion in the proof of Corollary \ref{examples alg comp} below.  From the previous question, $F_2\times F_2$ could be an example.  See also the discussion in Remark \ref{llp t} relative to this and the next question.  

There are a-T-menable groups that are not (R)FD, as there are non residually finite a-T-menable groups.  There are also more subtle examples of a-T-menable residually finite groups that are not FD: for example $SL(2,\Z[1/p])$ is not FD as it has property ($\tau$) (see for example \cite[Example 4.3.3 E]{Lubotzky:1994tw}) but not property (T), and it is a-T-menable (see for example \cite[Theorem 5.1]{Guentner:2005xr}) and residually finite (one can see this by mapping it to appropriate congruence quotients, or by using that finitely generated linear groups are residually finite \cite{malcev}).

\begin{question}
\label{q:T}
Is there an example of an infinite property (T) group (or just a group with property (T) relative to an infinite subgroup) with the (L)LP?
\end{question}

A particular case of the next question also appears in \cite[Question 6.5]{Lubotzky:2004xw}.

\begin{question}
Is there an example of an infinite property (T) group (or just a group with property (T) relative to an infinite subgroup) with (R)FD?
\end{question}

A particularly interesting class of groups with (relative) property (T) is given by higher rank arithmetic groups.  For these (R)FD typically fails: see \cite{Bekka:1999kx} and \cite[Theorem 3.1]{Lubotzky:2004xw}.  The LLP is also known to fail for some of these groups (for example, $SL(n,\Z)$ for $n>2$) through applications of results from  \cite{Ioana:2020aa}.  It seems to be open whether higher rank arithmetic groups always fail the (L)LP.

Compare Remark \ref{RFD:MAP:FD:RF} below for the next question.

\begin{question}
\label{q:rfdnotfd}
Is there a finitely generated group that is RFD, but not FD?
\end{question}

The next question is more purely operator algebraic: it is motivated by our results in Subsection \ref{subs:central}, but seems natural in its own right.

\begin{question}
\label{q:cocycles}
Let $\Gamma$ be a group with the (L)LP, and let $\sigma\in Z^2(\Gamma;S^1)$ be a normalized $S^1$-valued $2$-cocycle.  Does the twisted group $C^*$-algebra $C^*(\Gamma;\sigma)$ have the (L)LP? 
\end{question}

The analogous question has a negative answer for (R)FD: for example, if $\Gamma=\Z^2$ there are cocycles for which $C^*(\Gamma;\sigma)$ is a simple (infinite-dimensional) $C^*$-algebra and so has no non-zero finite-dimensional representations (see for example \cite[Example 2.8.15]{Echterhoff:2009jo}).

Finally, we note that all of our results about the LLP needed a countability assumption in order to be strengthened to the LP. We do not know to what extent this is always necessary.

\begin{question}
\label{q:lp:uncountable}
Is there an uncountable group with the LP, and in particular, does a free group, or a free abelian group, of uncountable rank have the LP?
\end{question}

This also relates to the question of whether the LP and LLP are non-equivalent for \emph{group} $C^*$-algebras: see Remark \ref{sep rem} below.

\subsubsection*{Outline of the paper}

Throughout the paper, we have tried to minimize the background required of the reader.  On $C^*$-algebras, we expect the reader to understand the basic theory up to what a $*$-representation of a $C^*$-algebra is, and the GNS construction relating states and cyclic representations; the reader will also have to take some facts about $C^*$-algebra tensor products on faith if they do not already have that background.  We have generally not assumed much knowledge of geometric or combinatorial group theory beyond what one might see in a first course in algebraic topology; in particular, the necessary background from Bass--Serre theory is recalled the first time it is used (Subsection \ref{ss:gog}).  Having said that, some examples will be difficult to understand without some more background in e.g.\ three manifold theory or the theory of amenable or property (T) groups.

The first three sections of the paper (after this introduction) focus on the (L)LP.  
In Section \ref{s:lp bg} we summarize necessary background on the (L)LP for $C^*$-algebras, including the basic definitions.  In Section \ref{s:plp} we discuss permanence properties for the class of groups whose $C^*$-algebra has the (L)LP, as well as summarizing the known examples and non-examples.  Much of this is known, or folklore, but the material on central extensions, extensions by amenable groups, and some of the material on graphs of groups is new; these will be the principal ingredients we use to give new examples.  In Section \ref{s:ex} we discuss our new examples of groups with the (L)LP, including all the groups in Theorem \ref{intro fd llp}.  

In Sections \ref{s:fd} we switch subjects to property FD.  We first recall (and to some extent generalize) some necessary background on the representation theory of $C^*$-algebras and density in the Fell topology, and then summarize the known results.  Our one new permanence result on property FD shows that it is preserved under free products; we actually show something a little more general than this in order to cover certain amalgamated products that are important for our applications.  In Section \ref{s:ex fd} we give our new examples with property FD, including all the groups in Theorem \ref{intro fd llp}: most of this is done using similar ideas to those we used for the (L)LP; the exception is the material for three-manifold groups, which requires some new ideas in the case of closed graph manifolds and is rather more substantial for FD than for the (L)LP.

In Section \ref{s:stability}, we discuss applications to representation stability.  In particular, we establish a rather general version of Theorem \ref{vfs intro}, and also discuss the connections with cohomological assumptions.

Finally, in Appendix \ref{appendix md}, we discuss property MD, and prove Theorem \ref{intro thm md}.

\begin{acks*}
FFF is supported by the Herchel Smith Postdoctoral Fellowship Fund. RW is supported by the US NSF (DMS 2247968) and the Simons Foundation (MP-TSM-00002363).  The authors thank Vadim Alekseev, Lewis Bowen, Will Cohen, Kristin Courtney, Alon Dogon, Siegfried Echterhoff, Sam Fisher, Jonathan Fruchter, Liam Hanany, Marco Linton, Tatiana Shulman, Pieter Spaas, Andreas Thom, Robin Tucker-Drob, Henry Wilton and Julian Wykowski for useful conversations. They also would like to thank the Isaac Newton Institute for Mathematical Sciences, Cambridge, for support and hospitality during the program Operators, Graphs, Groups, where work on this paper was undertaken. This work was supported by EPSRC grant EP/Z000580/1.
\end{acks*}

\section{Background on the (L)LP for $C^*$-algebras}\label{s:lp bg}

In this section, we define the (local) lifting property ((L)LP) for $C^*$-algebras, and record some relevant background.  The modern theory\footnote{There are several much earlier antecedents, however: ideas around the LP go back at least to work of Arveson \cite{Arveson:1974aa}, and around the LLP at least to work of Effros--Haagerup \cite[Theorem 3.2]{Effros:1985aa}.} of these notions starts with Kirchberg's seminal work \cite{Kirchberg:1993aa}.  See \cite[Chapter 13]{Brown:2008qy} and \cite[Chapter 9]{Pisier:2020aa} for relatively recent textbook treatments.

The following notion will not be used in a significant way, but is needed to state the definition of the local lifting property.

\begin{definition}
An \emph{operator system} is a self-adjoint subspace\footnote{Some authors also require it to be closed; this distinction will not be important to us.} of a unital $C^*$-algebra that contains the unit.
\end{definition}

The significance of operator systems is that they are the smallest subspaces of $C^*$-algebras where there is a good theory of positive elements, and of linear maps preserving positivity.  Here an element in an operator system is \emph{positive} if it is positive in the ambient $C^*$-algebra.  Note also that if $E\subseteq A$ is an operator system, then $M_n(E)\subseteq M_n(A)$ is also an operator system inside the $C^*$-algebra of $n\times n$ matrices over $A$, and thus we may talk about positive elements in matrices over $E$.

\begin{definition}
Let $E$ be an operator system or a $C^*$-algebra\footnote{The reason for the ``or a $C^*$-algebra'' hypothesis is to allow non-unital $C^*$-algebras; a unital $C^*$-algebra is itself an operator system.}, and let $A$ be a $C^*$-algebra.  A map $\phi \colon E\to A$ is \emph{completely positive} (cp) if it is linear and for all $n$, the induced map 
$$
\phi^{(n)} \colon M_n(E)\to M_n(A)
$$
defined by applying $\phi$ entry-wise takes positive elements to positive elements.  It is \emph{contractive completely positive} (ccp) if it is cp and the norm of each $\phi^{(n)}$ is at most one.  It is \emph{unital completely positive} (ucp) if $E$ is an operator system, and $\phi$ is cp and unital.
\end{definition}

\begin{remark}\label{ucp vs ccp 0}
There are close relations between ucp and ccp maps.  First note that a ucp map is automatically ccp: see for example \cite[Proposition 3.2]{Paulsen:2003ib}.  

As a sort of converse, recall that any $C^*$-algebra admits a \emph{unitization} $A^+$; which is the smallest unital $C^*$-algebra containing $A$ if $A$ is non-unital, and is isomorphic to $A\oplus \C$ if $A$ is unital: see for example \cite[Section 2.9]{Arveson:2002aa}.  If $\phi \colon A\to B$ is a linear map between $C^*$-algebras and $B$ is unital, then there is a unique unital linear map $\phi^+ \colon A^+\to B$ that restricts to $\phi$ on $A$.  One has that $\phi^+$ is ucp if and only if $\phi$ is ccp: see for example \cite[Proposition 2.2.1]{Brown:2008qy}.
\end{remark}

\begin{remark}\label{stinespring}
For an outsider to operator algebras, it is probably mysterious why one should care about ccp or ucp maps, as opposed to $*$-homomorphisms.   The main reason for their importance is \emph{Stinespring's dilation theorem}: this says that for a ucp map $\phi \colon A\to \mathcal{B}(H)$ from a unital $C^*$-algebra to the bounded operators on Hilbert space there are an isometric inclusion of Hilbert spaces $v \colon H\to H'$ and a unital $*$-homomorphism $\pi \colon A\to \mathcal{B}(H')$ such that $\phi(a)=v^*\pi(a)v$\footnote{Conversely, it is straighforward to see that any map of this form is ucp.} for all $a\in A$.  The original reference is \cite{Stinespring:1955aa}; see for example \cite[Theorem 1.5.3 and Remark 1.5.4]{Brown:2008qy} for a modern textbook treatment.

Stinespring's theorem has many consequences: an application we will use is that $C^*$-algebra tensor products are functorial for ccp maps (see for example \cite[Theorem 3.5.3]{Brown:2008qy}). 
We will not explicitly use Stinespring's theorem until much later in this note (Section \ref{s:stability}).
\end{remark}

\begin{definition}\label{llp def}
Let $A$ be a unital $C^*$-algebra, let $\pi \colon B\to B/J$ be a quotient $*$-homomorphism of $C^*$-algebras, and let $\phi \colon A\to B/J$ be a ccp map. 

The map $\phi$ is \emph{liftable} if the dashed arrow in the diagram below can be filled in with a ccp map
$$\xymatrix{ & B \ar[d]^-\pi \\
A \ar[r]_-\phi \ar@{-->}[ur] & B/J}$$
so that the diagram commutes.

The map $\phi$ is \emph{locally liftable} if for any finite dimensional operator system $E$ in $A$ with inclusion map $\iota \colon E\to A$, the dashed arrow in the diagram below
$$
\xymatrix{ & & B \ar[d]^-\pi \\
E  \ar@{-->}[urr]  \ar[r]_-\iota & A \ar[r]_-\phi  & B/J}
$$
can be filled in with a ccp map so that the diagram commutes.

A unital $C^*$-algebra $A$ has the \emph{lifting property} (LP) (respectively, \emph{local lifting property} (LLP)) if every ccp map from $A$ into a quotient $C^*$-algebra is liftable (respectively, locally liftable). A non-unital $C^*$-algebra has the (local) lifting property if its unitization $A^+$ does.
\end{definition}

\begin{remark}\label{ucp vs ccp}
Our definitions of the (L)LP are based on \cite[Definition 13.1.1]{Brown:2008qy}, and are not quite the same as those of Kirchberg \cite[Section 2, page 453]{Kirchberg:1993aa}, who requires that ucp maps into quotients of unital $C^*$-algebras (locally) lift to ucp maps.  The `ucp' and `ccp' variants are equivalent by \cite[Lemma 13.1.2]{Brown:2008qy}.
\end{remark}

We next give some fundamental results on the LP and LLP.  For the next definition, let $A$ and $B$ be $C^*$-algebras, and let $A\odot B$ be their algebraic tensor product (over $\C$) equipped with its canonical $*$-algebra structure as in \cite[Section 3.1]{Brown:2008qy}.  In general, $A\odot B$ admits many completions to a $C^*$-algebra: the two most important are the \emph{maximal} completion $A\otimes_{\max} B$ and the \emph{minimal} completion $A\otimes B$.  As the names suggest, $A\otimes_{\max} B$ is the completion for the  largest $C^*$-norm on $A\odot B$, and $A\otimes B$ is the completion for the smallest $C^*$-norm (the latter is \emph{Takesaki's theorem}, from \cite{Takesaki:1963aa}): see \cite[Chapter 3]{Brown:2008qy} for background.

\begin{definition}\label{nuc pair}
A pair $(A,B)$ of $C^*$-algebras is \emph{nuclear} if the canonical quotient map $A\otimes_{\max}B\to A\otimes B$ is the identity.  A $C^*$-algebra $A$ is \emph{nuclear} if $(A,B)$ is a nuclear pair for any $C^*$-algebra $B$.
\end{definition}

The following is a fundamental theorem of Kirchberg \cite[Proposition 2.2]{Kirchberg:1993aa}\footnote{Kirchberg's theorem was since given a simpler proof by Pisier \cite{Pisier:1996aa}; Pisier's proof forms the basis for multiple textbook expositions such as \cite[Corollary 13.2.5]{Brown:2008qy} or \cite[Theorem 9.38]{Pisier:2020aa}.}; it will be a useful criterion to deduce the LLP. 

\begin{theorem}[Kirchberg]\label{kir the}
A $C^*$-algebra $A$ has the LLP if and only if $(A,\mathcal{B}(H))$ is a nuclear pair, where $H$ is a separable infinite-dimensional Hilbert space. \qed
\end{theorem}

There is a useful recent characterization of the LP for separable $C^*$-algebras, due to Pisier \cite[Theorem 1.2]{Pisier:2022aa}; it can sometimes be used analogously to Theorem \ref{kir the}.  For the statement, recall that if $\{D_i\}_{i\in I}$ is a family of $C^*$-algebras, then $\prod_{i\in I}D_i$ denotes the $C^*$-algebra of bounded tuples $(d_i)_{i\in I}$ with $d_i\in D_i$, equipped with pointwise operations and the supremum norm.

\begin{theorem}[Pisier]\label{pis the}
Let $A$ be a separable $C^*$-algebra.  Then $A$ has the LP if and only if for any family $\{D_i\}_{i\in I}$ of $C^*$-algebras, the canonical map
$$
\Bigg(\prod_{i\in I}D_i\Bigg)\otimes_{\max} A\to \prod_{i\in I}(D_i\otimes_{\max} A)
$$
is injective.\qed
\end{theorem}

See \cite[Section 3]{Enders:2024aa} for other useful characterizations of the LP.  We will not, however, use those in this paper.

We next recall the \emph{Choi--Effros lifting theorem} from  \cite{Choi:1976aa}; this gives a fundamental class of $C^*$-algebras with the LP.  See for example \cite[Theorem C.3]{Brown:2008qy} or \cite[Theorem 3.3.6]{Higson:2000bs} for textbook expositions\footnote{Both of these textbook expositions are based on Arveson's simpler approach \cite[Theorem 7 and Corollary on page 351]{Arveson:1977aa} to the Choi--Effros lifting theorem.}.  One can also deduce it as an immediate consequence of Pisier's theorem \ref{pis the} above.

\begin{theorem}[Choi--Effros]\label{cet}
Any separable nuclear $C^*$-algebra has the LP. \qed
\end{theorem}

\begin{remark}
The Choi--Effros lifting theorem significantly predates Kirchberg's work on the (L)LP.  Following an idea of Arveson \cite{Arveson:1974aa}, the Choi--Effros lifting theorem is foundational for $K$-homology of $C^*$-algebras, which was the original motivation.
\end{remark}

\begin{corollary}\label{ce llp}
Any nuclear $C^*$-algebra satisfies the LLP.
\end{corollary}

\begin{proof}
This is immediate from Theorem \ref{kir the}.  Alternatively, it can be deduced from the Choi--Effros lifting theorem and the fact\footnote{This is not quite obvious.  The difficulty is that a (separable) $C^*$-subalgebra of a nuclear $C^*$-algebra need not be nuclear.  However it is always contained in a separable nuclear $C^*$-subalgebra: see for example \cite[Exercise 2.3.8]{Brown:2008qy}.}  that any finite-dimensional operator system in a nuclear $C^*$-algebra is contained in a separable nuclear $C^*$-subalgebra.
\end{proof}

We conclude this discussion with a useful result of Kirchberg that gives a sufficient condition for the LLP to imply the existence of a global lift.    The original reference is \cite[Proposition 2.2 (iv)]{Kirchberg:1993aa}\footnote{To deduce the result of Theorem \ref{llp to lp} from \cite[Proposition 2.2 (iv)]{Kirchberg:1993aa}, one needs also the fact that the WEP passes to ideals.}, or see \cite[Corollary 3.12]{Ozawa:2004ab} for a relatively self-contained proof.  The statement requires the notion of a ``QWEP'' $C^*$-algebra: what exactly this means is not important for us, but see \cite[page 452]{Kirchberg:1993aa} for the definition, and Corollary \ref{mat lift} below for an application.

\begin{theorem}[Kirchberg]\label{llp to lp}
Let $A$ be a separable unital $C^*$-algebra with the LLP, let $B$ be a unital QWEP $C^*$-algebra, and let $\phi \colon A\to B/J$ be a ucp map from $A$ to a quotient of $B$.  Then $\phi$ admits a ucp lift $\widetilde{\phi} \colon A\to B$.\qed
\end{theorem}

From the next section we will focus on group $C^*$-algebras, and we will typically present results on the LLP in general, and the same results on the LP under an additional countability assumption.  The following remark provides some justification for why one might expect that.

\begin{remark}\label{sep rem}
Kirchberg \cite[Proposition 8.1]{Kirchberg:1993aa} showed that a positive solution to the Connes embedding problem (a famous problem in operator algebra theory) is equivalent to every separable $C^*$-algebra being QWEP.  It would follow from this and Theorem \ref{llp to lp} that the LP and LLP are equivalent for separable $C^*$-algebras\footnote{Here we use that to check the LP for a separable unital $C^*$-algebra, it suffices to check that ccp maps to quotients $B/J$ with $B$ separable always admit ccp lifts.}. 
This provides some evidence for the equivalence of the LP and LLP in general for separable $C^*$-algebras.  However, it certainly cannot provide a proof, as the Connes embedding problem was recently shown to have a negative solution: see \cite{Salle:2023aa} and references discussed there.  

On the other hand, the LP and LLP are not equivalent for non-separable $C^*$-algebras\footnote{We do not know if they fail to be equivalent for non-separable \emph{group} $C^*$-algebras (Question \ref{q:lp:uncountable})}.
For example, if $\ell^\infty(\N)$ is the $C^*$-algebra of bounded functions from $\N$ to $\C$, and $C_0(\N)$ is the ideal of functions that tend to zero at infinity, then the quotient map $\ell^\infty(\N)\to \ell^\infty(\N)/C_0(\N)$ does not admit a bounded linear (in particular, ccp) splitting: see for example \cite[Exercise 13.1.1]{Brown:2008qy}.  Hence the $C^*$-algebra $\ell^\infty(\N)/C_0(\N)$ does not have the LP.  It is nuclear (as commutative: see for example \cite[Proposition 2.4.2]{Brown:2008qy}), however, so does have the LLP.  Note that the $C^*$-algebra $\ell^\infty(\N)$ is also QWEP (as injective, or as nuclear), so this also shows the necessity of the separability assumption on $A$ in Theorem \ref{llp to lp}.
\end{remark}

\section{Permanence properties for the (L)LP}\label{s:plp}

In this section, we survey known facts on the (L)LP for group $C^*$-algebras, and prove some new results. These will then be applied in the next section to cover several important classes of groups from topology, as well as geometric and combinatorial group theory.

\subsection{Group $C^*$-algebras and crossed products}

We need a basic definition from the representation theory of $C^*$-algebras.

\begin{definition}\label{dg rep}
A representation $\pi \colon A\to \mathcal{B}(H)$ of a $*$-algebra on a Hilbert space is \emph{nondegenerate} if $\pi(A)H$ is dense in $H$\footnote{If $A$ is unital, this is the same as the representation being unital.}.  It is \emph{degenerate} if it is not nondegenerate.
\end{definition}

For the next definition, let $\C[\Gamma]$ denote the complex group $*$-algebra, and recall that nondegenerate $*$-representations of $\C[\Gamma]$ are the same thing as (linear extensions of) unitary representations of $\Gamma$.  We will abuse notation slightly by not distinguishing between the two.

\begin{definition}\label{gp c*}
Let $\Gamma$ be a discrete group.  Then its \emph{group $C^*$-algebra}\footnote{Also called the \emph{maximal} or \emph{full} group $C^*$-algebra if one wants to distinguish it from other possible $C^*$-completions of $\C[\Gamma]$.} is the completion of the complex group $*$-algebra $\C[\Gamma]$ for the norm defined for $a\in \C[\Gamma]$ by
$$
\|a\|\coloneqq\sup\{\|\pi(a)\|_{\mathcal{B}(H)}\mid \pi \colon \Gamma\to \mathcal{B}(H) \text{ a unitary representation}\}.
$$ 
We write $C^*\Gamma$ or $C^*(\Gamma)$ for the group $C^*$-algebra of $\Gamma$.
\end{definition}

\begin{definition}
A discrete group $\Gamma$ has the \emph{(local) lifting property} ((L)LP) if its $C^*$-algebra $C^*\Gamma$ does.  We will also sometimes abuse grammar and say for example ``$\Gamma$ is LLP'' if that makes for smoother sentences.
\end{definition}

Note that group $C^*$-algebras are functorial: a homomorphism of groups induces a $*$-homomorphism of $C^*$-algebras.  The following basic facts about the maps on $C^*$-algebras induced by subgroup inclusions will be important for us: see for example \cite[Proposition 2.5.8 and Corollary 2.5.12]{Brown:2008qy} for a proof.

\begin{lemma}\label{subgp rem}
Let $\Lambda$ be a subgroup of a discrete group $\Gamma$.  Then the inclusion map $\Lambda\to \Gamma$ induces an injective map of $C^*$-algebras $C^*\Lambda\to C^*\Gamma$.  

Moreover, the map $\C[\Gamma]\to \C[\Lambda]$ that sets all coefficients of elements in $\Gamma\setminus \Lambda$ to zero extends to a ucp map $C^*\Gamma\to C^*\Lambda$ that splits this inclusion. \qed
\end{lemma}

For later purposes, we also recall the notion of a crossed product $C^*$-algebra associated to an action of a group on a $C^*$-algebra.  

\begin{definition}\label{act cp}
Let $\Gamma$ be a discrete group.  A \emph{$\Gamma$-action} on a $C^*$-algebra $A$ is a homomorphism 
$$
\alpha \colon \Gamma\to \mathrm{Aut}(A)
$$
from $\Gamma$ to the group of $*$-automorphisms of $A$.  We also call the pair $(A,\alpha)$, or just $A$ if there is no risk of confusion, a \emph{$\Gamma$-$C^*$-algebra}.

The \emph{algebraic crossed product of $A$ by $\Gamma$}, denoted $A\rtimes_{alg,\alpha}\Gamma$ or $A\rtimes_{alg}\Gamma$ if there is no risk of confusion, is the $*$-algebra of formal sums 
$$
\sum_{g\in \Gamma} a_g g
$$
where each $a_g$ is an element of $A$ with only finitely many $a_g$ non-zero; the adjoint is determined by stipulating that it restricts to the given adjoint on $A$ and satisfies $g^*=g^{-1}$ for all $g\in \Gamma$; and the multiplication is determined by stipulating that it restricts to the given multiplications on $A$ and $\Gamma$ and that it satisfies
$$
gag^*=\alpha_g(a)
$$
for all $g\in \Gamma$ and $a\in A$.

A \emph{covariant pair} for a $\Gamma$-$C^*$-algebra is a pair $(\pi,u)$ where $\pi \colon A\to \mathcal{B}(H)$ is a nondegenerate representation of $A$ on a Hilbert space $H$, and $u \colon \Gamma\to U(H)$ is a unitary representation of $\Gamma$ on the same Hilbert space such that 
$$
u_g\pi(a)u_g^*=\pi(\alpha_g(a))
$$
for all $a\in A$ and $g\in \Gamma$.  The \emph{integrated form} of a covariant pair is the $*$-homomorphism 
$$
\pi\rtimes u \colon A\rtimes_{alg}\Gamma\to \mathcal{B}(H), \quad \sum a_gg\mapsto \sum \pi(a_g)u_g.
$$
Finally, the \emph{(maximal) crossed product} of $A$ by $\Gamma$, denoted $A\rtimes_{\alpha}\Gamma$ or just $A\rtimes \Gamma$ if there is no risk of confusion, is defined to be the completion of $A\rtimes_{alg}\Gamma$ for the norm
$$
\|b\|\coloneqq\sup\{\|(\pi\rtimes u)(b)\|_{\mathcal{B}(H)}\mid (\pi,u)\text{ a covariant pair}\}.
$$
\end{definition}

\begin{remark}\label{sd rem}
Note that $\C\rtimes \Gamma$ is canonically isomorphic to $C^*\Gamma$.  More generally, if $\Gamma=K\rtimes \Lambda$ is a semidirect product group, then the conjugation action $\alpha$ of $\Lambda$ on $K$ makes $C^*K$ a $\Lambda$-$C^*$-algebra, and we have $C^*\Gamma=C^*(K)\rtimes \Lambda$.  Indeed, this follows as unitary representations of $\Gamma$ are essentially the same thing as covariant pairs for $(C^*(K),\alpha)$.
\end{remark}

\subsection{Groups without the (L)LP}
\label{non ex}

We briefly mention the known examples of groups without the (L)LP - the rest of the paper will be about positive results, and we will use a particular example from this subsection to show necessity of some hypotheses.

The first existence proofs for groups without the LP were due to Ozawa \cite[Corollary 5]{Ozawa:2004aa}.  Later, Thom gave an explicit example of a group without the LLP \cite[page 198]{Thom:2010aa}.  Further examples without the LLP were given by Buss--Echterhoff--Willett \cite[Corollary 4.8]{Buss:2018nm}; these are again not explicit.

Fairly recently, Ioana--Spaas--Wiersma \cite{Ioana:2020aa} very significantly improved our understanding of groups without the (L)LP.  They used cohomological methods to show that many natural and interesting groups do not have the (L)LP; we single out one example.

\begin{example}\label{cex}
Consider the free group $F_2$ as a finite-index subgroup of $SL(2, \mathbb{Z})$. Then the semidirect product $\mathbb{Z}^2 \rtimes F_2$ does not have the LLP \cite[Corollary B]{Ioana:2020aa}.
\end{example}

We will not attempt to summarize all of the results of \cite{Ioana:2020aa} here, and just refer to the original paper: suffice to say that the techniques rely on variants of (relative) property (T).

\begin{remark}\label{llp t}
Failure of the (L)LP seems quite closely connected to property (T).  Indeed, the examples of Ozawa \cite[Corollary 5]{Ozawa:2004aa} and Thom \cite[page 198]{Thom:2010aa} mentioned above have property (T), and the examples of Ioana--Spaas--Wiersma \cite{Ioana:2020aa} all have (at least) property (T) with respect to an infinite subgroup.  

The examples of Buss--Echterhoff--Willett \cite[Corollary 4.8]{Buss:2018nm} use Osajda's probabilistic construction \cite[Theorem 4]{Osajda:2014ys} of groups whose Cayley graphs contain expanders.  It seems quite likely that these also have property (T) (compare \cite{Naor:2011aa}), but we do not know this.  On the other hand, from the methods used in \cite[Corollary 4.8]{Buss:2018nm} it seems plausible that the non-exact a-T-menable\footnote{\label{hp}a-T-menable groups are also  called groups with the \emph{Haagerup property} in operator algebra theory, due to the appearance of this property in Haagerup's seminal paper \cite{Haagerup:1979rq}.} groups contructed by Osajda in \cite[Theorem 2]{Osajda:2014ys} also do not have the (L)LP, but again we do not know this; note that a-T-menable groups cannot have property (T) relative to any infinite subgroup.

We do not know an example of a group with the LLP that has property (T), or even property (T) with respect to an infinite subgroup (Question \ref{q:T}). Note that \cite[Theorem 1.6]{Dogon:2023aa} says in particular that if there exists a property (T) group $\Gamma$ with the LLP\footnote{The statement assumes \emph{weak ucp stability}, but this is implied by the LLP: see Remark \ref{mat lift wucps} below).} such that $H_1(\Gamma) = 0 \neq H_2(\Gamma)$, then there exists a non-hyperlinear group.
\end{remark}

\subsection{Free and amenable groups}

Free and amenable groups are the fundamental examples of groups with the (L)LP.  We explain how to derive these results from the literature here.  

We need the following two fundamental theorems about $C^*$-algebras of discrete groups.  The first comes from \cite[Theorem 4.2]{Lance:1973aa} and the second from \cite{Hulanicki:1964aa}.  See for example \cite[Theorem 2.6.8]{Brown:2008qy} for a textbook exposition of both results.  For the statement of the first theorem, recall that the \emph{reduced group $C^*$-algebra} of a discrete group $\Gamma$, denoted $C^*_r\Gamma$, is the completion of $\C[\Gamma]$ for the norm it inherits from its image under the (left) regular representation $\lambda \colon \C[\Gamma]\to \mathcal{B}(\ell^2\Gamma)$.

\begin{theorem}[Lance]\label{lance the}
A discrete\footnote{The result can fail for non-discrete groups: for example, a result of Connes \cite[Corollary 6.9(c)]{Connes:1976fj} implies that the (reduced) $C^*$-algebra of any second countable connected group is nuclear.  On the other hand, Theorem \ref{hul the} holds for general locally compact groups.} group $\Gamma$ is amenable if and only if $C^*_r\Gamma$ is nuclear. \qed
\end{theorem}

\begin{theorem}[Hulanicki]\label{hul the}
A discrete group $\Gamma$ is amenable if and only if the canonical quotient map $C^*\Gamma\to C^*_r\Gamma$ is the identity.  \qed
\end{theorem}

\begin{corollary}\label{amen cor}
Any countable amenable group has the LP, and any amenable group has the LLP.
\end{corollary}

\begin{proof}
The $C^*$-algebra of a discrete group is separable if and only if the group is countable.  The statement on the LP therefore follows from Theorems \ref{cet}, \ref{lance the}, and \ref{hul the}.    The statement on the LLP follows as given an arbitrary amenable group $\Gamma$, any finite dimensional operator system $E\subseteq C^*\Gamma$ is contained in a subalgebra of the form $C^*(\Gamma_0)$ where $\Gamma_0\leq \Gamma$ is a countable (amenable) subgroup.
\end{proof}

The other fundamental class of examples of (L)LP groups is due to Kirchberg: the original reference is \cite[Lemma 3.3]{Kirchberg:1994aa}; see \cite[Theorem 13.1.3]{Brown:2008qy} for a textbook exposition.

\begin{theorem}[Kirchberg]\label{free gp}
Let $F$ be a free group.  Then $C^*F$ has the LLP, and it has the LP if $F$ is countable. \qed
\end{theorem}

\begin{remark}
The reader will note that we need a countability assumption for the LP in both Corollary \ref{amen cor} and Theorem \ref{free gp}; this will be a common theme.  In fact, we do not know any example of a (discrete) uncountable group with the LP (Question \ref{q:lp:uncountable}).  We also do not know if uncountable free groups can have the LP, although experts seem to believe that they probably do not: compare \cite[Remark 13.1.5]{Brown:2008qy}.
\end{remark}

\begin{remark}
As a consequence of Theorem \ref{free gp}, note that if $\Gamma$ is a countable group, and we fix a surjection $F\to \Gamma$ from a countable free group, then $\Gamma$ has the LP if and only if the induced quotient map $C^*F\to C^*\Gamma$ admits a ucp splitting, and similarly for the LLP.  We will not use this fact in this paper.
\end{remark}

\begin{remark}\label{red rem}
One could also ask about the (L)LP for $C^*_r\Gamma$, but this is less interesting than for $C^*\Gamma$ for two reasons.  

First, one is often interested in applications to finite-dimensional representations, and $C^*_r\Gamma$ admits (non-zero) finite dimensional representations if and only if $\Gamma$ is amenable: essentially the same proof as \cite[Theorem 2.6.8, part (7)$\Rightarrow$ (1)]{Brown:2008qy} shows this, having replaced the one-dimensional representation appearing there with the character (trace) of a finite-dimensional representation.

Second, if we make a fairly minor assumption on $\Gamma$ such as hyperlinearity, then the LLP for $C^*_r\Gamma$ implies amenability of $\Gamma$: see for example \cite[Remark 6.5.12]{Brown:2006aa}.   Conversely, amenability of $\Gamma$ implies the LLP for $C^*_r\Gamma$ by Corollary \ref{ce llp} and Theorem \ref{lance the}.  Hence the LLP for $C^*_r\Gamma$ is very close to amenability (and possibly even the same).
\end{remark}

\subsection{Subgroups}

In this subsection, we show that the (L)LP passes to subgroups.  This is known (see for example \cite[Remark 9.14]{Pisier:2020aa} or \cite[Remark 1.2]{Ioana:2020aa}), but we provide a proof: this is partly to keep the paper self-contained, and partly as we will need the key lemma for other purposes later.

Variants of the next result are well-known: compare for example \cite[Corollary 2.6 (v)]{Kirchberg:1993aa}, which is in many ways much stronger.  We give a (short) proof here as we could not find exactly what we need in the literature.

\begin{lemma}\label{cp ret}
Let $\iota \colon B\to A$ and $\sigma \colon A\to B$ be ccp maps between $C^*$-algebras such that $\sigma\circ \iota=\mathrm{id}_B$\footnote{One might say that ``$B$ is a ccp retract of $A$'', although this is non-standard.}.  Then if $A$ has the LLP (respectively, LP), $B$ does too.
\end{lemma}

\begin{proof}
For the LP, we first assume that $A$ and $B$ are unital and that $\iota \colon B\to A$ and $\sigma \colon B\to A$ are ucp.  Let then $C$ be a unital $C^*$-algebra, $\pi \colon C\to C/J$ be a quotient map, and $\phi \colon B\to C/J$ a ucp map. By Remark \ref{ucp vs ccp}, we want to lift $\phi$ to a ucp map $B\to C$.  Consider the diagram
$$
\xymatrix{ A \ar@{-->}[r] \ar[d]^-\sigma \ar[dr]^-{\phi\circ \sigma} & C \ar[d]^-\pi \\
B \ar[r]^-\phi \ar@/^/[u]^-\iota  & C/J  }
$$
As $A$ has the LP, the dashed arrow can be filled in by a ucp map, say $\psi$, so that the upper triangle commutes.  The lift we want is then $\psi\circ \iota$.  

For the general case, assume at least one of $A$, $B$, $\iota$ or $\sigma$ fails to be unital.  We then replace this data by the ucp maps $\iota^+ \colon B^+\to A^+$ and $\sigma^+ \colon A^+\to B^+$ as in Remark \ref{ucp vs ccp 0}.  Then $A^+$ still has the LP (either by definition if $A$ is nonunital, or because it is isomorphic to $A\oplus \C$ if $A$ is unital).  Hence we may use the unital case to deduce that $B^+$ has the LP, whence $B$ does (again, either by definition if $B$ is nonunital, or because it is isomorphic to $B\oplus \C$).

The LLP can be handled similarly.
\end{proof}

The next corollary follows immediately from Lemmas \ref{subgp rem} and \ref{cp ret}\footnote{Lemma \ref{cp ret} is more than we need to establish Corollary \ref{subgp}, but we will use the stronger version later.}.

\begin{corollary}\label{subgp}
The LP and LLP both pass to subgroups. \qed
\end{corollary}

\subsection{Increasing unions}

The following result is due to Kirchberg \cite[Corollary 2.6 (vi)]{Kirchberg:1993aa} (in a more general form).

\begin{lemma}\label{dir lim}
Let $A=\overline{\bigcup_{i\in I} A_i}$ be the closure of a directed union of $C^*$-subalgebras, each of which has the LLP.  Then $A$ has the LLP.

Let $A=\overline{\bigcup_{n=1}^\infty A_n}$ be the closure of a increasing sequence $C^*$-subalgebras, each of which is separable and has the LP.  Assume moreover that for each $n$, the inclusion $A_n\to A_{n+1}$ admits a ccp splitting.  Then $A$ has the LP. \qed
\end{lemma}

\begin{corollary}\label{dir un}
Let $\Gamma=(\bigcup_{i\in I}\Gamma_i)$ be the increasing union of a directed net of subgroups.  If all the $\Gamma_i$ have the LLP, then so does $\Gamma$.

If moreover $\Gamma$ is countable and all the $\Gamma_i$ have the LP, then $\Gamma$ has the LP.
\end{corollary}

\begin{proof}
The first statement is immediate from Lemmas \ref{subgp rem} and \ref{dir lim}. For the second one, since $\Gamma$ is countable, we can extract a cofinal subsequence of $(i_j)_{j \in \mathbb{N}} \subset I$, that is $\Gamma$ is the directed union of the $\Gamma_{i_j}$ and $\Gamma_{i_j} < \Gamma_{i_{j+1}}$. So again we conclude from Lemmas \ref{subgp rem} and \ref{dir lim}.
\end{proof}

\subsection{Semidirect products by amenable groups}

The following result was shown in \cite[Theorems 7.2 and 7.4]{Buss:2022aa} (in more generality).  A different proof was subsequently given in \cite[Section 8.2]{Enders:2024aa}.

\begin{theorem}[Buss--Echterhoff--Willett]\label{llp cp}
Let $B$ be a $\Lambda$-$C^*$-algebra.  If $B$ has the LLP and $\Lambda$ is amenable, then $B\rtimes\Lambda$ has the LLP.

If moreover $B$ is separable and has the LP, and $\Lambda$ is countable and amenable, then $B\rtimes\Lambda$ has the LP. \qed
\end{theorem}

The following corollary is immediate from Theorem \ref{llp cp} and Remark \ref{sd rem}.

\begin{corollary}\label{sd prod}
Let $\Gamma=K\rtimes \Lambda$ be a discrete group that splits as a semidirect product.  If $C^*K$ satisfies the LLP and $\Lambda$ is amenable, then $C^*\Gamma$ satisfies the LLP.

If $\Gamma$ is moreover countable, $C^*K$ has the LP, and $\Lambda$ is amenable, then $C^*\Gamma$ has the LP. \qed
\end{corollary}

\begin{remark}\label{llp products remark}
It is not true that the LLP (or LP) is preserved under semi-direct products.  The group $\Z^2\rtimes F_2$ from Example \ref{cex} does not have the LLP; on the other hand, $\Z^2$ and $F_2$ have the LP by Corollary \ref{amen cor} and Theorem \ref{free gp} respectively.

It seems to be open whether the (L)LP is preserved under direct products.  For example, it is open whether $F_2\times F_2$ has the (L)LP (Question \ref{q:ftimesf}); this problem is generally considered difficult, see for example the discussion at the end of \cite[Section 3]{Ozawa:2004ab}.
\end{remark}

\subsection{Free products amalgamated over finite subgroups}

Let $A$ and $B$ be unital $C^*$-algebras, each unitally containing a third $C^*$-algebra $C$.  Then there is a notion of \emph{(unital) free product $C^*$-algebra of $A$ and $B$, amalgamated over $C$}, denoted $A \Asterisk_C B$ (see for example \cite[II.8.3.5]{Blackadar:2006eq} or \cite[Section 5.1]{Loring:1997aa}).  This has the property that if $\Gamma$ and $\Lambda$ are discrete groups containing a common subgroup $\Delta$, then 
\begin{equation}\label{fpc*gp}
C^*\Gamma \Asterisk_{C^*\Delta} C^*\Lambda = C^*(\Gamma \Asterisk_\Delta \Lambda)
\end{equation}
(see for example \cite[Lemma 3.1]{Eilers:2018ab}).

\begin{theorem}[Boca, Pisier, Ozawa, Enders--Shulman]\label{fp the}
Let $A$ and $B$ be unital $C^*$-algebras, unitally containing a common finite-dimensional $C^*$-subalgebra $C$.  If $A$ and $B$ have the LLP, then so does $A \Asterisk_C B$.  If moreover $A$ and $B$ are separable and have the LP, then $A \Asterisk_C B$ has the LP.
\end{theorem}

\begin{proof}
For the LLP, the case where $C$ is just scalar multiples of the unit is due to Pisier \cite[Theorem 0.2]{Pisier:1996aa}; Ozawa \cite[discussion below Proposition 3.2.1]{Ozawa:2004ab} explains how to extend it to free products amalgamated over a finite-dimensional $C^*$-subalgebra.  For the LP, the result without amalgamation is a consequence of a result of Boca \cite[Theorem 3.1]{Boca:1991aa}, and is due to Enders--Shulman \cite[Corollary 4.3]{Enders:2024aa} in general (they also give a different proof in the LLP case, assuming $A$ and $B$ are separable).
\end{proof}

The following corollary is immediate from Theorem \ref{fp the} and the canonical isomorphism in line \eqref{fpc*gp}.

\begin{corollary}\label{fp gp}
Let $\Gamma$ and $\Lambda$ be groups, and $\Delta$ a common finite subgroup.  Then if $\Gamma$ and $\Lambda$ have the LLP, so does $\Gamma \Asterisk_\Delta \Lambda$.  If moreover $\Gamma$ and $\Lambda$ are countable and have the LP, then $\Gamma \Asterisk_\Delta \Lambda$ has the LP. \qed
\end{corollary}

\begin{remark}
Corollary \ref{fp gp} is false for an amalgamated product over an infinite subgroup, even an abelian one.  The group $\mathbb{Z}^2 \rtimes F_2$ from Example \ref{cex} does not have the LLP.  Write now $\alpha$, $\beta$ for the action of the two generators of $F_2$ on $\Z^2$, and let $\Gamma=\Z^2\rtimes_\alpha\Z$ and $\Lambda=\Z^2\rtimes_\beta \Z$.  Let $\Delta=\Z^2$, considered as a subgroup in both $\Gamma$ and $\Lambda$ in the canonical way.  Then $\Gamma$ and $\Lambda$ have the LP by Corollary \ref{amen cor}, but $\Gamma \Asterisk_\Delta \Lambda=\Z^2\rtimes F_2$ does not, as we have already noted.  On the other hand, free products amalgamated over an abelian \emph{retract} do preserve the (L)LP: see Corollary \ref{fp r llp} below.

Whether or not the (L)LP is preserved by free products amalgamated over $\Z$ seems to be open.
\end{remark}

\subsection{Graphs of groups with finite edge groups}
\label{ss:gog}

Corollary \ref{fp gp} extends to a larger class of fundamental groups of graphs of groups. This was essentially remarked already in \cite[Corollaries 4.6 and 4.7]{Enders:2024aa}, although the hypotheses are more restrictive. Moreover \cite[Corollaries 4.6 and 4.7]{Enders:2024aa} refer back to \cite{Eilers:2018ab}, whose approach also uses a notion of HNN extension of $C^*$-algebras. To avoid repeating a very similar argument, we give an alternative proof here that only uses the result for amalgamated products, and basic Bass--Serre theory. This gives us an opportunity to introduce fundamentals of Bass--Serre theory, which will be used again in the rest of the paper.

\begin{proposition}\label{gog finite edge}

Let $\Gamma$ be the fundamental group of a graph of groups with finite edge groups. If all vertex groups have the LLP, then so does $\Gamma$. If moreover all vertex groups have the LP and $\Gamma$ is countable, then $\Gamma$ has the LP.
\end{proposition}

\begin{remark}
    Fundamental groups of graphs of groups with finite edge groups play an important role in geometric group theory. If $\Gamma$ is a finitely generated group, then there is a well-defined notion of \emph{number of ends} $e(\Gamma)$, which is the (graph-theoretic) number of ends in some (equivalently, every) Cayley graph of $\Gamma$. The following basic facts were proved in the 1940s by Freudenthal \cite{freudenthal:ends} and Hopf \cite{hopf:ends}: $e(\Gamma) \in \{0, 1, 2, \infty\}$, with $e(\Gamma) = 0$ if and only if $\Gamma$ is finite, and $e(\Gamma) = 2$ if and only if $\Gamma$ is virtually $\mathbb{Z}$. The case $e(\Gamma) = 1$ is in a sense the generic one, in which case we say $\Gamma$ is \emph{one-ended}.

    When $e(\Gamma) = \infty$, Stallings proved that $\Gamma$ is either an amalgamated product or an HNN extension over a finite subgroup \cite{stallings:tf, stallings:general}. One can iterate this process, and if it terminates, $\Gamma$ will be expressed as the fundamental group of a graph of groups with finite edge groups and vertex groups with at most one end: in this case we say that $\Gamma$ is \emph{accessible}. Not all finitely generated groups are accessible \cite{dunwoody:inaccessible}, but a fundamental theorem of Dunwoody says that the finitely presented ones are \cite{dunwoody:accessible}. Moreover, for a finitely presented group, the corresponding vertex groups are also finitely presented \cite{fp:vertexgroups}. Therefore Proposition \ref{gog finite edge} reduces the (L)LP for finitely presented groups to the one-ended case.
\end{remark}

Let us start by recalling some basic notions from Bass--Serre theory: see \cite{Serre:1980aa} or \cite[Chapter 2]{bogopolski} for more details. A \emph{graph of groups} $\mathcal{G}$ is the data of a simplicial graph $X = (V, E)$ with a collection of vertex groups $(\Gamma_v)_{v \in V}$, edge groups $(\Gamma_e)_{e \in E}$ and inclusions $\alpha_e, \omega_e \colon \Gamma_e \to \Gamma_{\alpha(e)},\Gamma_{\omega(e)}$ from an edge group into its initial and terminal vertex, respectively. We follow Serre's convention that every simplicial edge is a pair $\{e, \bar{e}\}$, once with each orientation, and that $\Gamma_e = \Gamma_{\bar{e}}$, with $\alpha_e = \omega_{\bar{e}}, \omega_e = \alpha_{\bar{e}}$. Graphs are allowed to be infinite, and to have loops and multiple edges.

The \emph{fundamental group} $\pi_1(\mathcal{G})$ is defined as follows. First, choose a spanning tree $T$ of $X$, and write $E_T$ and $E_X$ for their respective edge sets; note that their vertex sets coincide. For every edge $e$ outside of $T$, let $t_e$ be a fresh letter. Then
\begin{align*}
\pi_1(\mathcal{G})  \coloneqq  \langle &\Gamma_v : v \in V; t_e : e \in E_X \setminus E_T \\
&\mid t_{\bar{e}} = t_e^{-1} : e \in E_X \setminus E_T; \\
& \alpha_e(g) = \omega_e(g) : e \in E_T, g \in \Gamma_e; \\
& t_e^{-1} \alpha_e(g) t_e = \omega_e(g): e \in E_X \setminus E_T, g \in \Gamma_e \rangle.
\end{align*}
The isomorphism class of the fundamental group does not depend on the choice of the spanning tree. The elements $t_e \in \pi_1(\mathcal{G})$ are called \emph{stable letters}.

\begin{remark}
\label{gog basic operations}

Suppose that $X$ has $n$ vertices. Then $\pi_1(\mathcal{G})$ can be built by a sequence of amalgamated products and HNN extensions as follows. Start with a single vertex $v_1$, let $T_1 = \{ v_1 \}$ and let $\Gamma_1 = \Gamma_{v_1}$. Suppose by induction that a tree $T_k$ with vertices $\{v_1, \ldots, v_k\}$ has been constructed, with corresponding subgroup $\Gamma_k = \langle \Gamma_{v_1}, \ldots, \Gamma_{v_k} \rangle < \pi_1(\mathcal{G})$. Let $v_{k+1}$ be a new vertex connected to $T_k$ by an edge $e$ with other endpoint $v_j$, and let $T_{k+1}$ be the tree obtained by adding $e$ (hence $v_{k+1}$). Then we define
\[\Gamma_{k+1} = \Gamma_k \Asterisk_{\Gamma_e} \Gamma_{v_{k+1}},\]
where the inclusions are give by $\omega_e$ on the right, and $\alpha_e$, followed by the inclusion $\Gamma_{v_j} < \Gamma_k$, on the left. $T_n$ will then be a spanning tree and $\Gamma_n$ will be the subgroup of $\pi_1(\mathcal{G})$ generated by the vertex groups, satisfying the relations that do not involve the letters $t_e$.

To include these, now enumerate the remaining edges $e_1, \ldots, e_m$. We define
\[\Gamma_n^1 = \Gamma_n \Asterisk_{f} = \langle \Gamma_n,  t_{e_1} \mid t_{e_1}^{-1} g t_{e_1} = f(g) : g \in \alpha_{e_1}(\Gamma_{e_1}) \rangle,\]
where $f$ is the isomorphism between $\alpha_{e_1}(\Gamma_{e_1})$ and $\omega_{e_1}(\Gamma_{e_1})$, seen as subgroups of the corresponding vertex groups, included in $\Gamma_n$. We continue this way, with each new edge leading to an HNN extension $\Gamma_n^{k+1}$ of the previous group $\Gamma_n^k$ along the edge group $\Gamma_{e_{k+1}}$. The final group $\Gamma_n^m$ will be $\pi_1(\mathcal{G})$.
\end{remark}

The fundamental theorem of Bass--Serre theory is that $\pi_1(\mathcal{G})$ acts simplicially on a tree without edge inversions, with quotient graph $X$, and vertex and edge stabilizers conjugate to the corresponding vertex or edge groups of $\mathcal{G}$; conversely, if $\Gamma$ acts simplicially on a tree without edge inversions, then setting $X$ to be the quotient graph, and choosing the vertex and edge groups to be the vertex and edge stabilizers, and the natural inclusions of edge stabilizers into vertex stabilizers, defines a graph of groups $\mathcal{G}$ with $\Gamma = \pi_1(\mathcal{G})$ \cite[Section 5.4, Theorem 13]{Serre:1980aa}.

\begin{proof}[Proof of Proposition \ref{gog finite edge}]

We keep the notation from the discussion above, where $\mathcal{G}$ is a graph of groups with underlying graph $X$, (L)LP vertex groups $\Gamma_v$, finite edge groups $\Gamma_e$ and fundamental group $\Gamma$.

\emph{When $X$ is a finite tree}, by Remark \ref{gog basic operations}, $\Gamma$ is obtained by a finite sequence of amalgamated products with finite subgroups, so this follows by an inductive argument with Corollary \ref{fp gp}.

\emph{When $X$ is a tree} then $\Gamma$ is a directed union of groups as in the previous case, so this follows from Corollary \ref{dir un}.

\emph{When $X$ is a single loop} $\Gamma$ is an HNN extension over a finite group, more precisely
\[\Gamma = \langle \Gamma_v, t \mid t^{-1} g t = f(g) : g \in \Gamma_e \rangle,\]
where $\Gamma_e$ is a finite subgroup of $\Gamma_v$ and $f \colon \Gamma_e \to \Gamma_v$ is an embedding. Consider the group
\[\Lambda = \langle t^{-n} \Gamma_v t^n : n \in \mathbb{Z} \rangle < \Gamma.\]
Then $\Gamma = \Lambda \rtimes \langle t \rangle$, and $\Lambda$ is the fundamental group of a graph of groups whose underlying graph is a bi-infinite line, with vertex and edge groups isomorphic to $\Gamma_v$ and $\Gamma_e$ respectively \cite[Theorem 2.17.1]{bogopolski}. Therefore, by the previous case, $\Lambda$ has the (L)LP, and then by Corollary \ref{sd prod} we conclude that $\Gamma$ has the (L)LP.

\emph{When $X$ is finite}, Remark \ref{gog basic operations} shows that it can be obtained by a sequence of HNN extensions starting from the fundamental group of the subgraph defined on a spanning tree of $X$, so this follows by an inductive argument with the first and third cases above.

\emph{In general}, we conclude by Corollary \ref{dir un}.
\end{proof}

\begin{corollary}
\label{cd1}

If $\Gamma$ is a group of rational cohomological dimension $1$, then $\Gamma$ has the LLP. If $\Gamma$ is moreover countable, then it has the LP. In particular, virtually free groups have the LLP, and countable virtually free groups have the LP.
\end{corollary}

We will recover the result on virtually free groups later on via Cororollary \ref{fi cor}. Recall that if $\mathcal{P}$ is a group property, we say that $\Gamma$ is \emph{virtually $\mathcal{P}$} if it admits a finite-index subgroup with $\mathcal{P}$.

\begin{proof}
Dunwoody \cite{dunwoody} characterizes groups of rational cohomological dimension $1$ as fundamental groups of graphs of groups with finite vertex and edge groups, so we conclude by Proposition \ref{gog finite edge} and Corollary \ref{amen cor}. Virtually free groups are a special case, for instance thanks to Serre's Theorem \cite[Th{\'e}or{\`e}me 1.7.1]{Serre:cohomologie} (or rather its generalization to other rings \cite[Theorem II.5.11]{bieri}).
\end{proof}

In the finitely generated case, it is easier to see that virtually free groups are exactly the fundamental groups of finite graphs of groups with finite edge groups \cite{karrass:pietrowski:solitar}.

\subsection{More graphs of groups}

Let us present two more cases in which graphs of groups behave well with respect to the (L)LP.

Recall that if $f \colon \Gamma\to \Gamma$ is an injective homomorphism, then the \emph{(ascending) HNN extension} is defined by 
\[ \Gamma \Asterisk_f  \coloneqq  \langle \Gamma, t \mid t^{-1}gt = f(g) \text{ for all } g \in \Gamma \rangle.\]

\begin{corollary}
\label{ascending HNN}
Let $f \colon \Gamma \to \Gamma$ be an injective homomorphism, and let $\Gamma \Asterisk_f$ be the corresponding (ascending) HNN extension. If $\Gamma$ has the LLP, then so does $\Gamma \Asterisk_f$. Moreover, if $\Gamma$ is countable and has the LP, then so does $\Gamma \Asterisk_f$.
\end{corollary}

\begin{proof}
There is a natural semidirect product structure $\Gamma \Asterisk_f \cong \Lambda \rtimes \langle t \rangle$, where $\Lambda$ is the normal closure of $\Gamma$, namely the union of the following sequence:
\[\Gamma < t \Gamma t^{-1} < t^2 \Gamma t^{-2} < \cdots < \Lambda = \bigcup_{n \in \mathbb{N}} t^n \Gamma t^{-n}.\]
We conclude by Corollaries \ref{dir un} and \ref{sd prod}.
\end{proof}

Next, recall that a subgroup $\Lambda \to \Gamma$ is a \emph{retract} if there exists a homomorphism $\Gamma \to \Lambda$ that restricts to the identity on $\Lambda$. Equivalently, there exists a normal subgroup $N < \Gamma$ such that $\Gamma = N \rtimes \Lambda$.

\begin{corollary}\label{fp r llp}
    Let $\Gamma_1, \Gamma_2$ be groups with the LLP, with a common amenable retract $\Lambda$. Then $\Gamma_1 \Asterisk_{\Lambda} \Gamma_2$ has the LLP. If moreover $\Gamma_1, \Gamma_2$ are countable and have the LP, then $\Gamma_1 \Asterisk_{\Lambda} \Gamma_2$ has the LP.
\end{corollary}

\begin{proof}
    Write $\Gamma_i = N_i \rtimes \Lambda$. By comparing presentations, we see that
    \[\Gamma_1 \Asterisk_{\Lambda} \Gamma_2 \cong (\Gamma_1 \Asterisk \Gamma_2) \rtimes \Lambda.\]
    Because $\Gamma_i$ has the LLP, so does $N_i$ by Corollary \ref{subgp}. So by Corollary \ref{fp gp}, the free product $N_1 \Asterisk N_2$ has the LLP. Finally, by Corollary \ref{sd prod}, the semidirect product $(N_1 \Asterisk N_2) \rtimes \Lambda$ has the LLP. The proof for the LP is similar.
\end{proof}

\begin{remark}
    We do not know if the amenability hypothesis can be dropped. If it could, this would cover $F_2 \times F_2$, which is the free product of two copies of $F_2 \times \mathbb{Z}$ amalgamated along a common $F_2$-retract (Question \ref{q:ftimesf}).
\end{remark}

\subsection{Central extensions}\label{subs:central}

The results we can give on central extensions are partial, but at least cover some interesting examples and explain the basic structure.

We will need to recall the notion of a group $C^*$-algebra twisted by an $S^1$-valued $2$-cocycle (see for example \cite[Section 14]{Echterhoff:2009jo} for a useful survey of this), and of a $C(X)$-algebra.

\begin{definition}
Let $\tau$ be a normalized $S^1$-valued $2$-cocycle on a group $\Lambda$ representing a class in $H^2(\Lambda;S^1)$: precisely, $\tau$ is a map $\Lambda\times \Lambda\to S^1$ satisfying 
\begin{equation}\label{cocyc}
\tau(1,g)=\tau(g,1)=1\quad \text{and}\quad \tau(g,h)\tau(gh,k)=\tau(g,hk)\tau(h,k)
\end{equation}
for all $g,h,k\in \Lambda$.  Let $\C[\Lambda;\tau]$ have the same underlying vector space as the group algebra $\C[\Lambda]$, but with multiplication and adjoint determined by\footnote{Associativity and the relation $(u_gu_h)^*=u_h^*u_g^*$ are consequences of the formulas in line \eqref{cocyc}.} 
$$
u_gu_h \coloneqq \tau(g,h)u_{gh}\quad \text{and}\quad u_g^* \coloneqq \overline{\tau(g,g^{-1})}u_{g^{-1}}.
$$ 
We then define $C^*(\Lambda;\tau)$ to be the completion of $\C[\Lambda;\tau]$ for the norm defined for $a\in \C[\Lambda;\tau]$ by
$$
\|a\|\coloneqq\sup\{\|\pi(a)\|_{\mathcal{B}(H)}\mid \pi \colon \C[\Lambda;\tau]\to \mathcal{B}(H) \text{ a $*$-representation}\}.
$$ 
\end{definition}
We note that $C^*(\Lambda;\tau)$ only depends on the class of $\tau$ in $H^2(\Lambda;S^1)$ up to isomorphism. In particular, if $\tau$ represents the trivial class, then $C^*(\Lambda; \tau) \cong C^*\Lambda$.

\begin{definition}
Let $X$ be a compact Hausdorff space.  A unital\footnote{There is also an analgue for non-unital $C^*$-algebras, but it needs the notion of a multiplier algebra from Definition \ref{ma} below.} \emph{$C(X)$-algebra} is a unital $C^*$-algebra $A$ equipped with a unital $*$-homomorphism $C(X)\to Z(A)$ from $X$ to the center of $A$; we typically abuse notation and identify $C(X)$ with its image in $Z(A)$ (even if the map is not injective).  
\end{definition}

If $A$ is a $C(X)$-algebra, for each $x\in X$, let $I_x$ be the closed linear span of $C_0(X\setminus \{x\})\cdot A$\footnote{For a locally compact Hausdorff space $Y$, $C_0(Y)$ means the collection of continuous functions $f \colon Y\to \C$ such that for all $\epsilon>0$ there is compact $K\subseteq Y$ such that $|f(y)|<\epsilon$ for all $y\not\in K$; in this case, $C_0(X\setminus \{x\})$ identifies with the continuous functions on $X$ that vanish at $x$.}, which is an ideal in $A$, and define $A_x \coloneqq A/I_x$ to be the corresponding quotient. The natural map 
\begin{equation}\label{cf incl}
A\to \prod_{x\in X} A_x
\end{equation}
defined by taking the product of the quotient maps can then be shown to be injective: see for example \cite[Theorem 7.47]{Douglas:1972vn} or \cite[Proposition 1.2]{Rieffel:1989aa}.  One usually thinks of $A$ as an `algebra of sections' of the `field' $\{A_x\}_{x\in X}$ of `fibre' $C^*$-algebras $A_x$.

We are now ready to discuss central extensions.  Let $1\to K\to \Gamma\to \Lambda\to 1$ be a central extension of a group $\Lambda$ by an abelian group $K$.  Choose a set-theoretic splitting $s \colon \Lambda\to \Gamma$ that preserves the identity, and define 
\begin{equation}\label{k l co}
\sigma \colon \Lambda\times \Lambda \to K,\quad \sigma(g,h) \coloneq s(g)s(h)s(gh)^{-1},
\end{equation}
so $\sigma$ is a normalized $2$-cocycle representing the class in $H^2(\Lambda;K)$ associated to the extension in the usual way (compare for example \cite[Section IV.3]{Brow:1982rt}).  For each character $\omega \colon K\to S^1$, define 
\begin{equation}\label{s l co}
\sigma_\omega \coloneqq \omega\circ \sigma,
\end{equation} 
a $2$-cocycle representing a class in $H^2(\Lambda;S^1)$.

The next result is a special case of \cite[Lemma 6.3]{Echterhoff:1998aa}, which contains both a more general result\footnote{Almost more general: the authors of \cite{Echterhoff:1998aa} work with locally compact groups, and make a second countability assumption to avoid complications with Borel structures.  We do not need a (second) countability assumption as we only work with discrete groups and there are no Borel complications.}, and a statement giving more detailed information.  The authors of \cite{Echterhoff:1998aa} attribute the result to Packer--Raeburn \cite{Packer:1990aa,Packer:1992aa}; the papers of Packer--Raeburn contain much more general results. 

For the statement, for an abelian group $K$, let $\widehat{K}$ denote the Pontrjagin dual, i.e.\ the collection of all homomorphisms $K\to S^1$, equipped with the (compact, Hausdorff) topology of pointwise convergence.  

\begin{proposition}[Packer--Raeburn]\label{c ext struct}
Let $1\to K\to \Gamma\to \Lambda\to 1$ be a central extension and let $\sigma\in Z^2(\Lambda;K)$ be an associated $2$-cocycle as in line \eqref{k l co}.  Then the inclusion $C^*K\subseteq C^*\Gamma$ and the Fourier isomorphism $C^*K\cong C(\widehat{K})$ make $C^*\Gamma$ into a $C(\widehat{K})$-algebra.

Moreover, for each $\omega\in \widehat{K}$, the fibre $(C^*\Gamma)_\omega$ is canonically isomorphic to the twisted group $C^*$-algebra $C^*(\Lambda;\sigma_\omega)$, where $\sigma_\omega$ is as in line \eqref{s l co}. \qed
\end{proposition}

\begin{corollary}\label{c ext}
With notation as in Proposition \ref{c ext struct}, if all the twisted group $C^*$-algebras $C^*(\Lambda;\sigma_\omega)$ have the LLP for $\omega\in \widehat{K}$, then so does $C^*\Gamma$.

Moreover, if $\Gamma$ is countable, and all the $C^*$-algebras $C^*(\Lambda;\sigma_\omega)$ have the LP, then so does $C^*\Gamma$.
\end{corollary}

The hypothesis ``$C^*(\Lambda;\sigma_\omega)$ have the (L)LP for $\omega\in \widehat{K}$'' seems difficult to check in general, unless it reduces to a hypothesis on $C^*\Lambda$ as in Example \ref{no twist} below.  Conceivably, it can always be deduced from the (L)LP for $C^*\Lambda$, but we were unable to do that (Question \ref{q:cocycles}).

\begin{proof}[Proof of Corollary \ref{c ext}]
We first look at the LLP.  Write $B=\mathcal{B}(\ell^2(\N))$.  Using Theorem \ref{kir the}, it suffices to show that $(C^*\Gamma,B)$ is a nuclear pair.  Consider the commutative diagram 
$$
\xymatrix{ C^*\Gamma\otimes_{\max} B \ar[r] \ar[d] & \prod_{\omega\in \widehat{K}} (C^*(\Lambda;\sigma_\omega)\otimes_{\max} B ) \ar[d] \\
C^*\Gamma\otimes B \ar[r]  & \prod_{\omega\in \widehat{K}} (C^*(\Lambda;\sigma_\omega)\otimes B ) }
$$
where the horizontal maps are induced by the embedding in line \eqref{cf incl} associated to the $C(\widehat{K})$-algebra structure from Proposition \ref{c ext struct}, and the vertical maps are induced by the canonical quotient maps from the maximal to minimal crossed products.  As each $C^*(\Lambda;\sigma_\omega)$ has the LLP, Theorem \ref{kir the} implies that the right hand vertical map is the identity.  On the other hand, using the discussion in point 2.\ from \cite[page 678]{Kirchberg:1995aa}, the top\footnote{The bottom one is too, but we do not need this.} horizontal map is injective.  Hence the left vertical map is injective, and we are done.

We now look at the LP.  Let $\{D_i\}_{i\in I}$ be an arbitrary family of $C^*$-algebras, and let $D \coloneqq \prod_{i\in I}D_i$ be their product.  Using Theorem \ref{pis the}, it suffices to show that the canonical map 
$$
D\otimes_{\max}C^*\Gamma\to \prod_{i\in I}(D_i\otimes_{\max}C^*\Gamma)
$$
is injective.  Consider the commutative diagram
$$
\xymatrix{D\otimes_{\max}C^*\Gamma \ar[r] \ar[d] & \prod_{i\in I}(D_i\otimes_{\max} C^*\Gamma) \ar[d] \\
\prod_{\omega\in \widehat{K}}(D\otimes_{\max}C^*(\Lambda;\sigma_\omega)) \ar[r] & \prod_{\omega\in \widehat{K}}\Big(\prod_{i\in I}(D_i\otimes_{\max}C^*(\Lambda;\sigma_\omega))\Big)}
$$
where the two vertical maps are induced by the $C(\widehat{K})$-structure, and also a reordering of the product on the right.  The left vertical map is injective by the discussion in point 2.\ from \cite[page 678]{Kirchberg:1995aa} again, while the bottom horizontal map is injective by the LP for each $C^*(\Lambda;\sigma_\omega)$ and Theorem \ref{pis the}.  Hence the top horizontal map is injective, and we are done.
\end{proof}

\begin{example}\label{no twist}
Let $1\to K\to \Gamma\to \Lambda\to 1$ be a central extension of a group $\Lambda$ with the (L)LP, and let $\sigma$  be a $2$-cocycle as in line \eqref{k l co}.  Assume that the extension is \emph{pointwise trivial} in the sense of \cite[Section 3]{Echterhoff:2002aa}: this means that for each $\omega\in \widehat{K}$, if $\sigma_\omega$ is as in line \eqref{s l co}, then $[\sigma_\omega]=0$ in $H^2(\Lambda,S^1)$.  Then Corollary \ref{c ext} implies that any (countable) central extension of $\Lambda$ has the LLP (LP).

This holds in particular if $H^2(\Lambda; S^1) = 0$. To slightly demystify this condition, note that by the universal coefficient theorem \cite[Section I.0]{Brow:1982rt}, we have an exact sequence
\[0 \to Ext^1_{\mathbb{Z}}(H_1(\Lambda); S^1) \to H^2(\Lambda; S^1) \to Hom_{\mathbb{Z}}(H_2(\Lambda); S^1) \to 0.\]
Since $S^1$ is a divisible abelian group, that is an injective $\mathbb{Z}$-module, the $Ext$ term vanishes. Moreover, since non-trivial abelian groups have non-trivial characters, we conclude that
\[H^2(\Lambda; S^1) = 0 \quad \Leftrightarrow \quad H_2(\Lambda) = 0.\]

This assumption is restrictive, but still covers some interesting cases.  Indeed, $H_2(\Lambda) = 0$ for a finite cyclic group \cite[Section II.3]{Brow:1982rt}, hence by the Mayer--Vietoris sequence \cite[Corollary II.7.7]{Brow:1982rt}, the same is true for free products of finite cyclic groups.  Thus Corollary \ref{fp gp} implies that any (countable) central extension of such a group has the LLP (LP). This covers for example the \emph{torus knot group} $\langle a,b \mid a^p=b^q\rangle$, which is a central extension of $(\Z/p) \Asterisk (\Z/q)$ by $\Z$ (compare for example \cite[Example 7.3.4]{Cherix:2001fk}).
\end{example}

The paper \cite{Echterhoff:2002aa}, particularly Section 3, contains a much more detailed discussion of the pointwise trivial central extensions arising in Example \ref{no twist}.

\begin{remark}
    We will see in Corollary \ref{fi cor} that virtually LLP groups are LLP. It follows that central extensions of all virtually free groups are LLP.  
    Indeed, suppose that $\Lambda$ is virtually free and $F$ is a free subgroup of finite index. Let
    \[1 \to K \to \Gamma \to \Lambda \to 1\]
    be a central extension.  The preimage of $F$ in $\Gamma$ is a central extension of $F$ by $K$, hence isomorphic to $F \times K$, which is LLP by combining Theorems \ref{free gp} and \ref{sd prod}. Hence $\Gamma$ is virtually LLP.  The discussion carries over to the LP if we assume all relevant groups are countable. 
    
     This  gives another approach to the central extensions of free products of finite cyclic groups discussed in Example \ref{no twist}.  The particular case of torus knot groups will be covered again in Examples \ref{or center} and \ref{GBS}.
\end{remark}

\subsection{Amenable extensions}

Our final goal in this section is to improve Corollary \ref{sd prod} to the case of a non-split exact sequence: this is our most important new permanence result on the (L)LP for groups.

\begin{theorem}\label{llp the}
Let $\Gamma$ be a discrete group that fits into a short exact sequence $1\to K\to \Gamma\to \Lambda \to 1$.   If $C^*K$ satisfies the LLP and $\Lambda$ is amenable, then $C^*\Gamma$ satisfies the LLP.  

If $\Gamma$ is moreover countable, $C^*K$ has the LP, and $\Lambda$ is amenable, then $C^*\Gamma$ has the LP.
\end{theorem}

\begin{remark}
There is something of a disanalogy between extensions of groups and of $C^*$-algebras.  Indeed, the LLP for groups is not in general preserved by extensions: see Remark \ref{llp products remark} above.  On the other hand, if $0\to I\to A\to Q\to 0$ is an extension of $C^*$-algebras such that $I$ and $Q$ have the LLP, then $A$ also has the LLP\footnote{Thanks to Kristin Courtney for pointing this out.}.  Indeed, let $\mathcal{B}=\mathcal{B}(\ell^2(\N))$ and consider the commutative diagram 
$$
\xymatrix{ 0 \ar[r] & I\otimes_{\max}\mathcal{B} \ar[r] \ar[d] & A\otimes_{\max}\mathcal{B} \ar[r] \ar[d] & Q\otimes_{\max}\mathcal{B} \ar[r] \ar[d] & 0 \\
0 \ar[r] & I\otimes\mathcal{B} \ar[r] & A\otimes\mathcal{B} \ar[r] & Q\otimes\mathcal{B} \ar[r] & 0}
$$
where the vertical maps are the canonical quotients.  The top row is exact by general properties of the maximal tensor product (see for example \cite[Proposition 3.7.1]{Brown:2008qy}), and the bottom row is exact by Theorem \ref{kir the}, the LLP for $Q$ and \cite[Corollary 3.7.3]{Brown:2008qy} (or the easier half of \cite[Theorem 3.2]{Effros:1985aa}).  Hence using Theorem \ref{kir the} and the five lemma, $A$ has the LLP.

Unfortunately, the fact that the LLP behaves well under $C^*$-algebra extensions does not seem to help at all with group extensions. 
\end{remark}

The proof of Theorem \ref{llp the} will be based on the theory of Morita equivalence of $C^*$-algebras, which we will use largely as a black box.   The idea of Morita equivalence is due to Rieffel \cite{Rieffel:1982aa}\footnote{Rieffel originally called it \emph{strong Morita equivalence}, but that terminology is rarely used now.}: one treats the collection of all $C^*$-algebras as a category, where the morphism sets are not functions, but rather appropriate bimodules, and composition is defined by tensor product.  Two $C^*$-algebras are Morita equivalent if they are isomorphic in this category.  See for example \cite[Section 5]{Echterhoff:2009jo} for a good overview.

The following structure theorem is the key ingredient we need: one can deduce it as a simple special case of the work of Packer--Raeburn \cite{Packer:1990aa}, or of Echterhoff \cite{Echterhoff:1994aa}.  We will treat it as a black box for now, and explain how to deduce it from the literature later on.

\begin{theorem}[Packer--Raeburn, Echterhoff]\label{struct}
Let $1\to K\to \Gamma\to \Lambda \to 1$ be an extension of discrete groups.  Then there is a $\Lambda$-$C^*$-algebra $B$ such that $B$ is Morita equivalent to $C^*(K)$, and $B\rtimes \Lambda$ is Morita equivalent to $C^*(\Gamma)$.  
If moreover $\Gamma$ is countable, we may assume that $B$ is separable. \qed
\end{theorem}

Given Theorems \ref{struct} and \ref{llp cp}, to prove Theorem \ref{llp the}, it suffices to show that the LLP is preserved by Morita equivalences of $C^*$-algebras, and that the LP is preserved by Morita equivalences of separable $C^*$-algebras.  These facts are no doubt known to experts, but we could not find them in the literature so give proofs for the reader's convenience.

\begin{proposition}\label{kir cor}
If $A$ and $B$ are Morita equivalent $C^*$-algebras and $A$ has the LLP, then $B$ does too.
\end{proposition}

\begin{proof}
Exactly the same proof as \cite[Proposition 5.6]{Echterhoff:2009jo} shows that for any $C^*$-algebras $A$, $B$, and $C$, if $A$ and $B$ are Morita equivalent and $(A,C)$ is a nuclear pair, then $(B,C)$ is a nuclear pair also.  Apply this with $C=\mathcal{B}(H)$ and use Theorem \ref{kir the}.
\end{proof}

We now turn to the LP.  We split the proof up into two lemmas.  The first is due to Kirchberg \cite[Corollary 2.6 (iv)]{Kirchberg:1993aa}.  

\begin{lemma}[Kirchberg]\label{to tens}
Let $A$ be a separable $C^*$-algebra with the LP, and let $B$ be a separable nuclear $C^*$-algebra.  Then $A\otimes B$ has the LP. \qed
\end{lemma}

\begin{lemma}\label{from tens}
Let $A$ and $B$ be (non-zero) $C^*$-algebras, and assume that $A\otimes B$ has the LP.  Then $A$ has the LP.
\end{lemma}

\begin{proof}
Let $b\in B$ be a positive norm one contraction, and choose any state $\phi$ on $B$ such that $\phi(b)=1$.  Then the maps determined by
$$
\iota \colon A\to A\otimes B,\quad a\mapsto a\otimes b
$$
and 
$$
\sigma \colon A\otimes B\to A,\quad a\otimes b\mapsto \phi(b) a 
$$
satisfy the hypotheses of Lemma \ref{cp ret}.
\end{proof}

\begin{corollary}\label{lp mor}
If $A$ and $B$ are Morita equivalent separable $C^*$-algebras and $A$ has the LP, then $B$ does too.
\end{corollary}

\begin{proof}
If $A$ and $B$ are separable $C^*$-algebras, then the Brown--Green--Rieffel theorem \cite{Lawrence-G.-Brown:1977co} says that  $A$ is Morita equivalent to $B$ if and only if $A\otimes \mathcal{K}$ is isomorphic to $B\otimes \mathcal{K}$, where $\mathcal{K}$ is the $C^*$-algebra of compact operators on a separable infinite-dimensional Hilbert space.  It thus suffices to show that a separable $C^*$-algebra $A$ has the LP if and only if $A\otimes \mathcal{K}$ does.  This follows directly from Lemmas \ref{to tens} and \ref{from tens}, plus the fact that $\mathcal{K}$ is nuclear (see for example \cite[Proposition 2.4.1]{Brown:2008qy}).
\end{proof}

For the rest of the section, we will prove Theorem \ref{struct}, based on Echterhoff's work \cite{Echterhoff:1994aa}.

\subsection*{Twisted crossed products}

To state the definition of a twisted action, we need the notion of a multiplier algebra.  This would not be necessary if we only worked with unital $C^*$-algebras, but the theory of Morita equivalence typically takes one outside of the unital case.  The point of the multiplier algebra is that it gives one a setting to talk about things like inner automorphisms that do not directly make sense for non-unital algebras.

\begin{definition}\label{ma}
An ideal $I$ in a $C^*$-algebra $B$ is \emph{essential} if the only element $b$ of $B$ that satisfies $bI=\{0\}$ is $b=0$.

Let $A$ be a $C^*$-algebra.  The \emph{multiplier algebra} of $A$, denoted $M(A)$, is a $C^*$-algebra that contains $A$ as an essential ideal, and with the universal property that if $B$ is another $C^*$-algebra containing $A$ as an essential ideal, then there is a (unique injective)\footnote{As $B$ is essential in $A$, there can be at most one such extension, and if it exists, it must be injective.} $*$-homomorphism $B\to M(A)$ extending the identity on $A$.
\end{definition}

See for example \cite[Chapter 2]{Lance:1995ys} or \cite[Section 1.7]{Willett:2010ay} for more on multiplier algebras, including a proof that they exist (the definition implies they are unique, up to canonical isomorphism).  For example, if $\pi \colon A\to \mathcal{B}(H)$ is any faithful non-degenerate representation, then one may take 
$$
M(A)\coloneqq\{b\in \mathcal{B}(H)\mid b\pi(a),\pi(a)b\in A\text{ for all } a\in A\}
$$
to be the idealizer of $A$ in $\mathcal{B}(H)$.  Note that $A$ is unital if and only if $M(A)=A$.  Note also that if $\pi \colon A\to \mathcal{B}(H)$ is a nondegenerate representation, then there is a unique unital extension $\pi^M \colon M(A)\to \mathcal{B}(H)$.

We now define the twisted analogs of covariant pairs and crossed products.  The idea of a twisted action associated to a short exact sequence $1\to K\to \Gamma\to \Lambda \to 1$ is that it plays the role that `should' be played by a $\Lambda$-action in situations where the latter does not exist due to the sequence not splitting.  The definitions we use are due to Green \cite{Green:1978aa}.  See \cite[Chapter 12]{Echterhoff:2009jo} for a more recent exposition.

\begin{definition}\label{twist def}
Let $\Gamma$ be a discrete group that fits into a short exact sequence $1\to K\to \Gamma\to \Lambda \to 1$.  A \emph{$(\Gamma,K)$-action} on a $C^*$-algebra $A$ is an action $\alpha$ of $\Gamma$ in the usual sense together with a homomorphism 
$$
\tau \colon K\to U(M(A))
$$
from $K$ to the unitary group of the multiplier algebra $M(A)$ such that 
$$
\tau_ka\tau_k^*=\alpha_k(a)\quad \text{and}\quad \alpha_g(\tau_k)=\tau_{gkg^{-1}}
$$
for all $a\in A$, $k\in K$, and $g\in \Gamma$.   We also call the triple $(A,\alpha,\tau)$, or just $A$ if there is no risk of confusion, a \emph{$(\Gamma,K)$-$C^*$-algebra}.

A \emph{covariant pair} for a $(\Gamma,K)$-$C^*$-algebra $(A,\alpha,\tau)$ is a covariant pair $(\pi,u)$ for $(A,\alpha)$ in the usual sense and with the additional property that $\pi^M(\tau_k)=u_k$ for all $k\in K$.  Finally, the \emph{(maximal) twisted crossed product} of $A$ by $(\Gamma,K)$, denoted $A\rtimes_{\alpha,\tau}(\Gamma,K)$ or just $A\rtimes (\Gamma,K)$ if there is no risk of confusion, is defined to be the separated\footnote{This means one first takes the quotient by the subspace of $b\in A\rtimes_{alg}\Gamma$ satisfying $\|b\|=0$, then completes.} completion of $A\rtimes_{alg}\Gamma$ for the seminorm
$$
\|b\|\coloneqq\sup\{\|(\pi\rtimes u)(b)\|_{\mathcal{B}(H)}\mid (\pi,u)\text{ a covariant pair}\}.
$$
\end{definition}

\begin{example}\label{gp ex}
The basic, and most important, example of a twisted action is when $A=C^*(K)$, $\Gamma$ acts on $A$ via the conjugation action of $\Gamma$ on $K$, and $\tau \colon K\to U(A)$ is the canonical inclusion into the unitary group.   

In this case, the twisted crossed product $C^*(K)\rtimes (\Gamma,K)$ is canonically isomorphic to $C^*(\Gamma)$: this comes down to the fact that a covariant pair for $(C^*(K),\alpha,\tau)$ is essentially the same thing as a unitary representation of $\Gamma$.
\end{example}

\begin{example}\label{reduce}
Let $1\to K\to \Gamma\to \Lambda \to 1$ be a short exact sequence of groups.  We say a $(\Gamma,K)$ twisted action \emph{reduces to a $\Lambda$-action} if $\tau$ sends every element in $K$ to the identity, and $\alpha$ descends to an action $\overline{\alpha}$ of $\Lambda$.   In this case there is a canonical isomorphism 
$$
A\rtimes_{\alpha,\tau}(\Gamma,K)\cong A\rtimes_{\overline{\alpha}}\Lambda.
$$
This comes down to checking that for any covariant pair $(\pi,u)$ for $(A,\alpha,\tau)$, the integrated form $\pi\rtimes u$ factors through the canonical quotient map $A\rtimes_{alg,\alpha}\Gamma\to A\rtimes_{alg,\overline{\alpha}}\Lambda$, and that conversely any covariant pair for $(A,\overline{\alpha})$ arises from a twisted covariant pair for $(A,\alpha,\tau)$.
\end{example}

We are now ready to state the key facts we need about twisted actions, crossed products, and Morita equivalences.  We need some notation.  Let $1\to K\to \Gamma\to \Lambda\to 1$ be a short exact sequence of groups, and let $(A,\alpha,\tau)$ be a $(\Gamma,K)$-$C^*$-algebra.  Let $C_0(\Lambda,C^*(K))$ denote the $C^*$-algebra of bounded functions $f \colon \Lambda\to C^*(K))$ such that for all $\epsilon>0$ there is a finite subset $S$ of $\Lambda$ such that $\|f(g)\|<\epsilon$ for all $g\not\in S$.  The $C^*$-algebra $C_0(\Lambda,C^*(K))$ is equipped with pointwise operations, and with the $(\Gamma,K)$ action $(\beta,\sigma)$ defined for $f \colon \Lambda\to C^*(K)$ by 
$$
(\beta_gf)(h)\coloneqq\mathrm{ad}_g(f(g^{-1}h))
$$
(here $\mathrm{ad}$ is the action of $\Gamma$ on $C^*(K)$ induced by the conjugation $\Gamma$ action on $K$) and setting $\sigma_k$ to be the canonical unitary $k\in  C^*(K)$, considered as an element of the multiplier algebra $M(C_0(\Lambda,C^*(K)))$ via the formula 
$$
(\sigma_k f)(g)=k\cdot f(g)
$$
for $k\in K$, $g\in \Lambda$ and $f\in C_0(\Lambda,C^*(K))$.
We then define 
\begin{equation}\label{b def}
B \coloneqq C_0(\Lambda,C^*(K))\rtimes_{\beta,\sigma}(\Gamma,K),
\end{equation}
and note that the right action of $\Lambda$ on itself makes this into a $\Lambda$-$C^*$-algebra (i.e.\ to a $(\Gamma,K)$-$C^*$-algebra such that the right action reduces to a $\Lambda$ action in the sense of Example \ref{reduce}).

The following theorem is due to Echterhoff \cite{Echterhoff:1994aa} (based partly on ideas of Packer and Raeburn \cite{Packer:1990aa}).  See also \cite[Chapter 13]{Echterhoff:2009jo} for a more recent exposition.

\begin{theorem}[Echterhoff]\label{twist the}
Let $1\to K\to \Gamma\to \Lambda\to 1$ be a short exact sequence of groups.  Then there is a notion of equivariant Morita equivalence for $(\Gamma,K)$-$C^*$-algebras such that the following hold.
\begin{enumerate}[(i)]
\item \label{mor for} If $(A,\alpha,\tau)$ and $(B,\beta,\sigma)$ are equivariantly Morita equivalent $(\Gamma,K)$-$C^*$-algebras, then $A$ and $B$ are themselves Morita equivalent as $C^*$-algebras.
\item \label{mor cp} If $(A,\alpha,\tau)$ and $(B,\beta,\sigma)$ are equivariantly Morita equivalent $(\Gamma,K)$-$C^*$-algebras, then the crossed products $A\rtimes_{\alpha,\tau}(\Gamma,K)$ and $B\rtimes_{\beta,\sigma}(\Gamma,K)$ are Morita equivalent.
\item \label{mor red} Let $(A,\alpha,\tau)$ be a $(\Gamma,K)$-$C^*$-algebra, and let $B$ be the $C^*$-algebra constructed out of this data as in line \eqref{b def}.  Then $(A,\alpha,\tau)$ is equivariantly Morita equivalent to a $(B,\beta,\sigma)$ that reduces to a $\Lambda$-$C^*$-algebra in the sense of Example \ref{reduce}.
\end{enumerate}
\end{theorem}

\begin{proof}
We just point out where in \cite{Echterhoff:1994aa} to find the relevant material.  The notion of equivariant Morita equivalence is introduced in \cite[Definition 1, page 173]{Echterhoff:1994aa}; it is built on the standard notion of Morita equivalence by adding more structure whence part \eqref{mor for} is clear.  Part \eqref{mor cp} is part (2) of the Remark on \cite[pages 173-4]{Echterhoff:1994aa}.  Finally, part \eqref{mor red} is \cite[Theorem 1, page 177]{Echterhoff:1994aa}.
\end{proof}

\begin{proof}[Proof of Theorem \ref{struct}]
Let $A=C^*(K)$, considered as a $(\Gamma,K)$-$C^*$-algebra as in Example \ref{gp ex}.  Then $B$ as in Theorem \ref{twist the} part \eqref{mor red} is Morita equivalent to $C^*(K)$ (use part \eqref{mor for} of Theorem \ref{twist the}). Moreover, $B\rtimes \Lambda$ is Morita equivalent to $C^*(K)\rtimes_{\alpha,\tau}(\Gamma,K)$ by Theorem \ref{twist the} part \eqref{mor cp} and Example \ref{reduce}.  As $C^*(K)\rtimes_{\alpha,\tau}(\Gamma,K)=C^*(\Gamma)$ by Example \ref{gp ex} again, we are done.
\end{proof}

\begin{remark}
It seems likely that the proof of Theorem \ref{struct} can be made a bit more elementary and direct: indeed, the $C^*$-algebra $B$ appearing in the statement is given concretely by line \eqref{b def}, and one can probably establish the Morita equivalences $C^*K\sim B$ and $C^*\Gamma\sim B\rtimes \Lambda$ reasonably directly.  We wanted to go through the more general machinery, however, partly as we find it more conceptual, partly to advertise the theory of twisted crossed products, and partly as this exposition will make it clear to experts that the results generalize in various directions: for example, to locally compact groups, or to $C^*$-algebras equipped with appropriate (twisted) actions.
\end{remark}

\section{New examples of (L)LP groups}
\label{s:ex}

In this section, we show that the above results, especially Theorem \ref{llp the}, apply to several important classes of groups from topology, as well as geometric and combinatorial group theory.

\subsection{Consequences of Theorem \ref{llp the}}

\begin{corollary}\label{fi cor}
A virtually LLP group is LLP. A countable virtually LP group is LP.
\end{corollary}

Note that, combined with Theorem \ref{free gp}, this gives an alternative (Bass--Serre-free) proof of the (L)LP for virtually free groups (Corollary \ref{cd1}).

\begin{proof}
Let $\Gamma$ be a group with a finite-index subgroup $\Lambda$ with the LLP (the proof for the LP is similar). There is a finite-index normal subgroup $\Delta$ of $\Gamma$ contained in $\Lambda$.  Since $\Lambda$ has the LLP, so does $\Delta$ by Corollary \ref{subgp}.  On the other hand, $\Gamma/\Delta$ is finite, so amenable, and we are done by Theorem \ref{llp the}.
\end{proof}

\begin{remark}
It would be interesting to give a more self-contained proof of Corollary \ref{fi cor}.  Certainly one can give a simpler proof by `restricting' the proof of Theorem \ref{llp the} to this case, but the basic idea would still be the same.
\end{remark}

Next, we consider a class that has received special attention because of its connection to the Atiyah conjecture \cite{linnell}.

\begin{definition}
\emph{Linnell's class $\mathcal{C}$} is the smallest class of groups that contains free groups, is preserved by directed unions, and such that if
\[1 \to K \to \Gamma \to \Lambda \to 1\]
is a short exact sequence of groups, where $K$ is in $\mathcal{C}$ and $\Lambda$ is elementary amenable, then $\Gamma$ is in $\mathcal{C}$.
\end{definition}

\begin{remark}
\label{linnell:atmenable}
    For future reference, we remark that groups in Linnell's Class $\mathcal{C}$ are a-T-menable (equivalently have the Haagerup property: see Footnote \ref{hp}). Indeed, free groups are a-T-menable, and the class of a-T-menable groups is closed under increasing unions and extensions by amenable groups: see \cite[proof of Lemma 1.2]{Haagerup:1979rq} and \cite[Propositions 6.11 and 6.15]{Cherix:2001fk} respectively.
\end{remark}

\begin{corollary}\label{linnell cor}

If $\Gamma$ is in Linnell's class $\mathcal{C}$, then $\Gamma$ is LLP. If moreover $\Gamma$ is countable, then it is LP.
\end{corollary}

\begin{proof}
This follows from the definition of Linnell's class $\mathcal{C}$, together with the fact that free groups are LLP (Theorem \ref{free gp}), and that the LLP is preserved by directed unions (Corollary \ref{dir un}) and extensions with amenable quotient (Theorem \ref{llp the}). Similarly for the LP.
\end{proof}

A special case that will appear in our examples is the following. Recall that the commutator subgroup of a group is also called its \emph{first derived subgroup} $\Gamma' = [\Gamma, \Gamma]$, and we define the \emph{$n$-th derived subgroup} as $\Gamma^{(n)} = [\Gamma^{(n-1)}, \Gamma^{(n-1)}]$.

\begin{corollary}\label{derived cor}
Let $\Gamma$ be a group such that $\Gamma^{(n)}$ is free for some $n$. Then $\Gamma$ is LLP, and if it is countable, it is LP.
\end{corollary}

\begin{proof}
The quotient $\Gamma/\Gamma^{(n)}$ is $n$-step solvable, thus amenable, so we conclude by Theorems \ref{free gp} and \ref{llp the}.
\end{proof}

Later on we will be interested in combining the (L)LP with finite separation properties (Section \ref{s:fd}) in view of applications to representation stability (Section \ref{s:stability}). However, groups in Linnell's class $\mathcal{C}$ are in general not residually finite. So let us isolate one particular subclass for which we will be able to show these additional separation properties later on. This class has received particular attention in the past few years, sparked by the breakthrough work of Kielak--Linton \cite{Kielak:2024aa}, which was later improved upon by Fisher \cite{Fisher:2024aa} (see \cite{Fisher:2025} for a recent generalization).

A \emph{free-by-cyclic}\footnote{This should perhaps more properly be called ``free-by-(infinite cyclic)'', but following a standard (and sensible) convention, we will not do that.} group is by definition a discrete group $\Gamma$ that fits into a short exact sequence $1\to F\to \Gamma\to \Z\to 1$ with $F$ free.  We will assume that $F$ is countable, but not necessarily finitely generated; thus our free-by-cyclic groups will always be countable.  As $\Z$ is free, the sequence splits, so $\Gamma\cong F\rtimes \Z$, and we will usually consider free-by-cyclic groups in this form.

\begin{corollary}\label{vfbc}
A (countable) virtually free-by-cyclic group is LP.
\end{corollary}

\begin{proof}
A free-by-cyclic group has the LP by Theorem \ref{free gp} and Corollary \ref{sd prod}.  Hence a virtually free-by-cyclic group has the LP by Corollary \ref{fi cor}.
\end{proof}

We end with an extension of Proposition \ref{gog finite edge} that covers more graphs of groups. We refer the reader to the discussion after that section for the basic notions from Bass--Serre theory.

\begin{corollary}\label{gog amenable cor}
Let $\Gamma$ be the fundamental group of a graph of groups $\mathcal{G}$. Suppose that for every edge $e$ there exists a homomorphism from $\Gamma$ to an amenable group whose restriction to the edge group $\Gamma_e$ has finite kernel.

If all vertex groups have the LLP, then so does $\Gamma$. If moreover all vertex groups have the LP and $\Gamma$ is countable, then $\Gamma$ has the LP.
\end{corollary}

\begin{proof}
By Corollary \ref{dir un}, we may assume that the underlying graph $X$ is finite. For each edge $e$, let $A_e$ be the amenable group given in the statement. So we have a homomorphism
\[\Gamma \to \prod\limits_{e \in E} A_e,\]
whose restriction to every vertex group $\Gamma_e$ is finite. Let $K$ denote the kernel. By Theorem \ref{llp the}, it suffices to show that $K$ has the LLP (and the proof for the LP is the same).

Consider the action of $\Gamma$ on its Bass--Serre tree. Restrict this action to the normal subgroup $K$, this realizes $K$ as the fundamental group of a graph of groups whose vertex groups $K_v$ are vertex stabilizers, and whose edge groups $K_e$ are edge stabilizers. We claim that each $K_v$ has the LLP, and each $K_e$ is finite, so we conclude by Proposition \ref{gog finite edge}. Indeed, because $K$ is normal in $\Gamma$, it suffices to consider $v$ and $e$ from a fundamental domain for the action of $\Gamma$, which correspond to the vertices and edges of the original graph of groups. Now $K_v < \Gamma_v$, which has the LLP, so $K_v$ has the LLP by Corollary \ref{subgp}; and $K_e = K \cap \Gamma_e$ is finite by assumption.
\end{proof}

\begin{remark}\label{gog amenable rmk}
In practice, when we are in a position to apply Corollary \ref{gog amenable cor}, we can often say more. For example, suppose that $\Gamma$ is the fundamental group of a finite graph of groups such that for every vertex $v$, the vertex group $\Gamma_v$ embeds into $\Gamma/\Gamma^{(n)}$ for some $n$. Then $\Gamma^{(n)}$ is free for some $n$, and we are in the case of Corollary \ref{derived cor}.

Indeed, as in the proof of Corollary \ref{gog amenable cor}, we deduce that the action of $\Gamma^{(n)}$ on the Bass--Serre tree of $\Gamma$ has trivial vertex stabilizers, in other words it is a free action. Groups acting freely on trees are free: see for example \cite[3.3, Theorem 4]{Serre:1980aa}.
\end{remark}

\subsection{Low-dimensional manifolds}

\begin{example}\label{manifold lp}
Let $M$ be a connected (second countable) manifold of dimension at most $3$. Then $\pi_1(M)$ has the LP. Indeed, when $M$ is a $3$-manifold, $\pi_1(M)$ belongs to Linnell's class $\mathcal{C}$ \cite[Theorem 1.4]{kielak:linton:manifolds} hence has the LP by Corollary \ref{linnell cor}. The result for $1$- and $2$-manifolds follows by taking products.
\end{example}

\begin{remark}\label{surfaces rem}
In many cases, such groups are even virtually free-by-cyclic. This is the case for $3$-manifolds whose fundamental group is finitely generated and has rational cohomological dimension at most $2$, such as compact $3$-manifolds with non-empty boundary \cite[Theorem 1.1]{kielak:linton:manifolds}.

For surfaces this is easier to see. First, the fundamental group of a surface that is either non-compact or has non-empty boundary is free: see for example \cite[4.2.1 and 4.2.2]{Stillwell:1993aa}. For closed surfaces, the fundamental group is either finite in the case of the sphere and the projective plane, or maps to $\mathbb{Z}$. In the latter case, by covering theory, the kernel is the fundamental group of a non-compact surface, hence free.
\end{remark}

\begin{remark}
Example \ref{manifold lp} cannot be extended to higher dimensions. Indeed, every finitely presented group is the fundamental group of an orientable closed $4$-manifold, see for example \cite[Section 7]{delaharpe:manifolds}. This applies in particular to the groups from Example \ref{cex}, which do not have the LLP.
\end{remark}

\subsection{One-relator groups}
\label{ss:or}

A one-relator group is a group admitting a presentation of the form
$$
\Gamma=\langle x_1, \ldots, x_n\mid r\rangle,
$$ 
with finitely many generators and a single relation $r$.  We will assume that $r$ is freely and cyclically reduced\footnote{A word is \emph{freely reduced} if no consecutive pair of letters are mutual inverses, and \emph{cyclically reduced} if the first and last letters are not mutual inverses.}; this can always be arranged without changing the isomorphism class of $\Gamma$.

One-relator groups form a fundamental class in combinatorial and geometric group theory, with a long history and rich theory, both of which are surveyed in the book \cite{onerelator}. The past few years have seen significant advances: we single out important new structural understanding \cite{Louder:2022aa, Linton:2025ab}, landmark results on coherence \cite{coherence1, coherence2, coherence3}, and most relevantly for us, the proof that many one-relator groups are in fact virtually free-by-cyclic \cite{Kielak:2024aa}, which we will use in many of our corollaries.

\begin{remark}\label{one rel fg}
If $\Gamma=\langle X\mid r\rangle$ is an infinitely generated one-relator group, let $X_0$ be the (finite) set of letters in $X$ that appear in the word $r$.  Then $\Gamma$ splits as a free product 
$$
\Gamma=\langle X_0\mid r\rangle \Asterisk F_{X\setminus X_0}
$$
of a finitely generated one relator group and the free group on the set $X\setminus X_0$.  Using Corollary \ref{fp gp}, this implies that one can deduce the LP for $\Gamma$ from results about the LP for finitely generated one-relator groups. Hence the finite generation assumption we make on one-relator groups is not so important.
\end{remark}

\subsubsection*{Virtually free-by-cyclic examples}

Thanks to a recent work of Kielak--Linton \cite{Kielak:2024aa}, and some more classical results, in many cases one-relator groups are virtually free-by-cyclic, hence LP by Corollary \ref{vfbc}.

\begin{example}\label{or tor}
A one-relator group has torsion if and only if the relation $r$ is a proper power: see for example \cite[Propositions 5.17 and 5.18]{Lyndon:2001sh}.  Any one-relator group with torsion is virtually free-by-cyclic \cite[Corollary 1.2]{Kielak:2024aa}.
\end{example}

\begin{example}\label{or ni}
An important class is that of one-relator groups with \emph{negative immersions}. The easiest characterization to understand, due to Louder--Wilton \cite[Theorem 1.5 and Corollary 1.10]{Louder:2022aa}, is that a one-relator group has negative immersions if and only if every $2$-generated subgroup is free. Such groups are also virtually free-by-cyclic \cite[Corollary 1.4]{Kielak:2024aa}.
\end{example}

\begin{remark}\label{rem pr}
Let $r$ be a (freely and cyclically reduced) word in a finitely generated free group $F$.  Given a subgroup $\Lambda$ of $F$ (which is necessarily itself free), we write $\mathrm{rk}(\Lambda)$ for the rank of $\Lambda$, i.e.\ the minimal number of elements needed to generate $\Lambda$.  If moreover $\Lambda$ contains $r$, we say that $r$ is \emph{primitive} in $\Lambda$ if it is part of a free basis for $\Lambda$.  Following Puder \cite[Definition 1.7]{Puder:2014aa}, the \emph{primitivity rank} of $r$ is defined by 
$$
\pi(r)\coloneqq\min\{\mathrm{rk}(\Lambda)\mid ~r\text{ not primitive in }\Lambda\}
$$
(or infinite if $r$ is already primitive in $F$).

Note that if $\pi(r)<\infty$, then it is bounded above by the number of generators of $F$.  Note also that $\pi(r)=1$ if and only if $r$ is a proper power, so in that case we are in the torsion case of Example \ref{or tor}. Louder--Wilton \cite[Theorem 1.3]{Louder:2022aa} showed that $\pi(r) > 2$ if and only if the one-relator group has negative immersions. So we can summarize the previous two examples by saying that a one-relator group with $\pi(r) \neq 2$ is virtually-free-by-cyclic. Another way to characterize this class is as one-relator groups that are locally quasi-convex and hyperbolic \cite[Theorem 1.4]{linton:mappingtori}.
\end{remark}

\begin{example}\label{or sc}
A word $r$ in the free group satisfies the \emph{$C'(\lambda)$ small cancellation condition} if any subword of $r$ that appears in two distinct ways has length less than $\lambda$ times the length of $r$ \cite[Chapter V]{Lyndon:2001sh}. The typical example is the fundamental group of a closed orientable surface of genus at least $2$. If $r$ satisfies the $C'(1/6)$ small cancellation condition, then the corresponding one-relator group is virtually free-by-cyclic \cite[Corollary 1.5]{Kielak:2024aa}\footnote{The reference says ``small cancellation'' without specifying the parameter $\lambda$, but $1/6$ is implicit. For the skeptical reader, this is the condition that implies that the group is hyperbolic \cite[Section 4.7]{Gromov:987fv} and cocompactly cubulated \cite{Wise:2004aa} hence virtually compact special \cite{Wise:2021aa, Agol:2013aa}, which is needed to apply \cite[Theorem 1.1]{Kielak:2024aa}.}.
\end{example}

\begin{example}\label{or random}
A \emph{random} one-relator group is virtually free-by-cyclic. Indeed, such groups satisfy the $C'(1/6)$ small cancellation condition with overwhelming probability \cite[Section 9.B]{Gromov:1993tr}, so we are in the setting of Example \ref{or sc}.

More generally, Kielak--Kropholler--Wilkes \cite[Theorem B]{Kielak:2022aa} show that for $n \geq 3$, a random group with $n$ generators and at most $n-2$ relators is virtually free-by-cyclic\footnote{They state that the group embeds into a virtually (finitely generated free)-by-cyclic group, but this implies that the group itself is virtually free-by-cyclic.}, so LP. In the case of $n-1$ relators, it is free-by-cyclic with probability bounded away from $0$ \cite[Theorem A]{Kielak:2022aa} and $1$ \cite[Theorem D]{kudlinska} (see also \cite{Dunfield:2006aa} for the one-relator case).
\end{example}

\begin{remark}
Example \ref{or random} is about the \emph{few-relator} model of random groups, also known as the \emph{density $0$} model. An alternative widely used model is the \emph{density} model, which depends on a density parameter $d \in (0, 1)$. When $d > 1/2$, the random group has order at most $2$, while when $d < 1/2$, the random group is still hyperbolic \cite[Section 9.B]{Gromov:1993tr} but not small cancellation anymore. In fact, at densities $1/3 < d < 1/2$, the random group has property $(T)$\footnote{Small cancellation groups act properly on CAT(0) cube complexes \cite{Wise:2004aa}, hence they cannot have property $(T)$ \cite{niblo:reeves}.} \cite{Zuk:2003jf}, and therefore seems unlikely to have the LLP (see \cite[Proposition 1.7]{Dogon:2023aa} for a related result).
\end{remark}

\begin{example}\label{or center}
One-relator groups with non-trivial center are free-by-cyclic \cite{baumslagtaylor}.
\end{example}

\subsubsection*{Baumslag--Solitar groups}

In all of the previous examples, the one-relator groups were virtually free-by-cyclic, hence in particular residually finite \cite{vfbc:rf}.  We now look at an important class of one-relator groups that includes many non-residually finite examples. For $m, n \in \mathbb{Z} \setminus \{ 0 \}$, the \emph{Baumslag--Solitar group $BS(m, n)$} is
\[BS(m, n) \coloneqq \langle a, t \mid t^{-1} a^m t = a^{-n} \rangle.\]

\begin{remark}\label{bs conventions}
Clearly $BS(m, n) \cong BS(n, m) \cong BS(-m, -n)$. Moreover, $BS(m, n)$ and $BS(m, -n)$ have a common finite-index subgroup \cite[Lemma 6.1]{bs:commensurable}, so their (L)LP status is the same, by Corollaries \ref{subgp} and \ref{fi cor}. Hence for our purposes we may assume that $0 < m \leq n$.
\end{remark}

Baumslag--Solitar groups were introduced by Baumslag and Solitar in their landmark paper \cite{Baumslag:1962aa}, and they are still some of the most useful examples in combinatorial group theory. We split them into three classes, following a statement in the original paper \cite{Baumslag:1962aa} later corrected by Meskin \cite{Meskin:1972ab}.

\begin{example}\label{bs solvable}
There is an isomorphism
\begin{align*}
BS(1, n) &\to \mathbb{Z}[1/n] \rtimes_{n} \mathbb{Z} \\
a &\mapsto 1 \in \mathbb{Z}[1/n]; \\
t &\mapsto 1 \in \mathbb{Z}.
\end{align*}
The action of $\mathbb{Z}$ is induced by multiplication by $n$. In particular, $BS(1, n)$ is metabelian, hence amenable, hence has the LP by Corollary \ref{amen cor}.
\end{example}

\begin{example}\label{bs unimodular}
For $n \geq 1$, the group $BS(n, n)$ contains the non-trivial central elment $a^n$, hence it is free-by-cyclic by Example \ref{or center}.
\end{example}

The remaining Baumslag--Solitar groups are not residually finite, the main example being $BS(2, 3)$ which was the focus of the original paper \cite{Baumslag:1962aa}. Nevertheless, they are still LP. Indeed, Baumslag and Solitar claim without proof that $BS(m, n)$ has a free second derived subgroup. To see this, by Remark \ref{gog amenable rmk} it suffices to show that $BS(m, n)$ maps to a metabelian group into which $\langle a \rangle$ embeds. This map is analogous to the isomorphism in Example \ref{bs solvable}:
\begin{align*}
BS(m, n) &\to \mathbb{Q} \rtimes_{m/n} \mathbb{Z} \\
a &\mapsto 1 \in \mathbb{Q}; \\
t &\mapsto 1 \in \mathbb{Z}. \qedhere
\end{align*}

This argument works also for \emph{Generalized Baumslag--Solitar (GBS) groups}. These are fundamental groups of graphs of groups all of whose vertex and edge groups are infinite cyclic; Baumslag--Solitar groups correspond to the case where the graph is a single loop on one vertex. Kropholler \cite[Theorem C]{kropholler:GBS} characterizes GBS groups as the groups of cohomological dimension (at most) $2$ that have a commensurated\footnote{A subgroup $\Lambda$ of $\Gamma$ is \emph{commensurated} if $g\Lambda g^{-1}\cap \Lambda$ has finite index in both $\Lambda$ and $g\Lambda g^{-1}$ for all $g\in \Gamma$.} infinite cyclic subgroup; in particular all non-cyclic groups of cohomological dimension $2$ with non-trivial center such as the ones from Example \ref{or center}, are GBS groups.

\begin{example}\label{GBS}
GBS groups have free second derived subgroup \cite[Corollary 2]{kropholler:GBS} hence they are LP by Corollary \ref{derived cor}.
\end{example}

\subsubsection*{Residually solvable examples}

There are several more examples of one-relator groups with a free derived subgroups, all of these will be LP by Corollary \ref{derived cor}. If a group has a free derived subgroup, then it is residually solvable. A characterization is lacking, however Linton \cite{Linton:2025aa} characterized a slightly stronger property: that of being residually \emph{rationally} solvable; \cite[Corollary 1.2]{Linton:2025aa} implies that all such groups have a free derived subgroup, hence are LP. We single out the following special case.

\begin{example}
Let $\Gamma = \langle x_1, \ldots, x_n \mid r \rangle$ be a \emph{positive} one-relator group, namely no inverse of a generator appears in $r$. Then $\Gamma$ has a free derived subgroup \cite[Corollary 1.3]{Linton:2025aa}, hence it is LP by Corollary \ref{derived cor}.

Note that $BS(2, 3)$ can be written as a positive one-relator group, namely
\[\langle a,b\mid a^{-1}(ab)^2a(ab)^3\rangle,\]
upon reducing the initial $a^{-1} a$ \cite[Page 166]{Baumslag:1971aa}. This shows that positive one-relator groups need not be residually finite.
\end{example}

\subsubsection*{Baumslag--Gersten examples}

An interesting example of a one-relator group that is not residually solvable, nor residually finite (hence not virtually free-by-cyclic \cite{vfbc:rf}) is the \emph{Baumslag--Gersten group}
\[B \coloneqq \langle a,t\mid [a, tat^{-1}]=a\rangle.\]
This was introduced by Baumslag \cite{Baumslag:1969aa}, who proved that every finite quotient is cyclic, factoring through the retraction onto $\langle t \rangle$, with kernel (the normal closure of $a$) an infinitely generated perfect group. This example rose to prominence through the work of Gersten \cite{gersten}, who proved the very surprising fact that the Dehn function of $B$ grows faster than any tower of exponentials.

The construction was generalized by Baumslag, Miller and Troeger \cite{reflections} who defined for each $r, w \in \langle a, b \rangle$:
\[G_{r, w} \coloneqq \langle a, b \mid [r, wrw^{-1}] = r\rangle,\]
so $B = G_{a, b}$; they show that all finite quotients of $G_{r, w}$ factor through the one-relator group $\langle a, b \mid r \rangle$. We can rewrite the presentation as:
\[G_{r, w} = \langle a, b \mid (wrw^{-1})r(wrw^{-1})^{-1} = r^2\rangle,\]
which realizes $G_{r, w}$ as an HNN extension of $BS(1, 2)$ with cyclic associated subgroups. Berlai further generalized this recently \cite{berlai} by allowing other Baumslag--Solitar groups in the base:
\[G_{r, w}(l, k) \coloneqq \langle a, b \mid (wrw^{-1})r^l(wrw^{-1})^{-1} = r^k\rangle,\]
so $G_{r,w} = G_{r, w}(1, 2)$.

\begin{proposition}\label{prop:BG}
    For every $k, n \in \Z \setminus \{0\}$, the groups $G_{a, b^n}(1, k)$ are in Linnell's class $\mathcal{C}$, hence they are LP.
\end{proposition}

\begin{proof}
    The ``hence'' follows from Corollary \ref{linnell cor}. Let us start with $n = 1$. As we mentioned above, $G_{a, b}(1, k)$ can be realized as the HNN extension of $BS(1, k) = \langle a, t \mid tat^{-1} = a^2\rangle$, with stable letter $b$ identifying $bab^{-1} = t$. We can write $G_{a, b}(1, k) = \Lambda \rtimes \langle b \rangle$, where by \cite[Theorem 2.17.1]{bogopolski} (see also the proof of \cite[Theorem A]{berlai}), there is a graph of groups decomposition for $\Lambda$ whose underlying graph is a bi-infinite line:
\begin{center}
\begin{tikzpicture}[scale=0.85, baseline]

\tikzset{
  Hlabel/.style={font=\normalsize},
  Alabel/.style={font=\footnotesize}
}

\draw (-6,0) -- (6,0);

\node at (-6.5,0) {$\cdots$};
\node at (6.5,0) {$\cdots$};

\foreach \x/\label in {
    -5/\Lambda_{-2},
    -2.5/\Lambda_{-1},
    0/\Lambda_{0},
    2.5/\Lambda_{1},
    5/\Lambda_{2}
}{
    \draw (\x,0) circle (0.12);
    \node[Hlabel, above=3pt] at (\x,0) {$\label$};
}

\node[Alabel, below=1pt] at (-3.75,0) {$t_{-2} = a_{-1}$};
\node[Alabel, below=1pt] at (-1.25,0) {$t_{-1} = a_0$};
\node[Alabel, below=1pt] at (1.25,0) {$t_0 = a_1$};
\node[Alabel, below=1pt] at (3.75,0) {$t_1 = a_2$};
\end{tikzpicture}
\end{center}
Here $\Lambda_i = \langle a_i, t_i \mid t_i a_i t_i^{-1} = a_i^k \rangle \cong BS(1, k)$. By stability under extensions by $\Z$, it suffices to show that $\Lambda$ is in $\mathcal{C}$. By stability under directed unions, it suffices to show that the following finite graph of groups has fundamental group in $\mathcal{C}$:
\begin{center}
\begin{tikzpicture}[scale=0.85, baseline]

\tikzset{
  Hlabel/.style={font=\normalsize},
  Alabel/.style={font=\footnotesize}
}

\draw (-5,0) -- (2,0);
\draw (3,0) -- (5,0);

\node at (2.5,0) {$\cdots$};

\foreach \x/\label in {
    -5/\Lambda_{0},
    -2.5/\Lambda_{1},
    0/\Lambda_{2},
    5/\Lambda_{m}
}{
    \draw (\x,0) circle (0.12);
    \node[Hlabel, above=3pt] at (\x,0) {$\label$};
}

\node[Alabel, below=1pt] at (-3.75,0) {$t_0 = a_1$};
\node[Alabel, below=1pt] at (-1.25,0) {$t_1 = a_2$};
\node[Alabel, below=1pt] at (1.25,0) {$t_2 = a_3$};
\node[Alabel, below=1pt] at (3.75,0) {$t_{m-1} = a_m$};
\end{tikzpicture}
\end{center}
Call this group $\Gamma_m$; we prove by induction on $m$ that $\Gamma_m$ is in $\mathcal{C}$. To start, $\Gamma_0 = \Lambda_0 \cong BS(1, k)$ is solvable\footnote{In Example \ref{bs solvable} we only considered positive $k$, but as explained in Remark \ref{bs conventions}, $BS(1, -k)$ contains $BS(1, k)$ as an index $2$ subgroup.}, hence in $\mathcal{C}$. Now suppose by induction that $\Gamma_{m-1}$ is in $\mathcal{C}$, we can then write $\Gamma_m = \Gamma_{m-1} \Asterisk_{t_{m-1} = a_m} \Lambda_m$, and let $T$ be the Bass--Serre tree for this amalgamated product decomposition. There is a well-defined retraction $\Gamma_{m-1} \to \langle t_{m-1}\rangle$, which extends to a retraction $\rho \colon \Gamma_m \to \Lambda_m \cong BS(1, k)$; since $BS(1, k)$ is solvable, it remains to show that $\ker(\rho)$ is in $\mathcal{C}$. The vertex group $\Lambda_m$, and in particular the edge group $\langle a_m\rangle$, maps injectively under $\rho$, hence $\ker(\rho)$ has trivial edge stabilizers for its action on $T$, and all of its vertex stabilizers are conjugate into $\Gamma_{m-1}$ (see the proof of Corollary \ref{gog amenable cor}). It follows from the fundamental theorem of Bass--Serre theory that $\ker(\rho)$ is a free product of free groups and subgroups of $\Gamma_{m-1}$, hence it is in $\mathcal{C}$ by \cite[Proposition 1.3]{kielak:linton:manifolds}.

This concludes the proof that $G_{a, b}(1, k)$ is in $\mathcal{C}$. For the general case, mapping $a \mapsto a, b \mapsto b^{-1}$ is an automorphism, hence we may assume that $n \geq 1$. Once again, it suffices to show that the kernel of the retraction $G_{a, b^n}(1, k) \to \langle b \rangle$ is in $\mathcal{C}$. The proof of \cite[Theorem A]{berlai} shows that this kernel is a free product of $n$ copies of $\ker(\rho)$ from the $n = 1$ case, which we proved is in $\mathcal{C}$, so we conclude again by \cite[Proposition 1.3]{kielak:linton:manifolds}.
\end{proof}

We do not know if all one-relator groups have the LP (Question \ref{q:or:llp}). A related open question due to Arzhantseva is whether all one-relator groups are residually amenable \cite[Problem 18.6]{kourovka}. The family $G_{r, w}(l, k)$ already contains some candidates that we are not able to tackle, for example whenever $BS(l, k)$ is non-amenable the argument above does not apply.

\begin{remark}
\label{LP:soficity}
We remark that all of our proofs of the LP so far imply soficity as well: this is because free groups are sofic, and soficity is preserved by taking subgroups, directed unions, free products and amenable extensions \cite{elek:szabo}. In fact the proof of Proposition \ref{prop:BG} is inspired from the proof of soficity of Berlai \cite{berlai}, which extends previous results for subfamilies of $G_{r, w}$ \cite{bannon1, bannon2}. As mentioned in the introduction of \cite{berlai}, for most choices of $r, w$, the soficity of $G_{r, w}$ is not known (see also \cite[Question 4.9]{pestov}), so proving the LP for them would need a very different approach.
\end{remark}

\subsection{Limit groups}

A finitely generated group is called a \emph{limit group} if it is fully residually free; that is, for every finite subset, there is a free quotient into which the finite subset embeds. Limit groups are a fundamental class in geometric group theory, especially because of their relation with the works of Sela \cite{tarski:sela} and Kharlampovich--Myasnikov \cite{tarski:KM} on Tarski's problem about the first-order theory of free groups\footnote{They have, however, been studied much earlier than this, see for example \cite{baumslag:residuallyfree}.}. A typical example of a limit group is the fundamental group of a closed orientable surface. We refer the reader to \cite{limitgroups} for an exposition of the theory.

\begin{example}\label{limit lp}
    Let $\Gamma$ be a limit group. Then there exists a free normal subgroup $\Lambda < \Gamma$ such that $\Gamma/\Lambda$ is torsion-free nilpotent \cite{kochloukova}. In particular, $\Gamma$ is LP by Theorem \ref{llp the}.
\end{example}

\begin{example}
\label{graphs of free groups}
    A closely related class of groups is that of fundamental groups of graphs of groups with free vertex groups and cyclic edge groups (see for example \cite{wilton:graphpairs}). If such a group does not contain a subgroup isomorphic to $BS(m, n)$ with $|m| \neq |n|$, then it is virtually free-by-cyclic \cite[Corollary C]{hagen:wise:elementary}, hence LP.
\end{example}

\subsection{Right-angled Artin groups}

A \emph{right-angled Artin group} (RAAG) is a group $\Gamma$ associated to a finite graph $X$ as follows: for each vertex one includes a generator, and for each edge connecting vertices $x$ and $y$ one imposes the relation $xy=yx$.

RAAGs play a fundamental role in group theory and geometry, thanks to the fact that they contain all virtually compact special groups \cite{Haglund:2008aa}; this is essential in the proof of the virtual fibering conjecture for hyperbolic $3$-manifolds \cite{Wise:2021aa, Agol:2013aa}. We refer the reader to \cite{charney:intro} for an introduction.

\begin{example}\label{raag lp}
A graph is \emph{chordal} if every cycle of length at least $4$ admits a chord. Then Servatius--Droms--Servatius \cite[Theorem 2]{Servatius:1989aa} show that the defining graph of a RAAG $\Gamma$ is chordal if and only if the derived subgroup $\Gamma'$ is free, in which case $\Gamma$ has the LP by Corollary \ref{derived cor}.
\end{example}

\begin{remark}\label{racg}
    A closely related class is that of \emph{right-angled Coxeter groups} (RACG). This is defined similarly to a RAAG, with the added relations that the vertex generators have order $2$. RACGs on chordal graphs are also LP, but for an easier reason: they are virtually free. Indeed, the argument for RAAGS \cite[Theorem 2]{Servatius:1989aa} goes through verbatim to show that a RACG on a chordal graph has free derived subgroup; in this case the abelianization is generated by elements of order $2$, hence the derived subgroup has finite index.
\end{remark}

It is not clear whether a general RAAG will have the (L)LP. Note that $F_2\times F_2$ is a basic example of a RAAG, and as we previously mentioned, the (L)LP for this is quite open (Question \ref{q:ftimesf}).

\section{Property FD}
\label{s:fd}

In Section \ref{s:stability}, we will discuss some of the main applications of the (L)LP in representation stability. These applications require additional separation properties for the group $C^*$-algebra, which we introduce in this section. We lay some background and recall some known facts, and then prove that property FD is preserved by free products (Corollary \ref{fd fp cor}) and more generally free products amalgamated along amenable retracts (Corollary \ref{fd am}).

\subsection{Background on representation theory}

We will start by discussing some background on representation theory of $C^*$-algebras, based based on \cite{Fell:1962aa, Exel:1992aa}.  We are only interested in this material for applications to group theory, but state the results in general, as they are no harder to prove.

Contrary to some parts of the literature, we will allow possibly degenerate representations of $C^*$-algebras (see Definition \ref{dg rep} above).  For any representation $(\pi,H)$ of a $C^*$-algebra $A$ there is a splitting $H=H_1\oplus H_0$ into invariant subspaces where $H_1$ is the closed span of $\pi(A)H$, and $H_0=H_1^\perp$.  The subspace $H_1$ is called the \emph{essential space} of $\pi$, and is the largest subspace of $H$ on which $A$ acts nondegenerately.  On the other hand, $A$ acts as the zero representation on the (possibly non-zero) Hilbert space $H_0$.

Note also that unitary representations of a group $\Gamma$ are in one-one correspondence with nondegenerate representations of $C^*\Gamma$.  Thus general representations of $C^*\Gamma$ are in one-one correspondence with the following data: a Hilbert space $H$ and a unitary representation of $\Gamma$ on a closed subspace $H_1$ of $H$, with $\Gamma$ thought of as `acting via the zero operator' on the orthogonal complement $H_1^\perp$.

The following definition is based on \cite[Section 1]{Fell:1962aa}. 

\begin{definition}\label{rep a h}
Let $H$ be a Hilbert space, and let $A$ be a $C^*$-algebra.  Write $\mathrm{Rep}(A,H)$ for the set of (possibly degenerate) representations of $A$ on $H$. An element of $\mathrm{Rep}(A,H)$ is called \emph{essentially cyclic} if its restriction to its essential space is cyclic.

Let $\pi\in \mathrm{Rep}(A,H)$ with essential space $H_\pi\leq H$, let $H_0$ be a finite subset of $H_\pi$, let $A_0$ be a finite subset of $A$, let $\epsilon>0$, and define
$$
U(\pi;H_0,A_0) \coloneqq \{\rho\in \mathrm{Rep}(A;H)\mid \|\pi(a)\xi-\rho(a)\xi\|_H<\epsilon \text{ for all } a\in A_0,\xi\in H_0\}.
$$
We equip $\mathrm{Rep}(A,H)$ with the topology generated by these sets.  Note that a net $(\pi_i)_i$ converges to $\pi$ in this topology if and only if 
$$
\|\pi_i(a)\xi-\pi(a)\xi\|_H \xrightarrow{i \to \infty} 0
$$
for all $a\in A$ and $\xi$ in the essential space of $\pi$. 
\end{definition}

There are several small (but sometimes important) variants of the topology above in the literature.  Let us mention two of these to help situate the reader.

\begin{remark}\label{fell vs el}
Much of what we do with representation theory is based on work of Exel--Loring \cite{Exel:1992aa}.  We warn the reader that the topology on $\mathrm{Rep}(A,H)$ from Definition \ref{rep a h} is not the same as that from \cite[Definition 2.1]{Exel:1992aa}: to define the topology of \cite[Definition 2.1]{Exel:1992aa} one allows $H_0$ to be any finite subset of $H$ (and not just of the essential subspace for $\pi$).
\end{remark}

\begin{remark}\label{fell top rem}
Let us mention the relationship of the topology from Definition \ref{rep a h} to the \emph{Fell topology}.  We will not define the latter here, but see for example \cite[Appendix F]{Bekka:2000kx} for a modern textbook discussion of the Fell topology for (locally compact) groups.

The Fell topology is defined on unitary equivalence classes of  representations of a $C^*$-algebra (one bounds the dimension of the representations considered by some cardinal $\aleph$ to avoid set-theoretic difficulties).  It is then defined on unitary equivalence classes of unitary representations of a group by restricting to the special case of group $C^*$-algebras.    The Fell topology was introduced (for $C^*$-algebras and groups) by Fell in \cite[Section 2]{Fell:1962aa}, where it is called the \emph{inner hull-kernel topology}.  

The topology on $\mathrm{Rep}(A,H)$ from Definition \ref{rep a h} comes from the same paper by Fell \cite[Section 1]{Fell:1962aa}.  Passing to equivalence classes, it also induces a topology on the set of unitary equivalence classes of unitary representations of dimension at most $\aleph:=\mathrm{dim}(H)$; Fell calls this induced topology the \emph{quotient topology}.  Fell's quotient topology and inner hull-kernel topology are not the same, but are very closely related.  For example they agree when restricted to (equivalence classes of) irreducible representations, and infinite multiplicity representations: see \cite[Lemmas 2.3, 2.4 and surrounding discussion]{Fell:1962aa}.
\end{remark}

\begin{definition}\label{s state}
Let $A$ be a $C^*$-algebra, and let $S$ be a collection of representations of $A$.  Let $\langle S\rangle$ be the collection of representations of $A$ that are unitarily equivalent to a subrepresentation of a finite direct sum of representations in $S$.

A state $\phi$ on $A$ is an \emph{$S$-state} if its GNS representation is in $\langle S\rangle$.  For a Hilbert space $H$, a representation $\pi\in \mathrm{Rep}(A,H)$ is an \emph{approximately-$S$ representation, relative to $H$} if there is a net $(\pi_i)_i$ in $\mathrm{Rep}(A,H)$ that converges to $\pi$, and with each $\pi_i$ in $\langle S\rangle$.
\end{definition}

\begin{remark}
    If $S$ is closed under finite direct sums, unitary equivalence and taking subrepresentations, then $\langle S \rangle = S$. In this case the notion of approximation above is really with respect to elements of $S$, and this will be used in Lemma \ref{s corner} below.
\end{remark}

The following is very closely related to \cite[Theorem 2.4]{Exel:1992aa}: we give a proof mainly to explain where it differs from that reference.

\begin{theorem}\label{rep approx}
Let $A$ be a $C^*$-algebra and let $S$ be a family of representations of $A$.  Let $\aleph$ be an infinite cardinal that is strictly larger than the dimension of all cyclic representations of $A$.  Then the following are equivalent:
\begin{enumerate}[(i)]
\item \label{rep state} the collection of $S$-states is weak-$*$ dense in the state space of $A$;
\item \label{rep cyclic} for any Hilbert space $H$ of dimension at least $\aleph$, any essentially cyclic representation $\pi\in \mathrm{Rep}(A,H)$ is an approximately-$S$ representation relative to $H$;
\item \label{rep gen} for any Hilbert space $H$ of dimension at least $\aleph$, any representation $\pi\in \mathrm{Rep}(A,H)$ is an approximately-$S$ representation relative to $H$;
\item \label{rep faith} for any $a\in A\setminus\{0\}$, there is a representation $\sigma\in S$ such that $\sigma(a)\neq 0$. 
\end{enumerate}
\end{theorem}

\begin{proof}
We first show \eqref{rep state} $\Rightarrow$ \eqref{rep cyclic}.  Assume that the collection of $S$-states is weak-$*$ dense in the state space of $A$.  Let $H$ have dimension at least $\aleph$, let $\pi\in \mathrm{Rep}(A,H)$ be a cyclic representation with essential space $H_\pi$, and let $\xi\in H_\pi$ be a cyclic (unit) vector for the corestriction of $\pi$ to $H_\pi$ with associated vector state $\phi(a)=\langle \xi,\pi(a)\xi\rangle$.  If $A$ is not unital, extend $\pi$ to a unital representation of $A^+$ on $H_\pi$.  Let $\epsilon>0$ and finite subsets $A_0$ of $A$ and $H_0$ of $H_\pi$ be given, so we want to find a representation $\sigma$ in the set $U(\pi;A_0,H_0)$ of Definition \ref{rep a h} with $\sigma\in \langle S\rangle$.  Using that $\xi$ is cyclic to approximate vectors in $H_0$ by elements of the form $\pi(a)\xi$, and expanding $A_0$, we may assume that $H_0=\{\xi\}$.  

Now, let $(\phi_i)$ be a net of $S$-states that weak-$*$ converge to $\phi$, and for each $i$ let $(H_i,\pi_i,\xi_i)$ be the GNS representation of $\phi_i$; we also treat each $\pi_i$ as a unital representation of $A^+$ where convenient.  Then 
$$
\langle \pi_i(a)\xi_i,\pi_i(b)\xi_i\rangle =\phi_i(a^*b) \xrightarrow{i \to \infty} \phi(a^*b)= \langle \pi(a)\xi,\pi(b)\xi\rangle
$$
for all $a,b\in A^+$.  Let $H_{A_0}$ be the span of $\{\pi(a)\xi_i\mid a\in A_0\}\cup\{\xi_i\}$.  Arguing as in the proof of \cite[Theorem 2.4]{Exel:1992aa} using \cite[Lemma 2.5]{Exel:1992aa}, for each $i$ there is an isometric inclusion $u_i \colon H_{A_0}\to H_i$ such that 
$$
\|u_i\pi(a)\xi-\pi_i(a)\xi_i\| \xrightarrow{i \to \infty} 0
$$
for all $a\in A_0\cup\{1_{A^+}\}$.  Our assumption on $\aleph$ implies that the dimension of $H$ is at least that of each $H_i$, so we may extend each $u_i$ to a (surjective) coisometry\footnote{A \emph{coisometry} between Hilbert spaces is an operator $v$ such that $vv^*=1$; equivalently, its adjoint is an isometry.} $\widetilde{u_i} \colon H\to H_i$, and define $\sigma_i\in \mathrm{Rep}(A,H)$ by $\sigma_i(a) \coloneqq \widetilde{u}_i^*\pi_i(a)\widetilde{u}_i$.  Then for every $a\in A_0$ using that $\widetilde{u}_i^*\widetilde{u}_i\pi(a)\xi=\pi(a)\xi$ and $u_i\pi(a)\xi=\widetilde{u}_i\pi(a)\xi$ for all $a\in A_0\cup \{1_{A^+}\}$, we have 
\begin{align*}
\|\sigma_i(a)\xi-\pi(a)\xi\| & =\|\widetilde{u}_i^*\pi_i(a)\widetilde{u}_i\xi-\pi(a)\xi\| \\&  \leq \|\widetilde{u}_i^*\pi_i(a)(\widetilde{u}_i\xi-\xi_i)\|+\|\widetilde{u}_i^*\pi_i(a)\xi_i-\pi(a)\xi\| \\
& = \|\widetilde{u}_i^*\pi_i(a)(\widetilde{u}_i\xi-\xi_i)\| +\|\widetilde{u}_i^*\pi_i(a)\xi_i-\widetilde{u}_i^*\widetilde{u}_i\pi(a)\xi\| \\
& \leq \Big(\sup_{a\in A_0\cup\{1_{A^+}\}}\|a\|\|u_i\xi-\xi_i\|\Big)+\|\pi_i(a)\xi_i-u_i\pi(a)\xi\| 
\end{align*}
and this tends to zero as $i\to\infty$.  Our desired representation $\sigma$ can thus be taken to be $\sigma_i$ for $i$ suitably large.

We now show \eqref{rep cyclic} $\Rightarrow$ \eqref{rep gen}.  Let $\pi$ be an element of $\mathrm{Rep}(A,H)$.  We may write $\pi$ as a direct sum of $\bigoplus_{i\in I}\pi_i$ with $(\pi_i)_i$ a collection of mutually orthogonal essentially cyclic representations in $\mathrm{Rep}(A,H)$.  As $\pi=\lim_{F}\bigoplus_{i\in F}\pi_i$ where the limit is taken over finite subsets of $I$, we may assume that $\pi=\bigoplus_{i\in F}\pi_i$ is a sum of finitely many essentially cyclic representations.  As the dimension of $H$ is strictly larger than that of all cyclic representations of $A$, we may write $H$ as a direct sum $H=\bigoplus_{i\in F}H_i$, with each $H_i$ of the same (infinite) dimension as $H$, and each $\pi_i\in \mathrm{Rep}(A,H_i)$.  Using part \eqref{rep cyclic} applied to each $H_i$ and taking direct sums of the resulting approximants, we are done.

To see that \eqref{rep gen} $\Rightarrow$ \eqref{rep faith}, let $H$ be a Hilbert space of dimension at least $\aleph$ such that there is a faithful representation $\pi\in \mathrm{Rep}(A,H)$ (for example, a suitably large amplification of a given faithful representation will work).  Then part \eqref{rep gen} writes $\pi$ as a limit of a net $(\pi_i)_i$.  For each $a\in A\setminus \{0\}$, as $\pi(a)\neq 0$, there must exist $i$ with $\pi_i(a)\neq 0$.  As $\pi$ is (unitarily equivalent to) a subrepresentation of a direct sum of representations from $S$, there must exist a representation $\sigma\in S$ with $\sigma(a)\neq 0$.  

Finally, we show that \eqref{rep faith} $\Rightarrow$ \eqref{rep state}.  Indeed, note first that as $\langle S\rangle$ is closed under finite direct sums, and taking subrepresentations, the collection of $S$-states is convex.  The result now follows from a standard Hahn--Banach separation argument: see \cite[Theorem 2.4, proof of (e) $\Rightarrow$ (a)]{Exel:1992aa} for details. 
\end{proof}

\begin{definition}\label{fell top}
Let $A$ be a $C^*$-algebra, and $S$ be a collection of representations of $A$.  We say that $S$ is \emph{dense in the Fell topology} if it satisfies the equivalent conditions in Theorem \ref{rep approx}.

If $S$ is a collection of unitary representations of a group $\Gamma$, we say that $S$ is \emph{dense in the Fell topology} if it is dense in the Fell topology when considered as a collection of representations of $C^*\Gamma$.
\end{definition}

We will not need the Fell topology as such, just what it means to be dense in it, as in Definition \ref{fell top}: see also Remark \ref{fell top rem} for references on the Fell topology.

\subsection{Definition of property FD}

The next definition is due to Lubotzky--Shalom \cite{Lubotzky:2004xw}.

\begin{definition}\label{fd def}
A group $\Gamma$ has \emph{property FD}, or just \emph{is FD}, if the collection $S$ of representations of $\Gamma$ that factor through a finite quotient is dense in the Fell topology, in the sense of Definition \ref{fell top}.

More generally, if $\Lambda$ is a group acting via a homomorphism $\alpha \colon \Lambda\to \mathrm{Aut}(\Gamma)$ on $\Gamma$, we say that $\Gamma$ has \emph{FD relative to $\alpha$} (or relative to $\Lambda$, if the action $\alpha$ is clear from the context) if the set $S$ of representations $\pi$ of $\Gamma$ that factor through a finite quotient, and are such that the collection of representations $\{\pi\circ \alpha_l \mid l\in \Lambda\}$ is finite up to unitary equivalence, is dense in the Fell topology.
\end{definition}

\begin{remark}
\label{relative FD fg}
    If $\Gamma$ is finitely generated, then it has only finitely many homomorphisms to a given finite group, because there are only finitely many choices for the images of the generators. Therefore if $\pi$ is a representation of $\Gamma$ that factors through a finite quotient then $\{ \pi \circ \alpha \mid \alpha \in \mathrm{Aut}(\Gamma) \}$ is finite. Hence, if a finitely generated group is FD, then it is also FD relative to any action by automorphisms of another group $\Lambda$.
\end{remark}

Here is a related property that is a useful weakening of property FD.

\begin{definition}\label{rfd def}
    A group $\Gamma$ is \emph{residually finite-dimensional} (RFD)\footnote{One can analogously define RFD $C^*$-algebras: this is a well-studied class in abstract $C^*$-algebra theory.  There is no analog of property FD for general $C^*$-algebras.} if the collection of finite-dimensional representations is dense in the Fell topology.
\end{definition}

As any representation that factors through a finite quotient decomposes as a direct sum of finite-dimensional representations, this is a weaker property than FD.

\begin{remark}
\label{RFD:MAP:FD:RF}

Recall that a group is \emph{maximally almost periodic} (MAP) if it has a separating family of homomorphisms to finite-dimensional unitary groups\footnote{Using the Peter-Weyl theorem (see for example \cite[Corollary 15.1.6]{Dixmier:1977vl} for an appropriate statement), this is the same as having a separating family of homomorphisms to compact groups.}. 

In particular, if $\Gamma$ is RFD, then it is MAP; and if it is FD, then it is residually finite.  If $\Gamma$ is amenable, the converses to both of these statements hold (see \cite{bekka:louvet} for RFD, and \cite[Lemma 2.4]{Lubotzky:2004xw} for FD), but not in general: for example $SL(3, \mathbb{Z})$ is residually finite but not RFD \cite{bekka:congruence} (and so also not FD).

If $\Gamma$ is finitely generated and MAP, then it must be residually finite  \cite{malcev}. There are (non finitely generated) amenable groups that are MAP but have no finite quotients, such as $\Q$; such groups are therefore RFD but not FD.  We do not know any finitely generated groups that are RFD, but not FD (Question \ref{q:rfdnotfd}).
\end{remark}

\begin{remark}
\label{rem:fd:subgroups}
    Using for example Item \eqref{rep faith} in Theorem \ref{rep approx}, it is immediate that property (R)FD passes to subgroups.
\end{remark}

\subsection{Known results}

The fundamental results on property FD were proved by Lubotzky--Shalom in their original paper \cite{Lubotzky:2004xw}. The following is \cite[Theorem 2.2]{Lubotzky:2004xw}: note however that the original proof has a (small) gap \cite[Section 5]{Nowak:2011aa}.

\begin{theorem}[Lubotzky--Shalom]
\label{LS:free}

Free groups\footnote{The statement of \cite[Theorem 2.2]{Lubotzky:2004xw} is for finitely generated free groups.  However, because FD passes to subgroups (Remark \ref{rem:fd:subgroups}, it extends to countable free groups, as the authors of \cite{Lubotzky:2004xw} point out. From there it is straightforward to see that it is true regardless of cardinality.} have property FD. \qed
\end{theorem}

Lubotzky--Shalom also prove an important permanence property. We need some terminology first.

\begin{definition}
\label{def:efficient}
    A subgroup $\Lambda$ of a (residually finite) group $\Gamma$ is \emph{separable} if it is an intersection of finite-index subgroups of $\Gamma$. It is \emph{efficient} if every finite-index subgroup of $\Lambda$ is separable.
\end{definition}

Recall also that $\Lambda < \Gamma$ is \emph{co-amenable} if $\ell^2(\Gamma/\Lambda)$ weakly contains the trivial representation, see for example \cite{monod:popa}.

\begin{theorem}[Lubotzky--Shalom]
\label{LS:coamenable}

    Let $\Lambda < \Gamma$ be a co-amenable subgroup. Suppose that $\Lambda$ can be written as a directed union of subgroups $(\Lambda_i)_{i \in I}$, each of which is FD and efficient in $\Gamma$. Then $\Gamma$ is FD.
\end{theorem}

\begin{proof}
    This is essentially \cite[Corollary 2.5]{Lubotzky:2004xw}, with two small differences. First, we drop the assumption that $\Lambda$ is normal in $\Gamma$, which does not change the proof as mentioned in \cite[Remark 2.6]{Lubotzky:2004xw}. Second, the authors of \cite{Lubotzky:2004xw} require the property for every finitely generated subgroup of $\Lambda$, however in the proof this is applied only to a directed set of subgroups exhausting $\Lambda$, so our assumption is sufficient.
\end{proof}

We immediately deduce:

\begin{corollary}
\label{virtually fd}
    Virtually FD groups are FD. \qed
\end{corollary}

\begin{remark}
\label{virtually rfd}
    The analogous result for RFD also holds \cite[Lemma 1]{bekka:congruence}.
\end{remark}

From Theorem \ref{LS:coamenable}, Lubotzky--Shalom also deduce \cite[Theorem 2.8]{Lubotzky:2004xw}:

\begin{corollary}[Lubotzky--Shalom]
\label{LS:surface}

    Fundamental groups of compact surfaces are FD. \qed
\end{corollary}

A more recent permanence result is due to Shulman--Skalski \cite[Proposition 4.5 and Theorem 5.2]{Shulman:2023aa}.

\begin{theorem}[Shulman--Skalski]
\label{SS:semidirect}
    Let $\Gamma$ be be a group and let $\Lambda$ be an amenable group acting on $\Gamma$. Then $\Gamma \rtimes \Lambda$ is FD if and only if $\Gamma$ is FD relative to $\Lambda$ and $\Lambda$ is residually finite.
    
    In particular, if $\Gamma$ is moreover finitely generated, then $\Gamma \rtimes \Lambda$ is FD if and only if $\Lambda$ is residually finite. \qed
\end{theorem}

In particular, a group of the form $\Gamma \rtimes \Z$ is FD, if $\Gamma$ is either a free group (Theorem \ref{LS:free}) or a surface group (Corollary \ref{LS:surface}), which recovers \cite[Theorem 2.8]{Lubotzky:2004xw}.

\subsection{Free products}

An important permanence property for property RFD is that it passes to free products \cite[Theorem 3.2]{Exel:1992aa}. In this subsection, we prove the same fact for property FD, inspired by the work of Exel--Loring in \cite{Exel:1992aa}. In fact, we prove a relative version of the result.

\begin{theorem}\label{fd fp}
Let $\Gamma_1$ and $\Gamma_2$ be groups, and let $\Lambda$ be a group acting on both $\Gamma_1$ and $\Gamma_2$ by automorphisms.  Let $\Lambda$ act on the free product $\Gamma_1\Asterisk\Gamma_2$ via the induced action.  Then if $\Gamma_1$ and $\Gamma_2$ both have FD relative to $\Lambda$, so does $\Gamma_1\Asterisk\Gamma_2$.  
\end{theorem}

We record the following immediate corollary (it is the case $\Lambda$ is trivial) for ease of reference.

\begin{corollary}\label{fd fp cor}
Property FD is preserved under free products.  \qed
\end{corollary}

Let us also see how the relative version, combined with the work of Shulman--Skalski \cite{Shulman:2023aa}, gives a stronger permanence property, analogous to Corollary \ref{fp r llp} for the LLP.

\begin{corollary}\label{fd am}
Let $\Gamma_1$ and $\Gamma_2$ be FD groups, with a common amenable retract $\Lambda$.  Then $\Gamma_1\Asterisk_\Lambda \Gamma_2$ is FD.
\end{corollary}

\begin{proof}
    Write $\Gamma_i = N_i \rtimes \Lambda$. By comparing presentations, we see that
    \[\Gamma_1 \Asterisk_{\Lambda} \Gamma_2 \cong (N_1 \Asterisk N_2) \rtimes \Lambda.\]
    Because $\Gamma_i$ is FD, one direction of Theorem \ref{SS:semidirect} implies that $N_i$ is FD relative to $\Lambda$. Using Theorem \ref{fd fp}, the same is true for $N_1\Asterisk N_2$, hence the other direction of Theorem \ref{SS:semidirect} implies that $(N_1 \Asterisk N_2) \rtimes \Lambda$ is FD.
\end{proof}

\begin{remark}\label{f2 f2 not rfd}
The product $F_2\times F_2$ is a free product of $F_2\times \Z$ with itself, amalgamated along a common copy of $F_2$, which is a retract.  Due to the negative solution of the Connes embedding problem \cite{Salle:2023aa}, $F_2\times F_2$ is not RFD \cite{Ozawa:2004ab} (and so also not FD).  This shows that the amenability assumption is necessary in Corollary \ref{fd am}.
\end{remark}

\begin{remark}
Corollary \ref{fd am} also holds with RFD in place of FD.  Indeed, one can define RFD relative to the action of a group of automorphisms in the same way as for FD.  The analog of Theorem \ref{SS:semidirect} (without the finitely generated part) holds using \cite[Corollary 4.1]{Shulman:2023aa} in place of \cite[Proposition 4.5]{Shulman:2023aa} (that is, Theorem \ref{SS:semidirect}).  The proof of \cite[Theorem 3.2]{Exel:1992aa} then goes through almost  verbatim for RFD relative to a group of automorphisms giving the analog of Theorem \ref{fd fp} for RFD relative to a group of automorphisms, and we proceed as before.
\end{remark}

The rest of this subsection is devoted to the proof of Theorem \ref{fd fp}. We need two lemmas.

\begin{lemma}\label{fd fo}
Let $\Lambda$ be a group acting via a homomorphism $\alpha \colon \Lambda\to \mathrm{Aut}(\Gamma)$ on a group $\Gamma$.  Let $\pi \colon \Gamma\to \mathcal{U}(H)$ be a representation of $\Gamma$ with finite-index kernel $K\leq \Gamma$.  Then the set of representations $\{\pi\circ \alpha_l\mid l\in \Lambda\}$ is finite up to unitary equivalence if and only if the set $\{\alpha_l(K)\mid l\in \Lambda\}$ of finite-index subgroups of $\Gamma$ is finite.
\end{lemma}

\begin{proof}
Let $\{\pi\circ \alpha_l\mid l\in \Lambda\}/\sim$ be the collection of unitary equivalence classes of representations in the given set.  As representations with different kernels are inequivalent, there is a well-defined $\Lambda$-equivariant surjection 
$$
p \colon \left( \{\pi\circ \alpha_l\mid l\in \Lambda\}/\sim \right) \to \{\alpha_l(K)\mid l\in \Lambda\}.
$$
Let $\Lambda_K \coloneqq \{l\in \Lambda\mid \alpha_l(K)=K\}$.  Then there is a well-defined homomorphism $\Lambda_K\to \mathrm{Aut}(\Gamma/K)$ to a finite group.  It follows that $p^{-1}(K)$ consists of $\{\pi\circ \alpha_{[l]}\mid [l]\in \mathrm{Aut}(\Gamma/K)\}/\sim$, and is in particular finite.  By equivariance, the preimage of any point under $p$ is finite, which completes the proof.
\end{proof}

The following lemma is straightforward, and left to the reader (compare \cite[Lemma 3.1]{Exel:1992aa}).  

\begin{lemma}\label{up con}
Let $A$ be a $C^*$-algebra, let $H$ be a Hilbert space, let $\pi\in \mathrm{Rep}(A,H)$ be a representation, and let $(\pi_i)_i$ be a net of representations in $\mathrm{Rep}(A,H)$ that converges to $\pi$.  

Assume that for each $i$, we have a representation $\rho_i\in \mathrm{Rep}(A,H)$ such that $\rho_i$ agrees with $\pi_i$ on the essential space of $\pi_i$.  Then $(\rho_i)_i$ also converges to $\pi$. \qed
\end{lemma}

\begin{proof}[Proof of Theorem \ref{fd fp}]
Fix $a\in C^*(\Gamma_1\Asterisk\Gamma_2)\setminus \{0\}$.  We aim to find a representation $\sigma$ of $C^*(\Gamma_1\Asterisk\Gamma_2)$ that factors through a finite quotient, with finite $\Lambda$-orbit, and with $\sigma(a)\neq 0$ (see Item \eqref{rep faith} of Theorem \ref{rep approx}).

Let $\aleph$ be a cardinal strictly larger than the dimensions of all cyclic representations of $C^*(\Gamma_1\Asterisk\Gamma_2)$, and let $\pi$ be a nondegenerate faithful representation of  $C^*(\Gamma_1\Asterisk\Gamma_2)$ on a Hilbert space $H$ of dimension at least $\aleph$.   Item \eqref{rep gen} of Theorem \ref{rep approx} and property FD relative to $\Lambda$ for each $\Gamma_j$ gives a net $(\pi_i^{(j)})_i$ of representations of $\Gamma_j$ on $H$ with the following properties:
\begin{enumerate}[(a)]
\item the essential space $H_i^{(j)}$ of each $\pi_i^{(j)}$ is finite-dimensional;
\item \label{flo} the $\Lambda$-orbit of each $\pi_i^{(j)}$ is finite, up to unitary equivalence;
\item each $\pi_i^{(j)}$ factors through a finite quotient of $\Gamma_j$;
\item the net $(\pi_i^{(j)})_i$ converges to $\pi|_{\Gamma_j}$.
\end{enumerate}
We may assume that the directed sets indexing the nets $(\pi_i^{(1)})_i$ and $(\pi_i^{(2)})_i$ are the same by, for example, taking a product.

Now, fix $i$.  We claim that there exists a finite-dimensional subspace $H_i$ of $H$ and representations $\rho_i^{(j)}\in \mathrm{Rep}(C^*\Gamma_j,H)$ with the following properties:
\begin{enumerate}[(i)]
\item $\rho_i^{(1)}$ and $\rho_i^{(2)}$ both have $H_i$ as essential space;
\item \label{cont hi} both $H_{i}^{(1)}$ and $H_i^{(2)}$ are contained in $H_i$, and $\rho_i^{(j)}$ restricts to $\pi_i^{(j)}$ on $H_i^{(j)}$;
\item \label{f mult} $\rho_i^{(j)}$ is unitarily equivalent to a finite multiple of $\pi_i^{(j)}$.
\end{enumerate}

Indeed, take $H_i$ to be any subspace of $H$ containing $H_i^{(1)}$ and $H_i^{(2)}$ and with dimension a multiple of both the dimensions of $H_{i}^{(1)}$ and $H_{i}^{(2)}$.  Then define $\rho_i^{(j)}$ to be any representation on $H_i$ that is unitarily equivalent to a multiple of $\pi_i^{(j)}$, and that agrees with $\pi_i^{(j)}$ on $H_{i}^{(j)}$.  Note that by Lemma \ref{up con} and point \eqref{cont hi} above, we have that $(\rho_i^{(j)})$ converges to $\pi|_{\Gamma_j}$ in $\mathrm{Rep}(C^*\Gamma_j,H)$.

Now, by the universal property of the free product, there exists a unique representation $\rho_i\in \mathrm{Rep}(C^*(\Gamma_1\Asterisk\Gamma_2),H_i)$ that restricts to $\rho_i^{(j)}$ on $C^*\Gamma_j$.  As $\rho_i^{(j)} \xrightarrow{i \to \infty} \pi|_{\Gamma_j}$, one checks directly that $\rho_i \xrightarrow{i \to \infty} \pi$.  As $\pi$ is faithful, for our original element $a\in C^*(\Gamma_1\Asterisk\Gamma_2)\setminus \{0\}$, there exists $i$ such that $\rho_i(a)\neq 0$.  Fix this $i$ for the rest of the proof.

For $j\in \{1,2\}$, write $K_j$ for the kernel of the representation $\pi_i^{(j)}$ (equivalently, by point \eqref{f mult}, of $\rho_i^{(j)}$), a finite-index subgroup of $\Gamma_j$.  Define $F_j \coloneqq \Gamma_j/K_j$, a finite group, and note that $\rho_i$ factors through $C^*(F_1\Asterisk F_2)$.  Hence if $\phi \colon C^*(\Gamma_1\Asterisk\Gamma_2)\to C^*(F_1\Asterisk F_2)$ is the canonical quotient map, then $\phi(a)\neq 0$.  As $F_1\Asterisk F_2$ is virtually free, it has property FD by Theorem \ref{LS:free} and Corollary \ref{virtually fd}.  Hence, by Item \eqref{rep faith} of Theorem \ref{rep approx}, there exists a representation $\tau$ of $C^*(F_1\Asterisk F_2)$ that factors through a finite quotient $Q$, and with $\tau(\phi(a))\neq 0$.

It remains to show that if we consider $\sigma \coloneqq \tau\circ \phi$ as a representation of $C^*(\Gamma_1\Asterisk \Gamma_2)$, then the $\Lambda$-orbit of $\sigma$ is finite. Let $K \coloneqq \ker(\sigma)$, which is a finite-index subgroup of $\Gamma_1 \Asterisk \Gamma_2$; by Lemma \ref{fd fo}, we need to show that the set
\[\{\alpha_l(K) \mid l \in \Lambda\}\]
of finite-index subgroups of $\Gamma_1 \Asterisk \Gamma_2$ is finite. Now $\alpha_l(K)$ is the kernel of a homomorphism $\Gamma_1 \Asterisk \Gamma_2 \to Q$ factoring through $\Gamma_1/\alpha_l^{(1)}(K_1) \Asterisk \Gamma_2/\alpha_l^{(2)}(K_2)$. Using point \eqref{flo} above, and the other direction of Lemma \ref{fd fo}, the orbits
\[\{\alpha^{(j)}_l(K_j)\mid l\in \Lambda\}\]
are finite. Since moreover each of the groups $\Gamma_1/\alpha_l^{(1)}(K_1) \Asterisk \Gamma_2/\alpha_l^{(2)}(K_2)$ is finitely generated, hence admits only finitely many homomorphisms to $Q$ by Remark \ref{relative FD fg}, we conclude.
\end{proof}

\section{New examples of FD groups}\label{s:ex fd}

In this section, we show that many of the examples from Section \ref{s:ex} also satisfy property FD.

\subsection{Low-dimensional manifolds}

Our most significant new result about property FD is the following.

\begin{theorem}\label{manifold fd}
    Let $M$ be a connected manifold of dimension at most $3$. If $\pi_1(M)$ is finitely generated, then it is FD.
\end{theorem}

Note that by Scott's core theorem \cite{scott:core}, finitely generated fundamental groups of $3$-manifolds are equivalently fundamental groups of compact $3$-manifolds.

The proof of Theorem \ref{manifold fd} will require some deeper theory of $3$-manifold groups, we refer to \cite{3manifold} for a general overview. As in Example \ref{manifold lp}, we may assume that $M$ is a $3$-manifold. Our proof follows Kielak--Linton's proof that $3$-manifold groups are in Linnel's class $\mathcal{C}$ \cite{kielak:linton:manifolds} (which we used in Example \ref{manifold lp}) to reduce to the case of a closed aspherical manifold, and then the arguments of Friedl--L{\"u}ck \cite[Theorem 3.2(3)]{friedl:luck} for the closed aspherical case. This almost entirely goes through, thanks in particular to our Corollary \ref{fd fp cor}, with the exception of \emph{closed graph manifolds}, which will need a more involved argument.

\begin{proposition}\label{vfbc fd}
    Finitely generated virtually free-by-cyclic groups are FD.
\end{proposition}

\begin{proof}
    Linton shows in \cite[Theorem 1.2]{Linton:2025ac} that a finitely generated group of the form $F\rtimes \Z$ with $F$ countable and free embeds in a group of the form $F_0\rtimes \Z$, where now $F_0$ is a finitely generated free group.  The group $F_0\rtimes \Z$ is FD by Theorem \ref{SS:semidirect} (or \cite[Theorem 2.8]{Lubotzky:2004xw}).  Hence $F\rtimes \Z$ is also FD, by Remark \ref{rem:fd:subgroups}.  Hence a virtually free-by-cyclic group is FD by Corollary \ref{virtually fd}.
\end{proof}

As we mentioned after Theorem \ref{SS:semidirect}, Lubotzky--Shalom proved that a group of the form $\pi_1(\Sigma) \rtimes \Z$ is FD, when $\Sigma$ is a compact surface \cite[Theorem 2.8]{Lubotzky:2004xw}. Moreover, by Corollary \ref{virtually fd}, property FD passes to finite-index overgroups. Together with the deep \emph{virtual fibering theorems} \cite{agol:criteria, Agol:2013aa, liu:graph, PW:graph, PW:mixed, Wise:2021aa} this implies property FD in many cases.

\begin{theorem}[{\cite[(G.25)]{3manifold}}]\label{nongraph fd}
    If $M$ is a compact orientable irreducible $3$-manifold with empty or toroidal boundary, and $M$ is \emph{not} a closed graph manifold, then $\pi_1(M)$ is FD. \qed
\end{theorem}

Our new input is Corollary \ref{fd fp}, and the case of closed graph manifolds.

\begin{proposition}\label{graph fd}
    Let $M$ be a closed graph manifold. Then $\pi_1(M)$ is FD.
\end{proposition}

Let us assume this for now, and see how these ingredients come together to prove Theorem \ref{manifold fd}.

\begin{proof}[Proof of Theorem \ref{manifold fd}]
    We may reduce to the case of $3$-manifolds by taking products with the circle and using that FD passes to subgroups.
    Let $M$ be a connected $3$-manifold such that $\pi_1(M)$ is finitely generated. By \cite[Proposition 2.2]{kielak:linton:manifolds}, there exists a finitely generated free group $F$ and compact, aspherical $3$-manifolds $M_1, \ldots, M_n$, each with a (possibly trivial) incompressible boundary, such that a finite-index subgroup of $\pi_1(M)$ is isomorphic to
    \[F \Asterisk \left( \Asterisk_{i = 1}^n \pi_1(M_i) \right).\]
    By Theorem \ref{LS:free}, Corollary \ref{virtually fd} and Corollary \ref{fd fp}, it remains to show that each $\pi_1(M_i)$ is FD. So from now on we may assume that $M$ is compact, aspherical, and has incompressible boundary.

    If $M$ is closed, then it is FD by Theorem \ref{nongraph fd} or Proposition \ref{graph fd}. If $M$ has non-empty boundary, then $M$ is homotopy equivalent to its spine (see for example \cite[page 560 and Section IV]{Casler:1965aa}), which is an aspherical $2$-complex, hence $cd_{\mathbb{Q}}(\pi_1(M)) < 3$. \cite[Theorem 1.1]{kielak:linton:manifolds} implies that $\pi_1(M)$ is virtually free-by-cyclic, hence FD by Proposition \ref{vfbc fd}.
\end{proof}

It remains to prove Proposition \ref{graph fd}. This will require some more involved machinery from Bass--Serre theory and $3$-manifold topology, so we will leave it to the next subsection, which is independent of the rest of the paper.

\subsection*{Closed graph manifolds}

\subsubsection*{First easy cases}

We refer the reader to \cite[Section 1.5]{3manifold} for a precise definition of \emph{Seifert fibered $3$-manifold}. For our purposes, it suffices to know that if $M$ is a compact Seifert fibered manifold, then there is a short exact sequence
\[1 \to Z \to \pi_1(M) \to Q \to 1,\]
where $Z$ is cyclic and $Q$ has a finite-index subgroup isomorphic to the fundamental group of a compact surface, see for example \cite[Theorem 2.2]{martino:seifert}.

\begin{lemma}\label{graph fd easy}
    If $M$ is a compact Seifert fibered manifold, or if $M$ is finitely covered by a torus bundle, then $\pi_1(M)$ is FD.
\end{lemma}

\begin{proof}
    Suppose that $M$ is Seifert fibered. By Corollary \ref{virtually fd}, using that the automorphism group of a cyclic group is finite, we may pass to a finite-index subgroup of $Q$ that is isomorphic to the fundamental group of a compact surface, and acts trivially on $Z$ by conjugacy. We then obtain a group $\Gamma$ that fits into a central extension
    \[1 \to Z \to \Gamma \to \pi_1(\Sigma) \to 1,\]
    where $\Sigma$ is a compact surface.
    
    Let $\pi_1(\Sigma) \to \Z$ be an epimorphism with kernel $F$, which is free by Remark \ref{surfaces rem}. The pullback defines a homomorphism $\Gamma \to \Z$ whose kernel is isomorphic to $Z \times F$. Now $Z \times F$ is FD, by Theorems \ref{LS:free} and \ref{SS:semidirect}. Moreover, every finitely generated subgroup of $\pi_1(M)$, hence of $\Gamma$, is separable \cite[Corollary 5.1]{seifert:lerf}. Theorem \ref{LS:coamenable} applies and shows that $\Gamma$, hence $\pi_1(M)$ is FD.

    If $M$ is finitely covered by a torus bundle, then $\pi_1(M)$ is amenable. Moreover $\pi_1(M)$ is residually finite \cite{hempel} hence FD by Remark \ref{RFD:MAP:FD:RF}.
\end{proof}

A crucial property of Seifert fibered manifolds that we used here is that every finitely generated subgroup of the fundamental group is separable. This does not hold for general graph manifolds (see for example \cite{niblo:wise}), so we need to work harder to be in a position to apply Theorem \ref{LS:coamenable}.

\subsubsection*{Efficient graphs of groups}

General graph manifolds are fundamental groups of graphs of groups whose vertex groups are Seifert fibered. Before moving on to the general case, we need to establish some general profinite properties of graphs of groups.

Recall that the \emph{profinite topology} on a group $\Gamma$ is the unique topology for which finite-index subgroups form a basis of neighborhoods of the identity. Therefore a family of finite-index subgroups $\{N_i\}_{i \in I}$ of $\Gamma$ forms a basis of neighborhoods for the profinite topology if every finite-index subgroup of $\Gamma$ contains some $N_i$.

Rephrasing Definition \ref{def:efficient}, we see that $\Lambda < \Gamma$ is separable if and only if it is closed in the profinite topology, and it is efficient if and only if, moreover, the profinite topology on $\Gamma$ induces the profinite topology on $\Lambda$.
Our notion of efficiency for subgroups is not standard, but it is inspired by an established notion of efficiency for graphs of groups. We use Wilkes' definition \cite[Definition 8.2.1]{wilkes}. We identify a graph $X$ with vertex set $V$ and edge set $E$ with the set $V \sqcup E$, so whenever we write $\Gamma_x : x \in X$, we are talking at once about vertex groups and edge groups. Besides that, we use the notation established in Subsection \ref{ss:gog}.

\begin{definition}\label{def:efficient gog}
    Let $\mathcal{G}$ be a finite graph of groups with underlying graph $X$, and vertex and edge groups $\Gamma_x : x \in X$. We say that $\mathcal{G}$ is \emph{(fully) efficient} if there is an inverse system of finite-index normal subgroups $N_{x, i} \lhd \Gamma_x : x \in X, i \in I$, such that:
    \begin{enumerate}[(i)]
        \item\label{eff profinite} $\{ N_{x, i} \}_{i \in I}$ forms a basis of neighborhoods of the identity for the profinite topology on $\Gamma_x$, for all $x \in X$;
        \item\label{eff inclusion} $\omega_e^{-1}(N_{\omega(e), i}) = N_{e, i}$ for all $e \in E, i \in I$;
        \item\label{eff separate} $\bigcap_{i \in I} N_{\omega(e), i} \omega_e(\Gamma_e) = \omega_e(\Gamma_e)$.
    \end{enumerate}
\end{definition}

Recall that we follow Serre's convention \cite{Serre:1980aa}, so every simplicial edge appears twice with both possible orientations: this is why in Items \eqref{eff inclusion} and \eqref{eff separate} we only give a condition on $\omega_e$, the corresponding condition on $\alpha_e$ follows by considering the opposite edge.

A more global characterization of efficiency (which is used as the definition in \cite{graphmanifolds:profinite}) is \cite[Theorem 8.2.6]{wilkes}.

\begin{theorem}[Wilkes]\label{thm:efficient gog}
    A graph of groups $\mathcal{G}$ is efficient if and only if its fundamental group $\Gamma$ is residually finite and the subgroups $\Gamma_x : x \in X$ are efficient.
\end{theorem}

Our goal is to prove the following technical result.

\begin{lemma}
\label{efficient subgraph}
    Let $\mathcal{G}$ be an efficient graph of groups with underlying graph $X$. Let $Y$ be a connected subgraph of $X$, and let $\mathcal{H}$ be the subgraph of groups defined on $Y$. Then $\pi_1(\mathcal{H})$ is an efficient subgroup of $\pi_1(\mathcal{G})$.
\end{lemma}

We will need the following refinement of Theorem \ref{thm:efficient gog}. This is essentially \cite[Theorem 8.2.4]{wilkes}, with some additional properties that are not contained in the statement but follow from its proof.

\begin{theorem}[Wilkes]\label{thm:efficient gog precise}
    Let $\mathcal{G}$ be an efficient graph of groups with underlying graph $X$, vertex and edge groups $\Gamma_x : x \in X$ and fundamental group $\Gamma$. Let $\{N_{x, i} : x \in X, i \in I \}$ witness the efficiency of $\mathcal{G}$. Then there exists an inverse system  of finite-index normal subgroups $\{ M_j \lhd \Gamma\}_{j \in J}$ of $\Gamma$ with the following properties.
    \begin{enumerate}[(a)]
        \item\label{Eff profinite} $\{M_j\}_{j \in J}$ forms a basis of neighborhoods of the identity for the profinite topology on $\Gamma$;
        \item\label{Eff index} There is an order-preserving surjection $\kappa \colon J \to I$ such that $\Gamma_x \cap M_j = N_{x, \kappa(j)}$ for all $x \in X, j \in J$;
        \item\label{Eff separate} $\bigcap_{j \in J} M_j \Gamma_x = \Gamma_x$ for all $x \in X$.
    \end{enumerate}
\end{theorem}

\begin{proof}
    For each $i \in I$, \cite[Lemma 8.2.3]{wilkes} defines a graph of groups $\mathcal{G}^i$, with the same underlying graph $X$, and groups $\Gamma^i_x \coloneqq \Gamma_x/N_{x, i}$. The definition of efficiency implies that the edge inclusions from $\mathcal{G}$ descend to edge inclusions on $\mathcal{G}^i$, so that this is indeed a well-defined graph of groups. We denote by $\Gamma^i \coloneqq \pi_1(\mathcal{G}^i)$, which is the fundamental group of a finite graph of finite groups, hence it is virtually free and in particular residually finite. 
    
    There is an associated quotient map $\pi^i \colon \Gamma \to \Gamma^i$. For each finite quotient $\phi \colon \Gamma^i \twoheadrightarrow Q$, we define $M_{i, \phi} \coloneqq \ker(\phi \circ \pi^i)$. Therefore the index set is
    \[J \coloneqq \{ (i, \phi) : i \in I, \phi \colon \Gamma^i \twoheadrightarrow Q, |Q|< \infty \},\]
    whose order is induced by inclusion of subgroups; in the proof of \cite[Theorem 8.2.4]{wilkes} it is shown that $J$ is indeed directed. The map $\kappa \colon J \to I$ defined by $\kappa(i, \phi) = i$ is then an order-preserving surjection, and $M_{i, \phi} \cap \Gamma_x = N_{x, i}$, this establishes Item \eqref{Eff index}. Item \eqref{Eff separate} is \cite[Theorem 8.2.4, Item (iii)]{wilkes}. Moreover, \cite[Theorem 8.2.4, Item (i)]{wilkes} gives that the $M_j$ intersect trivially, but here we need the stronger fact that $\{ M_j \}_{j \in J}$ forms a basis of neighborhoods of the identity for the profinite topology on $\Gamma$ (under the stronger assumption of (full) efficiency of $\mathcal{G}$).

    Let $K < \Gamma$ be a finite-index subgroup, we need to show that there exists $j \in J$ such that $M_j < K$; up to passing to a deeper finite-index normal subgroup, we may assume that $K$ is normal. Let $K_x \coloneqq \Gamma_x \cap K$, which is a finite-index normal subgroup of $\Gamma_x$. Because $\{ N_{x, i} \}_{i \in I}$ forms a basis of neighborhoods of the identity for the profinite topology on $\Gamma_x$, and $X$ is finite, there exists $i \in I$ such that $N_{x, i} < K_x$ for every $x \in X$. Then $\ker(\pi^i) < K$, which means that the quotient $\Gamma \to \Gamma/K$ factors through $\Gamma^i$. Let $\phi \colon \Gamma^i \to \Gamma/K$ denote the induced quotient, then $M_{i, \phi} < K$, as desired.
\end{proof}

\begin{proof}[Proof of Lemma \ref{efficient subgraph}]
    Let $\Gamma \coloneqq \pi_1(\mathcal{G})$. Because $Y$ is connected, we may choose a spanning tree for $X$ that restricts to a spanning tree for $Y$. This exhibits $\Gamma$ as the fundamental group of a graph of groups $\mathcal{K}$ described as follows. The underlying graph, denoted by $Z$, is obtained from $X$ by collapsing $Y$ to a vertex $y$. Therefore $V(Z) = V(X) \setminus V(Y) \cup \{ y \}$ and $E(Z) = E(X) \setminus E(Y)$, where every $e \in E(X)$ with $\omega(e) \in V(Y)$ now has $\omega'(e) = y$. The vertex and edge groups are the given $\Gamma_x$ for $x \in X' \setminus \{ y \}$ and $\Gamma_y \coloneqq \pi_1(\mathcal{H})$. For every edge $e$ with $\omega'(e) = y$, we set $\omega'_e \colon \Gamma_e \to \Gamma_y$ to be the composition of $\omega_e \colon \Gamma_e \to \Gamma_{\omega(e)}$ with the inclusion $\Gamma_{\omega(e)} \to \pi_1(\mathcal{H}) = \Gamma_y$. We need to show that $\Lambda$ is an efficient subgroup of $\Gamma$, so by Theorem \ref{thm:efficient gog}, it suffices to show that the graph of groups $\mathcal{K}$ is efficient.

    Let $\{N_{x, i} : i \in I, x \in X \}$ witness efficiency of $\mathcal{G}$. It follows directly from the definition that $\{ N_{x, i} : i \in I, x \in Y \}$ witness efficiency of $\mathcal{H}$. Hence we may apply Theorem \ref{thm:efficient gog precise}, to obtain an inverse system of finite-index normal subgroups $\{ M_j \lhd \Lambda\}_{j \in J}$. We claim that this inverse system of finite-index normal subgroups of $\Gamma_y$, together with the inverse systems $\{ N_{x, \kappa(j)}\}_{j \in J}$ of finite-index normal subgroups of $\Gamma_x : x \in Z \setminus \{ y \}$, witness efficiency of $\mathcal{K}$.

    \eqref{eff profinite}: That $\{M_j\}_{j \in J}$ is a basis of neighborhoods for $\Gamma_y$ is Item \eqref{Eff profinite}. Moreover, $\{ N_{x, i} \}_{i \in I}$ is a basis of neighborhoods for $\Gamma_x : x \in Z \setminus \{ y \}$ by assumption, and since $\kappa$ is surjective $\{ N_{x, \kappa(j)} \}_{j \in J}$ is the same basis of neighborhoods, just with repeated indices.

    \eqref{eff inclusion}: This follows from efficiency of $\mathcal{G}$ whenever $\omega'(e) \neq y$. Now suppose that $\omega'(e) = y$. Using both the fact that $\{N_{x, i} \}$ witness efficiency of $\mathcal{G}$, and the properties of $\{ M_j \}$ from Theorem \ref{thm:efficient gog precise}:
    \[
    (\omega'_e)^{-1}(M_j) = \omega_e^{-1}(M_j \cap \Gamma_{\omega(e)}) = \omega_e^{-1}(N_{\omega(e), \kappa(j)}) = N_{e, \kappa(j)}.
    \]

    \eqref{eff separate}: Again, this follows from efficiency of $\mathcal{G}$ whenever $\omega'(e) \neq y$, so suppose that $\omega'(e) = y$. Because $\omega'_e(\Gamma_e) \subset \Gamma_{\omega(e)}$, we have
    \[M_j \omega'_e(\Gamma_e) \cap \Gamma_{\omega(e)} = (M_j \cap \Gamma_{\omega(e)}) \omega'_e(\Gamma_e) = N_{\omega(e), \kappa(j)} \omega_e(\Gamma_e).\]
    By Item \eqref{Eff separate} we have
    \[\bigcap_{j \in J} M_j \omega'_e(\Gamma_e)  \subset \bigcap_{j \in J} M_j \Gamma_{\omega(e)} = \Gamma_{\omega(e)}.\]
    Finally, using this and efficiency of $\mathcal{G}$:
    \begin{align*}
    \bigcap_{j \in J} M_j \omega'_e(\Gamma_e) &= \bigcap_{j \in J} (M_j \omega'_e(\Gamma_e) \cap \Gamma_{\omega(e)}) \\ &= \bigcap_{j \in J} N_{\omega(e), \kappa(j)} \omega_e(\Gamma_e) = \omega_e(\Gamma_e) = \omega'_e(\Gamma_e).
    \end{align*}

We conclude that $\mathcal{K}$ is an efficient graph of groups, hence by Theorem \ref{thm:efficient gog} the vertex group $\Gamma_y = \pi_1(\mathcal{H})$ is an efficient subgroup of $\Gamma = \pi_1(\mathcal{K}) = \pi_1(\mathcal{G})$.
\end{proof}

\subsubsection*{JSJ decomposition}

We refer the reader to \cite[Section 1.6]{3manifold} for a precise definition of \emph{JSJ decomposition}. A \emph{graph manifold} is a manifold whose JSJ decomposition has only Seifert fibered pieces.

In practice, if $M$ is a closed graph manifold, then it contains a finite embedded collection of disjoint incompressible tori $T_1, \ldots, T_m$, such that if $M_1, \ldots, M_n$ are the components of $M$ cut along the union of the $T_i$, then every $M_i$ is a Seifert fibered manifold, with boundary equal to a union of the $T_i$. This defines a graph of spaces that gives rise to a graph of groups decomposition of $\pi_1(M)$, with vertex groups $\pi_1(M_i)$ and edge groups $\pi_1(T_j)$. The tori $T_1, \ldots, T_m$ are called \emph{JSJ tori} of $M$, and they are unique in a suitable sense \cite[Proposition 1.6.2]{3manifold}.

For us, the key result about this graph of groups is \cite[Theorem A]{graphmanifolds:profinite}.

\begin{theorem}[Wilton--Zalesskii]\label{jsj efficient}
    The graph of groups coming from the JSJ decomposition of a closed graph manifold is efficient.
\end{theorem}

Next we reduce to graphs of groups that are not trees.

\begin{lemma}\label{graph fd separating}
    Let $M$ be a closed graph manifold that is not Seifert fibered and is not finitely covered by a torus bundle. Then there exists a finite cover of $M$ with a non-separating JSJ torus.
\end{lemma}

The torus $T$ is non-separating if its complement is connected. Given the way the graph of groups is constructed, this is the same as saying that the edge supporting $T$ is non-separating in the underlying graph.

\begin{proof}
    This is a step in the argument for largeness of $\pi_1(M)$, see \cite[Theorem 3.2.4]{long:reid}, or \cite[(C.14) and (C.15)]{3manifold} (see also the proof of \cite[Theorem 3.2(3)]{friedl:luck}). In fact, it is more generally true that any separable non-fiber surface lifts to a non-separating one in a finite cover, this is a consequence of Stallings' fibration theorem \cite{stallings:fibering}; JSJ tori are separable \cite{emily:hamilton}, and they are not fibers because we are assuming that $M$ is not finitely covered by a torus bundle. The fact that the lift of a JSJ torus is still a JSJ torus follows from uniqueness \cite[Proposition 1.6.2]{3manifold}.
\end{proof}

From now on we assume that $M$ is a closed graph manifold whose JSJ decomposition includes a non-separating edge, corresponding to the torus $T$. Let $\mathcal{G}$ be the corresponding graph of groups, with underlying graph $X$; the fundamental group $\pi_1(\mathcal{G})$ is $\pi_1(M)$, which as usual we denote by $\Gamma$. Let $e$ be the edge supporting $\pi_1(T)$, which we denote by $\Delta$. As in the proof of \cite[Theorem 3.2(3)]{friedl:luck}, we consider the map $(\pi_1(M) \to \mathbb{Z}) \in H^1(M; \Z)$ that is Poincar{\'e} dual to the second homology class $[T] \in H_2(M; \Z)$, but we instead interpret it in terms of the graph of groups.

Let $N \coloneqq M \setminus T$; this is a compact connected graph manifold with two boundary components, which we denote by $T^-$ and $T^+$, each identified with $T$. The JSJ decomposition of $N$ is the one induced from the JSJ decomposition of $M$, so the corresponding graph of groups $\mathcal{H}$ is obtained from $\mathcal{G}$ by removing the edge $e$. We denote by $\Lambda \coloneqq \pi_1(N)$, and by $\Delta^{\pm}$ the image of $\pi_1(T^{\pm})$ in $\Lambda$. If $v^{\pm}$ are the vertices adjacent to $e$, then $\Delta^{\pm} \to \Lambda$ factors through the vertex group $\Gamma_{v^{\pm}}$.

Choose a spanning tree for $X$ that does not contain the edge $e$: this is possible because $e$ is non-separating. Consider the corresponding presentation for $\Gamma$, built inductively as explained in Remark \ref{gog basic operations}, and let $t \coloneqq t_e$ be the stable letter corresponding to the edge $e$. Therefore $\Gamma$ is an HNN extension
\[\Gamma = \langle \Lambda, t \mid t^{-1} g t = f(g) : g \in \Delta^- \rangle,\]
where $f \colon \Delta^- \to \Delta^+$ is the isomorphism induced by identification with $\pi_1(T)$.
Let
\[K \coloneqq \langle t^{-n} \Lambda t^n : n \in \Z \rangle < \Gamma,\]
which is the kernel of the retraction $\Gamma \to \langle t \rangle$. By \cite[Theorem 2.17.1]{bogopolski}, $K$ is the fundamental group of the following graph of groups:

\begin{center}

\begin{tikzpicture}[scale=0.85, baseline]

\tikzset{
  Hlabel/.style={font=\normalsize},
  Alabel/.style={font=\footnotesize}
}

\draw (-6,0) -- (6,0);

\node at (-6.5,0) {$\cdots$};
\node at (6.5,0) {$\cdots$};

\foreach \x/\label in {
    -5/\Lambda_{-2},
    -2.5/\Lambda_{-1},
    0/\Lambda_{0},
    2.5/\Lambda_{1},
    5/\Lambda_{2}
}{
    \draw (\x,0) circle (0.12);
    \node[Hlabel, above=3pt] at (\x,0) {$\label$};
}

\node[Alabel, below=1pt] at (-3.75,0) {$\Delta^+_{-2} = \Delta^-_{-1}$};
\node[Alabel, below=1pt] at (-1.25,0) {$\Delta^+_{-1} = \Delta^-_0$};
\node[Alabel, below=1pt] at (1.25,0) {$\Delta^+_0 = \Delta^-_1$};
\node[Alabel, below=1pt] at (3.75,0) {$\Delta^+_1 = \Delta^-_2$};
\end{tikzpicture}
\end{center}
where each $\Lambda_i$ is a copy of $\Lambda$ corresponding to the $t^i$-conjugate of $\Lambda = \Lambda_0$, and the amalgamation $\Delta_i^+ = \Delta_{i+1}^-$ is defined by the isomorphism $f$ above, induced by identification with $\pi_1(T)$.

We define $K_p \coloneqq \langle \Lambda_i : 0 \leq i < p \rangle < K$.

\begin{lemma}\label{graph fd final step}
    Each $K_p$ is FD, and an efficient subgroup of $\Gamma$.
\end{lemma}

\begin{proof}
Let us start by describing geometrically the group $K$. Let $\widetilde{M}$ denote the $\Z$-cover of $M$ corresponding to $K$. This has a graph of spaces decomposition corresponding to the graph of groups decomposition of $K$:

\begin{center}
\begin{tikzpicture}[scale=0.85, baseline]

\tikzset{
  Hlabel/.style={font=\normalsize},
  Alabel/.style={font=\footnotesize}
}

\draw (-6,0) -- (6,0);

\node at (-6.5,0) {$\cdots$};
\node at (6.5,0) {$\cdots$};

\foreach \x/\label in {
    -5/N_{-2},
    -2.5/N_{-1},
    0/N_{0},
    2.5/N_{1},
    5/N_{2}
}{
    \draw (\x,0) circle (0.12);
    \node[Hlabel, above=3pt] at (\x,0) {$\label$};
}

\node[Alabel, below=1pt] at (-3.75,0) {$T^+_{-2} = T^-_{-1}$};
\node[Alabel, below=1pt] at (-1.25,0) {$T^+_{-1} = T^-_0$};
\node[Alabel, below=1pt] at (1.25,0) {$T^+_0 = T^-_1$};
\node[Alabel, below=1pt] at (3.75,0) {$T^+_1 = T^-_2$};
\end{tikzpicture}
\end{center}
where each $N_i$ is a copy of $N$ and the gluing $T_i^+ = T_{i+1}^-$ is induced by identification with $T$. The element $t \in \Gamma$ acts as a deck transformation shifting the graph of spaces to the right by one edge.

Now consider the space $\widetilde{M}_p \coloneqq \widetilde{M}/\langle t^p \rangle$. This is a $\Z/p\Z$-cover of $M$, corresponding to the index-$p$ subgroup $K \rtimes \langle t^p \rangle < K \rtimes \langle t \rangle = \Gamma$. Since $t$ acts by automorphisms on the graph of spaces for $\widetilde{M}$, it follows that $\widetilde{M}_p$ inherits a quotient graph of spaces decomposition:

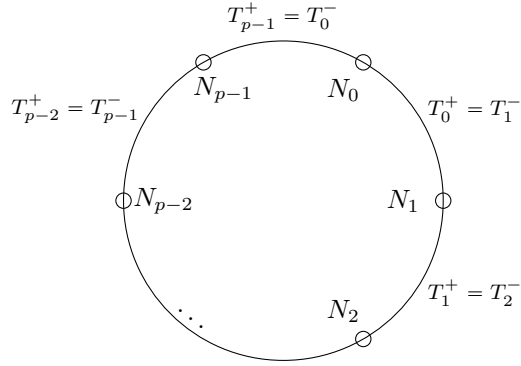
\begin{center}
\begin{tikzpicture}[scale=0.85, baseline]

\tikzset{
  Hlabel/.style={font=\normalsize},
  Alabel/.style={font=\footnotesize}
}

\def\R{2.5} 

\draw (0,0) circle (\R);

\foreach \k/\angle in {
  {p-2}/180,
  {p-1}/120,
   0/60,
   1/0,
   2/-60
}{
  \draw ({\R*cos(\angle)}, {\R*sin(\angle)}) circle (0.12);

  \node[Hlabel] at ({0.75*\R*cos(\angle)}, {0.8*\R*sin(\angle)}) {$N_{\k}$};
}

\node[Hlabel] at ({0.9*\R*cos(-130)}, {0.9*\R*sin(-130)}) {$\ddots$};

  \node[Alabel] at ({1.30*\R*cos(180)}, {1.15*\R*sin(150)}) {$T^+_{p-2} = T_{p-1}^-$};
  \node[Alabel] at ({1.15*\R*cos(90)}, {1.15*\R*sin(90)}) {$T^+_{p-1} = T_0^-$};
  \node[Alabel] at ({1.20*\R*cos(0)}, {1.15*\R*sin(30)}) {$T^+_0 = T_1^-$};
  \node[Alabel] at ({1.20*\R*cos(0)}, {1.15*\R*sin(-30)}) {$T^+_1 = T_2^-$};

\end{tikzpicture}
\end{center}

Refining this further with the JSJ decomposition of each $N_i$ yields, by uniqueness again \cite[Proposition 1.6.2]{3manifold}, the JSJ decomposition of $\widetilde{M}_p$, which is also a closed graph manifold. Hence the graph of groups decomposition for $K \rtimes \langle t^p \rangle$ arising from the above graph of spaces:
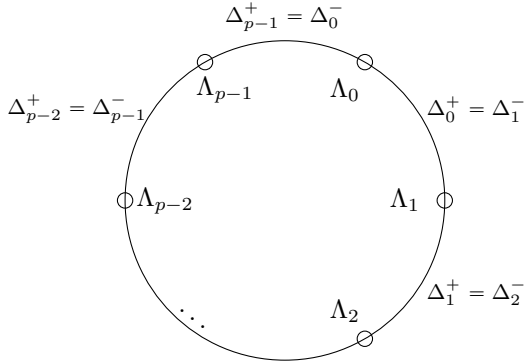
\begin{center}
\begin{tikzpicture}[scale=0.85, baseline]

\tikzset{
  Hlabel/.style={font=\normalsize},
  Alabel/.style={font=\footnotesize}
}

\def\R{2.5} 

\draw (0,0) circle (\R);

\foreach \k/\angle in {
  {p-2}/180,
  {p-1}/120,
   0/60,
   1/0,
   2/-60
}{
  \draw ({\R*cos(\angle)}, {\R*sin(\angle)}) circle (0.12);

  \node[Hlabel] at ({0.75*\R*cos(\angle)}, {0.8*\R*sin(\angle)}) {$\Lambda_{\k}$};
}

\node[Hlabel] at ({0.9*\R*cos(-130)}, {0.9*\R*sin(-130)}) {$\ddots$};

  \node[Alabel] at ({1.30*\R*cos(180)}, {1.15*\R*sin(150)}) {$\Delta^+_{p-2} = \Delta_{p-1}^-$};
  \node[Alabel] at ({1.15*\R*cos(90)}, {1.15*\R*sin(90)}) {$\Delta^+_{p-1} = \Delta_0^-$};
  \node[Alabel] at ({1.20*\R*cos(0)}, {1.15*\R*sin(30)}) {$\Delta^+_0 = \Delta_1^-$};
  \node[Alabel] at ({1.20*\R*cos(0)}, {1.15*\R*sin(-30)}) {$\Delta^+_1 = \Delta_2^-$};

\end{tikzpicture}
\end{center}
is efficient by Theorem \ref{jsj efficient}. The group $K_p$ is the fundamental group defined on the subgraph of groups obtained by removing the non-separating edge connecting $\Lambda_{p-1}$ to $\Lambda_0$. Therefore it is an efficient subgroup of $K \rtimes \langle t^p \rangle$ by Lemma \ref{efficient subgraph}. But every finite-index subgroup of $K \rtimes \langle t^p \rangle$ is also a finite-index subgroup of $\Gamma$, so $K_p$ is an efficient subgroup of $\Gamma$ as well. Finally, $K_p$ is the fundamental group of the space obtained by cutting $\widetilde{M}_p$ along a JSJ torus, and so it is the fundamental group of a compact graph manifold with non-empty (toroidal) boundary, which, as we have seen in Theorem \ref{nongraph fd}, is FD.
\end{proof}

\subsubsection*{End of the proof}

We finally have everything in place to prove Proposition \ref{graph fd}, and thus complete the proof of Theorem \ref{manifold fd}.

\begin{proof}[Proof of Proposition \ref{graph fd}]
    Let $M$ be the fundamental group of a closed graph manifold. If $M$ is Seifert fibered, or finitely covered by a torus bundle, then $\pi_1(M)$ is FD by Lemma \ref{graph fd easy}. Otherwise, Corollary \ref{virtually fd} allows to pass to a finite cover of $M$, so by Lemma \ref{graph fd separating}, we may assume that $M$ has a non-separating JSJ torus. With the notation established in the previous paragraph, we have
    \[\Gamma \coloneqq \pi_1(M) = K \rtimes \langle t \rangle.\]
    By Theorem \ref{LS:coamenable}, it remains to show that $K$ is a directed union of FD, efficient subgroups of $\Gamma$. Now $K$ is the union of the groups $\bar{K}_p \coloneqq \langle \Lambda_i : -p \leq i < p \rangle$. Because $\bar{K}_p$ is conjugate to $K_{2p}$, it is an FD, efficient subgroup of $\Gamma$ by Lemma \ref{graph fd final step}, and we conclude.
\end{proof}

\subsection{One-relator groups}

As in Remark \ref{one rel fg}, an infinitely generated one-relator group is a free product of a finitely generated one-relator group, and a free group, so thanks to Theorem \ref{LS:free} and Corollary \ref{fd fp cor} the question of which one-relator groups have property FD reduces to the finitely generated case.

In Subsection \ref{ss:or}, we showed that many one-relator groups are LP by showing that they are virtually free-by-cyclic. This also implies FD, by Proposition \ref{vfbc fd}. Property RFD for some one-relator groups with non-trivial centre was proved with different methods in \cite[Theorem 11]{Hadwin:2018aa}.

Many of the remaining examples are not even residually finite, such as the Baumslag--Solitar groups $BS(m, n)$, when $1 \neq |m| \neq |n| \neq 1$, or the Baumslag--Gersten groups, hence they cannot be FD or even RFD. On the other hand, the group $BS(1, n)$ from Example \ref{bs solvable} is FD, being amenable and residually finite (Remark \ref{RFD:MAP:FD:RF}).

\subsection{Limit groups}

We saw in Example \ref{limit lp} that limit groups are LP.

\begin{proposition}
\label{limit fd}
    Limit groups are FD.
\end{proposition}

This is not new, indeed limit groups were already known to have property MD \cite[Theorem 4.4]{treeable}, which is stronger than property FD. We discuss more on this in the appendix.

\begin{proof}
    Let $\Gamma$ be a limit group. Then there exists a free normal subgroup $\Lambda < \Gamma$ such that $\Gamma/\Lambda$ is torsion-free nilpotent \cite{kochloukova}. We claim that this satisfies the hypotheses of Theorem \ref{LS:coamenable}. Indeed, $\Lambda$ is a directed union of finitely generated free groups, which are FD by Theorem \ref{LS:free}. The separability condition follows from the fact that every finitely generated subgroup of a limit group is separable \cite{wilton:hall}.
\end{proof}

Also the related groups in Example \ref{graphs of free groups} are LP, being finitely generated and virtually free-by-cyclic, by Proposition \ref{vfbc fd}.

\subsection{Right-angled Artin groups}

We saw in Example \ref{raag lp} that a right-angled Artin group defined on a chordal graph is LP.

\begin{proposition}
\label{raag fd}
    A right-angled Artin group defined on a (finite) chordal graph is FD.
\end{proposition}

\begin{proof}
    Denote by $A(X)$ the right-angled Artin group defined on the graph $X$. We prove the statement by induction on the number of vertices of $X$. When $X$ has a single vertex, $A(X) \cong \Z$.

    Now suppose that the statement is true for all chordal graphs with less vertices than $X$. If $X$ is disconnected, say $X = X_1 \sqcup X_2$ with both $X_i$ proper induced subgraphs and no edge connecting $X_1$ to $X_2$, then $A(X) \cong A(X_1) \Asterisk A(X_2)$, so by induction and Corollary \ref{fd fp cor}, $A(X)$ is FD.

    Suppose instead that $X$ is connected. By a theorem of Dirac \cite{dirac:graphs} we can decompose $X$ as $X_1 \cup X_2$, where the $X_i$ are proper induced subgraphs, $C \coloneqq X_1 \cap X_2$ is a (non-empty) clique, and there is no edge connecting $X_1 \setminus C$ to $X_2 \setminus C$\footnote{This in fact characterizes chordal graphs.}. Now $A(X) \cong A(X_1) \Asterisk_{A(C)} A(X_2)$; moreover $A(C)$ is free abelian, hence amenable, and it is a retract of both $A(X_i)$: the retraction is given by setting all generators labelled by vertices in $X_i \setminus C$ to be equal to the identity. Hence by induction and Corollary \ref{fd am}, $A(X)$ is FD.
\end{proof}

\section{Applications to group stability}
\label{s:stability}

In this section we explore the connection between the (L)LP, property (R)FD, and a topic of increasing interest in group theory: the stability\footnote{Here we are only concerned with \emph{pointwise} stability, as opposed to more classical notions of \emph{uniform} (also known as \emph{Ulam}) stability, as in \cite{Kazhdan:1982aa}. We refer the reader to the introduction of \cite{ultrametric:stability} for a discussion and detailed comparison between the two notions.} of metric approximations of groups.

\subsection{Definitions}

\begin{definition}\label{uin}
Let $U_n=U(M_n(\C))$ be the $n$-dimensional unitary group.  A norm $\|\cdot\|_\nu$ on $M_n(\C)$ will be called \emph{unitarily invariant} if $\|uav\|_\nu=\|a\|_\nu$ for all $a\in M_n(\C)$ and all $u,v\in U_n$, and will be called \emph{normalized} if $\|1\|_\nu=1$.
\end{definition}

Throughout this section, we keep the notation $\| \cdot \|$ for the operator norm: all other norms will be decorated with a subscript for clarity.
In general, unitarily invariant norms on matrices can be characterized via appropriate functions of the singular values: see for example \cite[Section 7.4.7]{Horn:2013aa}.  Examples of unitarily invariant norms include the operator norm, and the Schatten $p$-norm defined by $\|a\|_p \coloneqq (\mathrm{trace}((a^*a)^{p/2}))^{1/p}$ for fixed $p\in [1,\infty)$.  Note that $\|1\|_p=n^{1/p}$, so the Schatten $p$-norm is not normalized; one may also consider a normalized version by just dividing by $n^{1/p}$, however.  In particular, the \emph{Hilbert--Schmidt norm} is defined to be the normalized Schatten $2$-norm.

\begin{remark}\label{com norm rem}
Let us make some observations about a unitarily invariant norm $\|\cdot \|_\nu$ on $M_n(\C)$.
\begin{enumerate}[(i)]
\item We have $\|u\|_\nu=\|1\|_\nu$ for all $u\in U_n$. In particular, a normalized unitarily invariant norm induces a bi-invariant metric on $U_n$ of diameter at most two.
\item \label{ideal part} The Russo--Dye theorem\footnote{Or the finite-dimensional Krein--Milman theorem coupled with the fact that the extreme points of the unit ball in $M_n(\C)$ for the operator norm are exactly the unitaries \cite[II.3.2.17]{Blackadar:2006eq}.} (see \cite{Gardner:1984aa} for a short proof) implies that the closed unit ball in $M_n(\C)$ for the operator norm $\|\cdot \|$ is the convex hull of $U_n$.  Let $a\in M_n(\C)$ satisfy $\|a\|\leq 1$, so we can write $a=\sum t_i u_i$ as a convex combination of unitaries.  Then for any $b\in M_n(\C)$ 
$$
\|ab\|_\nu\leq \sum t_i \|u_ib\|_\nu=\sum t_i \|b\|_\nu\leq \|b\|_\nu.
$$
Hence for any $a,b\in M_n(\C)$, $\|ab\|_\nu\leq \|a\|\|b\|_\nu$ and similarly $\|ab\|_\nu\leq \|a\|_\nu\|b\|$.  Thus while unitarily invariant norms need not be submultiplicative\footnote{Let $\|\cdot \|_2$ be the Hilbert--Schmidt norm on $M_2(\C)$, and let $p$ be a rank one projection.  Then $\|p^2\|_2=\|p\|_2=2^{-1/2}\not\leq 2^{-1}=\|p\|_2^2$. The failure of submultiplicativity for the Hilbert--Schmidt norm is the main obstacle to apply the cohomological methods from \cite{Chiffre:2018ds, oppenheim:lubotzky} to produce a non-hyperlinear group.}, they have a submultiplicative property `relative to the operator norm'. 
\item \label{op big} If $\|\cdot \|_\nu$ is in addition normalized, then using part \eqref{ideal part} (or the Russo--Dye theorem directly), we see that $\|a\|_\nu=\|a1\|_\nu\leq \|a\|$.  Hence any normalized unitarily invariant norm is dominated by the operator norm, which is thus the `largest' such norm.
\end{enumerate}
\end{remark}

\begin{definition}\label{stab def}
Let $\mathcal{U}$ denote a fixed sequence ($\|\cdot\|_n)_{n=1}^\infty$ of unitarily invariant norms on the unitary groups $U_n$. 
Let $\Gamma$ be a countable\footnote{We work only with countable groups as approaching stability problems through sequences is typical in the literature.  This is not really necessary, but (for the sake of simplicity and consistency with the established literature) we leave the appropriate `nettified' versions to the reader.} discrete group, and let $S$ be a fixed family of (not necessarily finite-dimensional) unitary representations of $\Gamma$.

An \emph{asymptotic representation (with respect to $\mathcal{U}$)} is a sequence  $\phi=(\phi_n \colon \Gamma \to U_{k_n})_n$ of maps such that
\[\| \phi_n(gh) - \phi_n(g)\phi_n(h) \|_{k_n} \xrightarrow{n \to \infty} 0 \text{ for all } g, h \in \Gamma.\]

Assume we have a sequence of Hilbert spaces $H_n$ and a sequence of (genuine) representations $\psi=(\psi_n \colon \Gamma\to U(H_n))$ from $S$.  We say that $\phi$ is \emph{close to a corner of $\psi$} if there exist isometric inclusions $v_n \colon \C^{k_n}\to H_n$ such that 
\[ \|\phi_n(g)-v_n^*\psi_n(g)v_n\|_{k_n} \xrightarrow{n \to \infty} 0.\]

With notation as above, we say that $\Gamma$ is:
\begin{enumerate}[(i)]
\item \emph{$\mathcal{U}$-$S$-stable} if every asymptotic representation is close to a (corner of a) sequence of representations from $S$, with $H_n=\C^{k_n}$;
\item \emph{flexibly $\mathcal{U}$-$S$-stable} if every asymptotic representation is close to a corner of a sequence of representations from $S$, with $\mathrm{dim}(H_n)/k_n\to 1$;
\item \emph{very flexibly $\mathcal{U}$-$S$-stable} if every asymptotic representation is close to a corner of a sequence of representations from $S$.
\end{enumerate}
We will also talk about \emph{((very) flexible) $\mathcal{U}$-stability} without specifying $S$; in that case, $S$ should be assumed to be the collection of all finite-dimensional representations.
\end{definition}

In this paper, we will almost exclusively consider very flexible stability, but we mention the other notions for completeness.  Let us make some brief comparisons to the literature.

\begin{remark}\label{stability other notions}
\begin{enumerate}[(i)]
\item The notion of ((very) flexible) $\mathcal{U}$-stability (so $S$ is the family of all finite-dimensional representations) is by now fairly well-established in the literature.  It was proposed by Becker--Lubotzky in the context of approximate permuation representations \cite[Section 4.4]{Becker:2020aa}, compare also with \cite[Section 6.4]{Enders:2024aa} (for $C^*$-algebras and both the operator and Hilbert--Schmidt norms) and \cite[Section 6]{eckhardt:shulman} (for groups and the Hilbert--Schmidt norm).
\item \label{weak stab} If $S$ is the collection of \emph{all} representations, and $\mathcal{U}$ consists of the Hilbert--Schmidt norms, then very flexible $\mathcal{U}$-$S$ stability could be called \emph{ucp stability}. Dogon defines a notion of \emph{weak ucp stability} in \cite[Definition 1.5]{Dogon:2023aa}.
The difference is that there one does not require the stability property to hold for all asymptotic representations, but only for hyperlinear approximations, that is those for which $\lim\limits_{n \to \infty} \| \phi_n(g) - 1 \|_{k_n}$ is as large as possible, for every $g \in \Gamma \setminus \{ 1 \}$. This softening of the notion of stability is called ``weak'' stability in other contexts: see for example \cite[Definition 7.1]{Arzhantseva:2015aa}.
\item The case where $S$ consists of all representations that factor through a finite quotient seems interesting, and it will be the subject of our strongest applications.  We are not aware of this having been studied in the literature before.
\end{enumerate}
\end{remark}

\subsection{Very flexible stability}

The following observation is fundamental to our operator-algebraic approach to stability questions. 

\begin{remark}\label{ass hom}
Let $\phi=(\phi_n \colon \Gamma\to U_{k_n})_n$ be an asymptotic representation with respect to a sequence $\mathcal{U}=(\|\cdot\|_n)_n$ of normalized unitarily invariant norms on the unitary groups $U_n$.  Define $M \coloneqq \prod_n M_{k_n}(\C)$ to be the $C^*$-algebra of (operator norm) bounded sequences $(a_n)_n$ with each $a_n$ in $M_{k_n}(\C)$, equipped with the supremum (operator) norm: $\|(a_n)_n\| \coloneqq \sup_n\|a_n\|$.  Define $J \coloneqq \{(a_n)_n\in M\mid \|a_n\|_{k_n} \xrightarrow{n \to \infty} 0\}$; using properties \eqref{ideal part} and \eqref{op big} from Remark \ref{com norm rem}, $J$ is an ideal in $M$, closed for the norm on $M$.

It follows from the definition of asymptotic representation that $\phi$ uniquely determines a representation of $\Gamma$ with values in the unitary group of $M/J$.  Moreover, by the universal property of $C^*\Gamma$, this homomorphism extends uniquely to a $*$-homomorphism 
$$
\phi \colon C^*\Gamma\to M/J.
$$
We call this the \emph{$*$-homomorphism associated to $\phi$}.
\end{remark}

We next record a corollary of Theorem \ref{llp to lp} above.  It is part of the folklore of the subject, appearing in closely related forms in \cite[Corollary 1.7]{Ioana:2020aa} and \cite[Proposition 2.7]{Willett:2024aa}, for example.

\begin{corollary}\label{mat lift}
With notation as in Remark \ref{ass hom}, let $\phi$ be an asymptotic representation of $\Gamma$, with associated $*$-homomorphism $\phi \colon C^*\Gamma\to M/J$.  If $\Gamma$ has the LLP, then $\phi$ admits a ucp lift.
\end{corollary}


\begin{proof}
Using Theorem \ref{llp to lp}, it suffices to show that $M$ is QWEP.  Indeed, it is actually \emph{injective} in the category of unital $C^*$-algebras and ucp maps, which is much stronger: see for example \cite[Examples IV.2.1.2]{Blackadar:2006eq} and surrounding discussion.
\end{proof}

\begin{remark}\label{mat lift wucps}
Corollary \ref{mat lift} implies that if $\Gamma$ has the LLP, then it is ucp stable in the sense of Remark \ref{stability other notions}, part \eqref{weak stab}. In particular, it is weakly ucp stable in Dogon's sense \cite[Definition 1.5]{Dogon:2023aa}, as already observed by Dogon in that paper.
\end{remark}

\begin{lemma}\label{llp approx}
Let $\mathcal{U}$ denote a sequence ($\|\cdot\|_n)_n$ of normalized unitarily invariant norms on the unitary groups $U_n$. Let $\Gamma$ be a countable\footnote{One could prove an analogous (`nettified') statement for uncountable groups, but one would need to assume the LP rather than the LLP, as Theorem \ref{llp to lp} would no longer be available.  Such a statement does not seem to be of any interest currently: see Question \ref{q:lp:uncountable}.} group with the LLP.   Let $(\phi_n \colon \Gamma\to U_{k_n})_n$ be an asymptotic representation.  

Then there is a sequence $(\psi_n \colon C^*\Gamma\to M_{k_n}(\C))_n$ of ucp maps such that $\|\phi_n(g)-\psi_n(g)\|_{k_n} \xrightarrow{n \to \infty} 0$ for all $g\in \Gamma$.
\end{lemma}

\begin{proof}
With notation as in Remark \ref{ass hom}, let $\phi \colon C^*\Gamma\to M/J$ be the associated $*$-homomorphism for $C^*\Gamma$.  Using the LLP and Corollary \ref{mat lift} there is a ucp lift $\psi \colon C^*\Gamma\to M$.  It is elementary to check that $\psi$ is necessarily of the form $(\psi_n)_n$ with each $\psi_n \colon C^*\Gamma\to M_{k_n}(\C)$ a ucp map.  The fact that $\psi$ lifts $\phi$ implies that $(\psi_n(g)-\phi_n(g))_n \in J$ for all $g\in \Gamma$, or in other words, that $\|\phi_n(g)-\psi_n(g)\|_{k_n} \xrightarrow{n \to \infty} 0$ for all $g\in \Gamma$.
\end{proof}

We need one more representation-theoretic lemma.

\begin{lemma}\label{s corner}
Let $A$ be a unital $C^*$-algebra, and let $S$ be a collection of $*$-representations of $A$ that is dense in the Fell topology, and that is closed under finite direct sums, unitary equivalence, and taking subrepresentations.  Then for any ucp map $\phi \colon A\to M_n(\C)$, any $\epsilon>0$, and any finite subset $A_0$ of $A$ there exists a unital $*$-representation $\pi \colon A\to \mathcal{B}(H)$ in $S$, and an isometry $v \colon \C^n\to H$ such that $\|v^*\pi(a)v-\phi(a)\|<\epsilon$ for all $a\in A_0$.
\end{lemma}

\begin{proof}
Using Stinespring's dilation theorem (see Remark \ref{stinespring}) there exists a unital representation $\sigma \colon A\to \mathcal{B}(H)$ and an isometry $w \colon \C^n\to H$ such that 
\begin{equation}\label{sst}
\phi(a)=w^*\sigma(a)w.
\end{equation}
Let $\aleph$ be a cardinal with the property in Theorem \ref{rep approx}.  Taking the direct sum of $H$ with a suitably large Hilbert space (and setting $\sigma$ to be zero on the other summand), we may assume $H$ has dimension at least $\aleph$; this loses the property that $\sigma$ is unital, but note that $w$ has image contained in the essential space $H_\sigma$ of $\sigma$.  Then Theorem \ref{rep approx} gives a net $(\pi_i)$ of representations on $H$ from $S$ such that 
\begin{equation}\label{pi i sot s}
\|\pi_i(a)\xi- \sigma(a)\xi\|_H\xrightarrow{i \to \infty}0 \quad \text{for all}\quad \xi\in H_\sigma \text{ and } a\in A.
\end{equation} 
Now, let $p=ww^*$, and note that $p$ is a finite-rank projection onto a subspace of $H_\sigma$.  It follows from this and line \eqref{pi i sot s} that 
\begin{equation}\label{comp p}
\|(\pi_i(a)-\sigma(a))p\|\xrightarrow{i \to \infty} 0.
\end{equation}
Applying this to $a=1$, and using that $\sigma(1)p=p$, we get that $\|\pi_i(1)p-p\|\xrightarrow{i \to \infty} 0$.  A standard application of the $C^*$-algebra functional calculus (see for example \cite[Lemma 2.3]{Bratteli:1998vg}) then gives a net $(u_i)_i$ of unitaries on $H$ such that $\|u_i-1\|\to 0$ and $u_ipu_i^*\leq \pi_i(1)$ for all $i$.  Let $H_i\leq H$ be the essential space of $\pi_i$; then the fact that $u_ipu_i^*\leq \pi_i(1)$ implies that $u_i$ restricted to $\mathrm{Range}(p)=\mathrm{Range}(w)$ takes image in $H_i$.  Let now $\overline{u}_i \colon \mathrm{Range}(w)\to H_i$ be the restricted and corestricted map coming from $u_i$, and let $\overline{\pi_i} \colon A\to \mathcal{B}(H_i)$ be the unital corestriction of $\pi_i$.  Define now $v_i \coloneqq \overline{u}_iw$.  We claim that for $i$ suitably large, $v \coloneqq v_i$ and $\pi \coloneqq \overline{\pi_i}$ have the properties needed in the lemma.  It suffices to show that $\|v_i^*\overline{\pi}_i(a)v_i^*-\phi(a)\|\xrightarrow{i \to \infty} 0$ for any $a\in A$. 

Indeed, we have 
\begin{align*}
\|v_i^*\overline{\pi}_i(a)v_i^*-\phi(a)\| & =\|w^*u_i^*\pi_i(a)u_iw -\phi(a)\| \\
& \leq 2\|u_i-1\|\|a\|+\|w^*\pi_i(a)w-\phi(a)\| \\
& = 2\|u_i-1\|\|a\|+\|w^*(\pi_i(a)-\sigma(a))w\| \\
& = 2\|u_i-1\|\|a\|+\|ww^*(\pi_i(a)-\sigma(a))ww^*\| \\
&\leq 2\|u_i-1\|\|a\|+\|(\pi_i(a)-\sigma(a))p\|
\end{align*}
where the third to last equality uses line \eqref{sst}, the second to last equality uses that $b\mapsto wbw^*$ is an isometry for the operator norm, and the last inequality uses that $ww^*=p$.  Using line \eqref{pi i sot s} and that $\|u_i-1\|\to 0$, we are done.
\end{proof}

\begin{remark}
If one assumes that $S$ is closed under general direct sums, then Lemma \ref{s corner} holds for ucp maps $\phi \colon A\to \mathcal{B}(H)$ for any Hilbert space $H$.  This can be deduced from Voiculescu's theorem \cite{Voiculescu:1976aa}\footnote{The proof of Voiculescu's theorem was subsequently simplified by Arveson \cite{Arveson:1977aa}, and this forms the basis for several modern textbook treatments such as \cite[Section 1.7]{Brown:2008qy}, \cite[Section II.5]{Davidson:1996jq} or \cite[Sections 3.4-6]{Higson:2000bs}.} for separable $A$, and from Hadwin's generalizations \cite[Section 3]{Hadwin:1981aa} of Voiculescu's theorem in general.
\end{remark}

The following result generalizes some results in the literature, for example \cite[Theorem 6.4]{eckhardt:shulman} (for the Hilbert--Schmidt norms and $S$ the family of finite-dimensional representations).  It follows directly from Lemmas \ref{llp approx} and \ref{s corner}.  

\begin{theorem}\label{stability thm}
Let $\Gamma$ be a countable discrete group with the LLP.  Let $S$ be a family of representations of $\Gamma$ that is dense in the Fell topology, and is closed under finite direct sums, unitary equivalence, and taking subrepresentations.  Let $\mathcal{U}$ be a sequence of unitarily invariant normalized norms on $M_n(\C)$.  

Then $\Gamma$ is very flexibly $\mathcal{U}$-$S$ stable. \qed
\end{theorem}

The most interesting examples of $S$ where we can apply the above theorem seem to be the family of all finite-dimensional representations (in which case it applies to RFD $\Gamma$) and the family of all representations that factor through finite quotients (in which case it applies to FD $\Gamma$).

\begin{corollary}
\label{examples very flexibly}
    Let $\mathcal{U}$ be a sequence of unitarily invariant normalized norms on $M_n(\C)$. For a countable discrete group $\Gamma$, let $S$ be the family of representations of $\Gamma$ that factor through a finite quotient. If $\Gamma$ is finitely generated and satisfies one of the following properties, then it is very flexibly $\mathcal{U}$-$S$ stable, hence in particular very flexibly $\mathcal{U}$-stable.
    \begin{enumerate}[(i)]
        \item\label{examples very flexibly amenable} $\Gamma$ is amenable and residually finite;
        \item\label{examples very flexibly manifold} $\Gamma = \pi_1(M)$, where $M$ is a connected manifold of dimension at most $3$;
        \item\label{examples very flexibly vfbc} $\Gamma$ is virtually free-by-cyclic;
        \item\label{examples very flexibly 1rel} $\Gamma$ is a one-relator group, and $\Gamma$ has either torsion, negative immersions, non-trivial center or a small cancellation relation;
        \item\label{examples very flexibly limit} $\Gamma$ is a limit group;
        \item\label{examples very flexibly raag} $\Gamma$ is a right-angled Artin group on a chordal graph.
    \end{enumerate}
\end{corollary}

\begin{proof}
    By Theorem \ref{stability thm}, we need to show that under these conditions, $\Gamma$ is LLP and FD.
    \eqref{examples very flexibly amenable} is Corollary \ref{amen cor} and Remark \ref{RFD:MAP:FD:RF}.
    \eqref{examples very flexibly manifold} is Example \ref{manifold lp} and Theorem \ref{manifold fd}.
    \eqref{examples very flexibly vfbc} is Corollary \ref{vfbc} and Proposition \ref{vfbc fd}.
    \eqref{examples very flexibly 1rel} follows from \eqref{examples very flexibly vfbc} and Examples \ref{or tor}, \ref{or ni}, \ref{or center} and \ref{or sc}.
    \eqref{examples very flexibly limit} is Example \ref{limit lp} and Proposition \ref{limit fd}.
    \eqref{examples very flexibly raag} is Example \ref{raag lp} and Proposition \ref{raag fd}.
\end{proof}

\begin{remark}
\label{very flexibly amenable converse}
    Note that in Item \eqref{examples very flexibly amenable} we did not use the finite generation hypothesis; moreover, by Remark \ref{RFD:MAP:FD:RF}, to obtain very flexible $\mathcal{U}$-stability it suffices to assume that $\Gamma$ is MAP, hence RFD. For the Hilbert--Schmidt norm, this is \cite[Theorem 6.6]{eckhardt:shulman}. The statement also includes a converse: if an amenable group is very flexibly Hilbert--Schmidt stable, then it is MAP; this is a direct consequence of the hyperlinearity of amenable groups. We can obtain a similar converse for the operator norm, using that amenable groups are MF \cite{Tikuisis:2015kx}. With the same argument we also obtain an operator norm analog of \cite[Theorem 4.16]{local:stability}, which characterizes similarly very flexible local Hilbert--Schmidt stability.
\end{remark}

\begin{remark}
    Some one-relator groups with center were previously shown to be Hilbert--Schmidt stable \cite[Theorem 8]{Hadwin:2018aa}. Very recently, Spaas proved that a RAAG on a chordal graph is Hilbert--Schmidt stable \cite{pieter:chordal} (see \cite{atkinson} for a related result).
\end{remark}

\begin{example}
The group $\Gamma=F_2\times F_2$ is not very flexibly stable with respect to the operator norm.  Indeed, using \cite[Proposition 7.4.5]{Brown:2008qy}, $C^*\Gamma$ embeds into $(\prod_n M_n(\C))/(\bigoplus_n M_n(\C))$.  If it were very flexibly operator norm stable, we could conclude that $F_2\times F_2$ is RFD, which is false by the negative solution to the Connes embedding problem (see Remark \ref{f2 f2 not rfd} above). It follows that a right-angled Artin group on a graph that contains an induced square is not very flexibly stable: indeed, such a group contains a copy of $F_2 \times F_2$ as a retract, and it is easy to see that a retract of a very flexibly stable group must be very flexibly stable, see for example \cite[Lemma 4.2]{glebe:cup}. This motivates the question of whether a right-angled Artin group on a graph that does not contain induced squares, for example the pentagon, is very flexibly operator stable, see Question \ref{q:raag}.

Using other methods, Ioana has shown that $F_2\times F_2$ is not very flexibly stable in permutations \cite{ioana:tau}, nor flexibly stable with respect to the Hilbert--Schmidt norm \cite{Ioana:2024aa} .  Whether or not $F_2\times F_2$ is very flexibly stable with respect to the Hilbert--Schmidt norm seems to be open.

In permutations, the situation is even more restrictive, indeed $F_2 \times \Z$ is not very flexibly stable \cite{ioana:tau}. It is however Hilbert--Schmidt stable, in fact Hilbert--Schmidt stability is preserved by direct products with amenable groups \cite{ioana:spaas}. 
\end{example}

\subsection{Algebraic complements of asymptotic representations}

We now turn to another version of stability that is studied in \cite{Dadarlat:384aa} and \cite{Willett:2024aa} (for the operator norm).  To motivate the main definition, we start with a definition and a lemma.

\begin{definition}\label{comp}
A sequence $(\|\cdot\|_n)_n$ of norms on each $M_n(\C)$ is \emph{compatible} if for any $n\leq m$, any isometry $v \colon \C^n\to \C^m$, and any $a\in M_n(\C)$, we have $\|vav^*\|_m\leq \|a\|_n$.

Let $\mathcal{U}=(\|\cdot\|_n)_n$ be a compatible family of normalized unitarily invariant norms on $M_n(\C)$ as in Definition \ref{uin}.  Let $S$ be a family of unitary representations of $\Gamma$ (equivalently, of nondegenerate $*$-representations of $C^*\Gamma$).

Let $\phi = (\phi_n \colon \Gamma\to U_{k_n})_n$ be an asymptotic representation.  We say that $\phi$ is \emph{$S$-complementable} if there is an asymptotic representation $\psi = (\psi_n:\Gamma\to U_{l_n})_n$ such that if $\phi \oplus \psi  = (\phi_n\oplus\psi_n:\Gamma\to U_{k_n+l_n})_n$ is the associated block sum asymptotic representation, and 
$$
\phi\oplus \psi \colon C^*\Gamma\to M/J
$$
the associated $*$-homomorphism as in Remark \ref{ass hom}, then $\phi\oplus \psi$ lifts to a $*$-homomorphism $\theta \colon C^*\Gamma\to M$ such that each component $\theta_n \colon C^*\Gamma\to M_{k_n+l_n}(\C)$ is in $S$.

We say that $\phi$ is \emph{algebraically $S$-complementable} if $\psi$ as above can be chosen to be a sequence $(\psi_n \colon \Gamma\to U_{l_n})_n$ of genuine representations with each $\psi_n$ in $S$.

Finally, if $S$ is the family of all finite-dimensional representations, we just say \emph{complementable} and \emph{algebraically complementable}.
\end{definition}

Note that the operator norm, and also the normalized Schatten $p$-norms and Hilbert--Schmidt norm, are compatible normalized families in the sense of Definition \ref{comp}.  The following lemma is implicit in \cite[Section 6.4]{Enders:2024aa} for the Hilbert--Schmidt norms.  

\begin{lemma}
With notation as in Definition \ref{comp}, the following are equivalent.
\begin{enumerate}[(i)]
    \item \label{vfs} $\Gamma$ is very flexibly $\mathcal{U}$-$S$-stable; 
\item \label{scomp} every asymptotic representation of $\Gamma$ with respect to $\mathcal{U}$ is $S$-complementable. 
\end{enumerate}
\end{lemma}

\begin{proof}
Assume \eqref{vfs} holds and let $(\phi_n \colon \Gamma\to U_{k_n})_n$ be an asymptotic representation.  According to the definition of very flexible $\mathcal{U}$-$S$-stability, there are $m_n\geq k_n$, a sequence $(\theta_n \colon \Gamma\to U_{m_n})_n$ of representations in $S$ and a sequence $v_n \colon \C^{k_n}\to \C^{m_n}$ of isometries such that 
\begin{equation}\label{t p 0}
\|v_n^*\theta_n(g)v_n-\phi_n(g)\|_{k_n}\xrightarrow{n \to \infty} 0.
\end{equation}
Let $\mathrm{ad}_n \colon (M_{k_n}(\C),\|\cdot\|_{k_n})\to (M_{m_n}(\C),\|\cdot \|_{m_n})$ be the $*$-homomorphism $b\mapsto v_nbv_n^*$, and note that compatibility of the norms implies that $\mathrm{ad}_n$ is contractive.  Define $p_n \coloneqq v_nv_n^*$.  Then line \eqref{t p 0} implies that 
\begin{align}\label{ptp-p}
\|p_n\theta_n(g)p_n-v_n\phi_n(g)v_n^*\|_{m_n}& =\|\mathrm{ad}_n(v_n^*\theta_n(g)v_n-\phi_n(g))\|_{m_n} \nonumber \\ & \leq \|v_n^*\theta_n(g)v_n-\phi_n(g)\|_{k_n}\xrightarrow{n \to \infty} 0
\end{align}
for all $g\in \Gamma$.  

Consider now $M \coloneqq \prod_{n} M_{m_n}(\C)$
and $J$ the ideal $\{(a_n)_n \in M\mid \|a_n\|_{m_n} \xrightarrow{n \to \infty} 0\}$ as in Remark \ref{ass hom}.  Let $p\in M/J$ be the image of the sequence $(p_n)_n \in M$.  Then the sequences $(\theta_n)_n$ and $(\mathrm{ad}_n\circ\phi_n)_n$ induce $*$-homomorphisms 
$$
\theta,\phi \colon C^*\Gamma\to M/J
$$
with $\theta$ unital, and $\phi(1)=p$.  Using line \eqref{ptp-p}, $p\theta(g)p=\phi(g)$ in $M/J$ for all $g\in \Gamma$.  Hence for any $g\in \Gamma$ the $C^*$-identity in $M/J$ gives 
$$
\|p\theta(g)(1-p)\|^2=\|p\theta(g)(1-p)\theta(g)^*p\|=\|p\theta(1)p-p\phi(1)p\|=0.
$$
Replacing $g$ with $g^{-1}$ and taking adjoints, $(1-p)\theta(g)p=0$ also, whence 
$$
p\theta(g)-\theta(g)p=p\theta(g)(1-p)-(1-p)\theta(g)p=0
$$
and so $p$ commutes with $\theta(C^*\Gamma)$ in $M/J$.  It follows that if we define $\psi_n(g) \coloneqq (1-p_n)\theta_n(g)(1-p_n)$, $(\psi_n)_n$ will have the properties required by $S$-complementability.

The converse is straightforward.
\end{proof}

\begin{remark}
We can also understand flexible $\mathcal{U}$-stability in this language for some important normalized families.  

Indeed, assume first that $(\|\cdot\|_n)_n$ is a compatible family of normalized unitarily invariant norms such that $\|p_n\|_n \xrightarrow{n \to \infty} 0$ when $(p_n)_n$ is any sequence of projections with $\mathrm{rank}(p_n)/n\to 0$; the normalized Schatten $p$-norms satisfy this condition, but the operator norm does not.  Then flexible $\mathcal{U}$-stability just says that for any asymptotic representation $(\phi_n)_n$ the associated homomorphism $\phi \colon C^*\Gamma\to M/J$ of Remark \ref{ass hom} admits a lift to a (not necessarily unital) $*$-homomorphism $\phi \colon C^*\Gamma\to M$.  

If on the other hand $(\|\cdot \|_n)_n$ is the family of operator norms, this lifting property is exactly stability.
\end{remark}

Thus algebraic complementability is a natural strengthening of our earlier notion of very flexible stability.  Algebraic complementability was studied in \cite{Dadarlat:384aa} and \cite{Willett:2024aa} for the operator norm\footnote{Under different names.  In particular, in \cite{Dadarlat:384aa} Dadarlat  uses ``weak stability'' but we prefer to avoid that as it clashes with the usage of Remark \ref{stability other notions} part \eqref{weak stab} above.}.  This notion has quite a different flavour to complementability, as $K$-theoretic\footnote{In the literature, these $K$-theoretic conditions are often translated into more familiar (co)homological terms using the Chern character; however, it is $K$-theory that is more directly relevant.} conditions come into play. 
Historically, the connection of $K$-theory / cohomology to operator norm stability goes back to the original paper of Voiculescu \cite[page 431]{Voiculescu:1983km}.  Later, connections were made to higher index theory in \cite{Connes:1990vo}, and (at least implicitly) to the $K$-theoretic classification program for $C^*$-algebras in \cite{Gong:1998aa,Eilers:1998aa}.  There has recently been a lot of activity in the area, starting with work of Dadarlat \cite{Dadarlat:2011kx,Dadarlat:2011uq} and Eilers--Shulman--S\o{}rensen \cite{Eilers:2018ab}.

As an example theorem, in \cite[Theorem 1.1]{Dadarlat:384aa}, Dadarlat shows that if $\Gamma$ is a countable discrete MF group that is coarsely embeddable in Hilbert space, and if moreover $H^{2k}(\Gamma;\Q)\neq 0$ for some $k>0$, then $\Gamma$ admits an asymptotic representation with respect to the operator norm that is not algebraically complementable, and in particular $\Gamma$ is not stable. This applies in the classical case of $\mathbb{Z}^2$, essentially due to Voiculescu \cite{Voiculescu:1983km}; notice that $\mathbb{Z}^2$ is very flexibly stable by Corollary \ref{examples very flexibly}.

The results of \cite{Willett:2024aa} give a partial converse to this: they show that under appropriately strong $K$-theoretic assumptions, and assuming moreover that $\Gamma$ is LLP and RFD, then an asymptotic representation is algebraically complmentable if and only if an appropriate $K$-theoretic obstruction vanishes.

More precisely, let $\phi=(\phi_n \colon \Gamma\to U_{k_n})_n$ be an asymptotic representation and assume $\Gamma$ has the LLP. Then one can show that $\phi$ induces a map on $K$-theory 
$$
\phi_* \colon K_0(C^*\Gamma)\to \frac{\prod_n \Z}{\bigoplus_n\Z}
$$
(see for example \cite[Proposition 3.22]{Willett:2024aa}: this shows something rather more technical and precise, but the same ideas show the above).  Assuming for simplicity that $\Gamma$ admits a finite CW complex model for its classifying space $B\Gamma$, the Baum--Connes conjecture identifies $K_0(C^*\Gamma)$ rationally with the even rational group homology, i.e.\ there is an isomorphism 
$$
K_0(C^*\Gamma)\cong \bigoplus_{k=0}^\infty H_{2k}(\Gamma;\Q).
$$
The maps $\phi_*$ are then often computable in terms of index-theoretic or cohomological data: see for example \cite{Dadarlat:2022aa} for an elegant interpretation in terms of winding numbers.  Let $\widetilde{K}_0(C^*\Gamma)$ be the \emph{reduced $K$-theory of $C^*\Gamma$}, i.e.\ the kernel of the map $K_0(C^*\Gamma)\to \Z$ induced by the trivial representation.  Combining the main results of \cite{Willett:2024aa} with our results in this paper, we get the following; one could cover some somewhat more general classes, but we just discuss particularly interesting  illustrative special cases.

\begin{corollary}
\label{examples alg comp}
Let $\Gamma$ be a finitely generated torsion-free group belonging to one of the following classes.
    \begin{enumerate}[(i)]
        \item\label{examples very flexibly manifold 2} $\Gamma = \pi_1(M)$, where $M$ is a connected manifold of dimension at most $3$;
        \item\label{examples very flexibly vfbc 2} $\Gamma$ is virtually free-by-cyclic;
        \item\label{examples very flexibly 1rel 3} $\Gamma$ is a one-relator group, and $\Gamma$ has either negative immersions, non-trivial center or a small cancellation relation.
        \item\label{examples very flexibly limit 3} $\Gamma$ is a limit group that does not contain $\mathbb{Z}^4$.
        \item\label{examples very flexibly raag 3} $\Gamma$ is a right-angled Artin group whose defining graph is chordal and does not contain $4$-cliques.
    \end{enumerate}
Let $S$ be the family of all finite-dimensional representations of $\Gamma$ that factor through a finite quotient.  Then an asymptotic representation $\phi=(\phi_n \colon \Gamma\to U_{k_n})_n$ is algebraically $S$-complementable if and only if the induced map 
$$
\phi_* \colon \widetilde{K}_0(C^*\Gamma)\to \frac{\prod \Z}{\oplus \Z}
$$
vanishes. 
\end{corollary}

\begin{proof}
This follows from \cite[Theorem 7.6]{Willett:2024aa}, combined with the observations that our work here shows the groups above satisfy the hypotheses of that theorem.

More precisely, we need to check that the groups above are LLP, FD, admit a cellular classifying space $B\Gamma$ of dimension at most three, 
and are such that $\Gamma$ satisfies a strong enough version of the Baum--Connes conjecture\footnote{For experts: one needs that the assembly map $RK_*(B\Gamma)\to K_*(C^*\Gamma)$ is an isomorphism, and that $C^*\Gamma$ satisfies the UCT.}.  The LLP and FD follow as in Corollary \ref{examples very flexibly}.
 
For ease of citation, we recall that the existence of a $3$-dimensional classifying space is equivalent to the cohomological dimension being at most $3$ \cite{eilenberg:ganea}.

In case \eqref{examples very flexibly manifold 2}, by \cite[Proposition 2.2]{kielak:linton:manifolds} there is a finite-index subgroup $\Lambda < \Gamma$ that is a free product of a free group and finitely many fundamental groups of compact aspherical $3$-manifolds, hence $\Lambda$ has cohomological dimension at most $3$; it then follows from Serre's Theorem \cite[Th{\'e}or{\`e}me 1.7.1]{Serre:cohomologie} that $\Gamma$ also has cohomological dimension at most $3$. In case \eqref{examples very flexibly vfbc 2}, first we claim that free-by-cyclic groups have cohomological dimension at most $2$: indeed a classifying space can be taken to be the mapping cylinder of a map from a wedge of circles to itself, and again the general case follows from Serre's Theorem; case \eqref{examples very flexibly 1rel 3} is a special case, see Subsection \ref{ss:or}. In case \eqref{examples very flexibly limit 3}, a limit group is either free, hence of cohomological dimension $1$, or has cohomological dimension $\max\{2, n\}$, where $n$ is the maximal rank of a free abelian subgroup \cite[Proposition 1.1(8)]{limit:cat0}. Finally, in case \eqref{examples very flexibly raag 3}, a right-angled Artin group has cohomological dimension equal to size of the largest clique in the defining graph, this follows directly from the construction of the Salvetti complex \cite{salvetti}, which is a classifying space.

For the Baum--Connes assumptions, we note that it suffices to show that $\Gamma$ is a-T-menable: compare \cite[Remarks 5.6 and 6.2]{Willett:2024aa}.  All of the examples considered are in Linnell's class $\mathcal{C}$: for \eqref{examples very flexibly manifold 2} see Example \ref{manifold lp}; for \eqref{examples very flexibly vfbc 2} it follows from the definition; \eqref{examples very flexibly 1rel 3} is a special case of \eqref{examples very flexibly vfbc 2}; for \eqref{examples very flexibly limit 3} see Example \ref{limit lp}; and finally for \eqref{examples very flexibly raag 3} see Example \ref{raag lp}. Finally, groups in Linnell's class $\mathcal{C}$ are a-T-menable (Remark \ref{linnell:atmenable}).
\end{proof}

\begin{remark}
In parts \eqref{examples very flexibly limit 3} and \eqref{examples very flexibly raag 3} of Theorem \ref{examples alg comp}, we made assumptions on the limit groups and RAAGs involved in order to guarantee that the classifying space of $\Gamma$ is low-dimensional.  These assumptions are in order to deal with possible torsion obstructions from the odd $K$-homology group $RK_1(B\Gamma)$.  If $RK_1(B\Gamma)$ is known to be torsion-free, they could be avoided.  Alternatively, one could work `rationally', and then conclude that $\phi_*=0$ if and only if some finite multiple of $\phi$ is algebraically complemented: compare \cite[Theorem 7.1]{Willett:2024aa}.
\end{remark}

\appendix

\section{Property MD}
\label{appendix md}

In this appendix, which is independent of the rest of the paper, we show how our results on property FD can be strengthened to imply \emph{property MD}, a property introduced by Kechris in \cite{MD:kechris}, alongside the possibly stronger \emph{property EMD}. This also implies a stability property in the setting of permutations, namely \emph{stability in finite actions}, introduced by Gohla and Thom in \cite{gohla:thom}. As remarked by Kechris \cite{MD:kechris}, property MD implies property FD, however we chose to present these results separately to stay within the world of representations for the main body of the paper.

We refer to \cite{MD:burton:kechris} for a survey of this property and its applications in ergodic theory. Our main result is Theorem \ref{MD list} below, that proves property MD for all of the groups for which we proved property FD. The case of $3$-manifolds is especially interesting, indeed some special cases had previously been obtained in \cite[Proposition 3.10]{ifsv} and \cite[Proposition 5.12]{sisv}, with applications to the study of simplicial volume and its integral foliated and stable integral variants\footnote{These applications actually use \emph{property EMD$^*$}, also introduced by Kechris \cite{MD:kechris}, however this is equivalent to property MD \cite[Theorem 1.4]{MD:tuckerdrob}.}. We hope that property MD for general finitely generated $3$-manifold groups will be useful for applications to other gradient invariants of $3$-manifolds, for instance via \cite[Corollary 16.5]{cheapembedding}.

\subsection{Definitions}

We refer the reader to \cite{MD:kechris} for more details. Throughout, $\Gamma$ is a finitely generated residually finite group. A \emph{p.m.p.\ action} of $\Gamma$ will always mean a measure preserving action on a standard non-atomic\footnote{A standard measure space is one that is isomorphic to the Borel structure on a Polish space.  A non-atomic standard probability space is isomorphic to $[0,1]$ with Lebesgue measure.} probability space $(X,\mu)$.  Given such a space $(X,\mu)$, we denote by $A(\Gamma, X, \mu)$ the space of all p.m.p.\ actions of $\Gamma$ on $(X, \mu)$.  

For $a \in A(\Gamma, X, \mu)$, we denote by $g^a$ the $a$-action of the element $g \in \Gamma$. An action is \emph{finite} if it factors through a finite group. Let $F(\Gamma)$ denote the set of finite actions of $\Gamma$.

We say that $a \in A(\Gamma, X, \mu)$ is \emph{weakly contained} in $b \in A(\Gamma, Y, \nu)$, denoted $a \prec b$, if for every collection of Borel sets $A_1, \ldots, A_n \subset X$, elements $g_1, \ldots, g_m \in \Gamma$ and $\epsilon > 0$ there are Borel sets $B_1, \ldots, B_n \subset Y$ such that
\begin{equation}
\label{weak containment}
    |\mu(g_i^a(A_j) \cap A_k) - \nu(g_i^b(B_j) \cap B_k)| < \epsilon
\end{equation}
for all $i,j,k$.  Similarly, we say that a net $(a_\alpha)$ in $A(\Gamma, X_\alpha, \mu_\alpha)$ converges to $a \in A(\Gamma, X, \mu)$ in the \emph{weak topology} if for every collection of Borel sets $A_1, \ldots, A_n \subset X$, elements $g_1, \ldots, g_m \in \Gamma$ and $\epsilon > 0$ there exists $\alpha_0$ such that for all $\alpha\geq \alpha_0$ there are Borel sets $B_1, \ldots, B_n \subset X_\alpha$ satisfying \eqref{weak containment}. Hence $a \prec b$ if and only if $a$ is in the closure of the (isomorphism class of) $b$ for the weak topology.

\begin{definition}
    Let $\Gamma$ be a finitely generated residually finite group. We say that $\Gamma$ has \emph{property MD} if every p.m.p.\ action of $\Gamma$ is a weak limit of finite actions.
\end{definition}

In symbols, $\Gamma$ is MD if $a \in \overline{F(\Gamma)}$ for every $a \in A(\Gamma, X, \mu)$. The following equivalent characterization is used as the definition in some of the results that we will reference.

\begin{proposition}[{\cite[Proposition 4.8]{MD:kechris}}]
    Let $\Gamma$ be a finitely generated residually finite group. Let $i_\Gamma$ be the trivial action on some standard probability space, and let $p_\Gamma$ be the action on its profinite completion. Then $a \in \overline{F(\Gamma)}$ if and only if $a \prec i_\Gamma \times p_\Gamma$. \qed
\end{proposition}

\begin{remark}
\label{MD vs EMD}
    In \cite{MD:kechris}, the author also defines the related properties \emph{EMD} and \emph{EMD$^*$}. We will not define these here, suffice it to say that EMD implies MD by definition, and that Tucker-Drob proved that MD and EMD$^*$ are equivalent \cite[Theorem 1.4]{MD:tuckerdrob}, and that MD and EMD are equivalent for groups without property (T) \cite[Corollary 4.7]{MD:tuckerdrob}.
\end{remark}

\begin{remark}
\label{stability in finite actions}
A related property is that of \emph{stability in finite actions}, introduced by Gohla and Thom in \cite{gohla:thom} as a possible pathway to producing non-sofic groups (see also \cite{CDL:gohla:thom}). The definition is the same as property MD, except only for the p.m.p.\ actions arising from sofic representations, hence property MD implies stability in finite actions by definition.
\end{remark}

\subsection{First examples}

The first and main examples of groups with property MD are the usual ones \cite[Theorem 1.1 and p. 466]{MD:kechris}

\begin{theorem}[Kechris]
\label{MD free amenable}
    Free groups and residually finite amenable groups are MD.\qed
\end{theorem}

As with property FD, our proofs of property MD will use this as a starting point, together with permanence properties. The next two sections are devoted to these.

\subsection{Co-amenability}

The following is part of a list of permanence properties for MD explained in \cite[p. 486]{MD:kechris}.

\begin{lemma}[Kechris]
\label{MD virtually}
    A subgroup of an MD group is MD. A virtually MD group is MD. \qed
\end{lemma}

The main permanence property in the direction of co-amenability was proved in \cite[Theorem 1.4]{MD:bowen:tuckerdrob}, and we will revisit it in the next section (Theorem \ref{MD amenableext iff}). 
We also need an analog of Theorem \ref{LS:coamenable} (that is \cite[Corollary 2.5]{Lubotzky:2004xw}). We refer the reader to Definition \ref{def:efficient} for the notion of efficient subgroups.

\begin{theorem}
\label{MD LS}

    Let $\Lambda < \Gamma$ be a normal subgroup such that $\Gamma/\Lambda$ is infinite, amenable and residually finite. Suppose that $\Lambda$ can be written as a directed union of subgroups $(\Lambda_i)_{i \in I}$, each of which is MD and efficient in $\Gamma$. Then $\Gamma$ is MD.
\end{theorem}

Theorem \ref{MD LS} had previously been obtained in \cite[Proposition 4.3]{treeable}, using the framework of existentially closed actions. In the spirit of keeping this paper as self-contained as possible, we give a different proof, that mirrors the one of Lubotzky and Shalom for property FD\footnote{The possibility of such an argument is also suggested in \cite[Remark 4.5]{treeable}.}.

We need to recall the coinduction technique from \cite[Appendix A]{MD:kechris}. Let $\Lambda < \Gamma$ and let $a \in A(\Lambda, X, \mu)$, then $\Gamma$ has a natural action on the product $X^{\Gamma/\Lambda}$ with the product measure $\mu^{\Gamma/\Lambda}$, which is denoted by $\cind_\Lambda^\Gamma(a) \in A(\Gamma, X^{\Gamma/\Lambda}, \mu^{\Gamma/\Lambda})$. We will only need to consider $\cind_\Lambda^\Gamma(a|_{\Lambda})$, where $a \in A(\Gamma, X, \mu)$. In this case, $\tilde{a} \coloneqq \cind_\Lambda^\Gamma(a|_{\Lambda}) \in A(\Gamma, X^{\Gamma/\Lambda}, \mu^{\Gamma/\Lambda})$ is given by
\[(g^{\tilde{a}} \cdot f)(h\Lambda) = g^a \cdot f(g^{-1}h\Lambda),\]
up to isomorphism \cite[Proposition A.3]{MD:kechris}.

\begin{lemma}
\label{MD dirun}
    Let $\Lambda < \Gamma$ and suppose that $\Lambda$ can be written as a directed union of subgroups $(\Lambda_i)_{i \in I}$. Then $\cind_{\Lambda_i}^\Gamma(a|_{\Lambda_i}) \to \cind_\Lambda^\Gamma(a|_\Lambda)$.
\end{lemma}

\begin{proof}
    Let $A_1, \ldots, A_n \subset X^{\Gamma/\Lambda}, g_1, \ldots, g_m \in \Gamma$ and $\epsilon > 0$. We claim that if $K$ is a sufficiently large subgroup of $\Lambda$, in the sense that it contains a prescribed finite set, then if we denote $\tilde{a} \coloneqq \cind_\Lambda^\Gamma(a|_\Lambda)$ and $\hat{a} \coloneqq \cind_K^\Gamma(a|_K)$, there exist $B_1, \ldots, B_n \subset X^{\Gamma/K}$ such that
    \eqref{weak containment} holds for all $i, j, k$. The lemma will follow since we can choose $K$ to be a term of the directed union giving $\Lambda$.

    The space $X^{\Gamma/\Lambda}$ is endowed with the product measure, so every Borel set $A$ can be approximated by a finite union of \emph{cylinders} $C$, that is, sets that decompose as products $\prod C^{h\Lambda}$, where $C^{h\Lambda}$ is a Borel subset of $X$, and $C^{h\Lambda} = X$ for all but finitely many $h \Lambda$. Since the $\Gamma$-action preserves cylinders, up to approximation, we may assume that each $A_j$ is a cylinder $A_j = \prod A_j^{h\Lambda}$. In this case, we claim that we can find the sets $B_1, \ldots, B_n$ so that \eqref{weak containment} is satisfied with $\epsilon = 0$.
    
    The action of $\Gamma$ on a cylinder takes the following form:
    \[g^{\tilde{a}}(A) = g^{\tilde{a}} \left( \prod_{h\Lambda \in \Gamma/\Lambda} A^{h\Lambda} \right) = \prod_{h\Lambda \in \Gamma/\Lambda} g^a (A^{g^{-1}h\Lambda}),\]
    hence
    \begin{equation}
    \label{mu cylinders}
        \mu^{\Gamma/\Lambda}(g_i^{\tilde{a}} (A_j) \cap A_k) = \prod_{h \Lambda \in \Gamma/\Lambda} \mu(g_i^a (A_j^{g_i^{-1}h\Lambda}) \cap A_k^{h \Lambda}).
    \end{equation}
    We let $H$ denote the set of all elements $h \Lambda \in \Gamma/\Lambda$ such that $A_j^{h \Lambda} \neq X$ for some $j$: this set is finite, and we identify it with a section $H \subset \Gamma$. Then the value of \eqref{mu cylinders} depends entirely on the action of $g_1, \ldots, g_n$ on $X$, and on $H\Lambda$.

    We define $B_j^{hK} = A_j^{h\Lambda}$ for $h \in H$, and $B_j^{hK} = X$ otherwise, and we choose $K \subset \Lambda$ to be large enough that for all $i$ and all $h, h' \in \Lambda$:
    \[g_i^{-1}h\Lambda = h'\Lambda \Leftrightarrow g_i^{-1}hK = h'K.\]
    Then the computation in \eqref{mu cylinders} gives the same value for $\hat{a}$ and the $B_j, B_k$, and we conclude.
\end{proof}

The following is an analog of \cite[Lemma 2.4]{Lubotzky:2004xw}.

\begin{lemma}
\label{MD efficient}
    Let $\Lambda < \Gamma$ be an efficient subgroup. Suppose that $\Lambda$ is MD. Then for every $a \in A(\Gamma, X, \mu)$ it holds that $\cind_\Lambda^\Gamma(a|_\Lambda) \in \overline{F(\Gamma)}$.
\end{lemma}

Given $\Lambda < \Gamma$, we denote by $s_{\Gamma, \Gamma/\Lambda, X}$ the shift $\Gamma \curvearrowright X^{\Gamma/\Lambda}$.

\begin{proof}
    Because $\Lambda$ is MD, there is a sequence of finite actions $a_i \in A(\Lambda, X, \mu)$ such that $a_i \to a|_{\Lambda}$. Because $\cind$ is continuous \cite[p. 502]{MD:kechris}, we have $\cind_\Lambda^\Gamma(a_i) \to \cind_\Lambda^\Gamma(a|_\Lambda)$; it remains to show that $\cind_\Lambda^\Gamma(a_i) \in \overline{F(\Gamma)}$ for all $i$.
    
    By \cite[Lemma 4.15]{MD:kechris}, there exist finite-index normal subgroups $N_i < \Lambda$ such that $a_i \prec s_{\Lambda, \Lambda/N_i, X}$. Hence $\cind_\Lambda^\Gamma(a_i) \prec \cind_\Lambda^\Gamma(s_{\Lambda, \Lambda/N_i, X})$ by \cite[Proposition A.1]{MD:kechris}, and this latter is isomorphic to $s_{\Gamma, \Gamma/N_i, X}$ by \cite[Proposition A.2]{MD:kechris}. By assumption $N_i$ is separable in $\Gamma$, hence there exists a sequence of nested finite-index subgroups $M_j$ that intersect to $N_i$, and $s_{\Gamma, \Gamma/N_i, X}$ is the weak limit of the finite actions $s_{\Gamma, \Gamma/M_j, X}$.
\end{proof}

\begin{proof}[Proof of Theorem \ref{MD LS}]
    Let $a \in A(\Gamma, X, \mu)$. By \cite[Theorem 1.1]{MD:bowen:tuckerdrob}, we have that $a \prec \cind_{\Lambda}^\Gamma(a|_{\Lambda}) \times p_{\Gamma/\Lambda}$. By \cite[Lemma 2.3]{MD:bowen:tuckerdrob}, because $\Gamma/\Lambda$ is residually finite, $p_{\Gamma/\Lambda} \prec \iota_\Gamma \times p_\Gamma$. Hence by \cite[Lemma 2.2]{MD:bowen:tuckerdrob}, it remains to show that $\cind_{\Lambda}^\Gamma(a|_{\Lambda}) \in \overline{F(\Gamma)}$. Now $\cind_{\Lambda_i}^\Gamma(a|_{\Lambda_i}) \to \cind_{\Lambda}^\Gamma(a|_{\Lambda})$ by Lemma \ref{MD dirun}, so we conclude by Lemma \ref{MD efficient}.
\end{proof}

\subsection{Amalgamated products}

It is known that property MD passes to free products \cite[Theorem 4.8]{MD:tuckerdrob}.

\begin{theorem}[Tucker-Drob]
\label{MD freeprod}
    A free product of two MD groups is MD. \qed
\end{theorem}

However, to tackle RAAGs on chordal graphs, we need to strengthen Theorem \ref{MD freeprod} to an analog of Corollary \ref{fd am}.

\begin{definition}
    If $\alpha \colon \Lambda \to \mathrm{Aut}(\Gamma)$ is an action by automorphisms, we say that $\Gamma$ has \emph{property MD relative to $\alpha$} (or relative to $\Lambda$, if the action $\alpha$ is clear from the context) if every p.m.p.\ action of $\Gamma$ is a weak limit of actions that factor through a finite quotient $\Gamma/N$ such that $\alpha_l(N) = N$ for all $l \in \Lambda$.
\end{definition}

As in Remark \ref{relative FD fg}, note that if $\Gamma$ is finitely generated and MD, then it is MD relative to $\Lambda$, for any $\alpha \colon \Lambda \to \mathrm{Aut}(\Gamma)$.



First we strengthen \cite[Theorem 1.4]{MD:bowen:tuckerdrob} to an analog of Theorem \ref{SS:semidirect} (that is \cite[Proposition 4.5 and Theorem 5.2]{Shulman:2023aa}).

\begin{theorem}
\label{MD amenableext iff}
    Let $\Gamma$ be a group and let $\Lambda$ be an amenable residually finite group acting on $\Gamma$. Then $\Gamma \rtimes \Lambda$ is MD if and only if $\Gamma$ is MD relative to $\Lambda$.
\end{theorem}

\begin{proof}
    Suppose that $\Gamma \rtimes \Lambda$ is MD, and let $a \in A(\Gamma, X, \mu)$. Then $\cind_{\Gamma}^{\Gamma \rtimes \Lambda}(a) \in \overline{F(\Gamma \rtimes \Lambda)}$. By \cite[p. 502]{MD:kechris}, we have $a \prec \cind_{\Gamma}^{\Gamma \rtimes \Lambda}(a)|_{\Gamma}$, and this is a limit of finite actions that factor through restrictions of quotients of $\Gamma \rtimes \Lambda$. Hence these actions factor through quotients by normal subgroups that are $\Lambda$-invariant, so $\Gamma$ is MD relative to $\Lambda$.

    The converse direction is only a slight strengthening of \cite[Theorem 1.4]{MD:bowen:tuckerdrob}, whose proof uses \cite[Theorem 1.1]{MD:bowen:tuckerdrob} combined with the argument in \cite[p. 488]{MD:kechris}. This argument only uses that the normal subgroups of $\Gamma$ appearing as kernels of a dense set of finite actions can be chosen to be normal in $\Gamma \rtimes \Lambda$, and this follows from the assumption of relative MD.
\end{proof}

We can now prove a relative version of Theorem \ref{MD freeprod}, which is an analog of Theorem \ref{fd fp}.

\begin{theorem}
\label{MD freeprod relative}
Let $\Gamma_1$ and $\Gamma_2$ be groups, and let $\Lambda$ be a group acting on both $\Gamma_1$ and $\Gamma_2$ by automorphisms.  Let $\Lambda$ act on the free product $\Gamma_1\Asterisk\Gamma_2$ via the induced action.  Then if $\Gamma_1$ and $\Gamma_2$ both have MD relative to $\Lambda$, so does $\Gamma_1\Asterisk\Gamma_2$.
\end{theorem}

\begin{proof}
    Let $a \in A(\Gamma_1 \Asterisk \Gamma_2, X, \mu)$. The proof of \cite[Theorem 4.8]{MD:tuckerdrob} approximates $a$ by finite actions, that factor through a finite quotient of the form $(\Gamma_1/N_1 \Asterisk \Gamma_2/N_2)/M$, where $N_i$ is a finite index normal subgroup of $\Gamma_i$, and $M$ is a finite index normal subgroup of $\Gamma_1/N_1 \Asterisk \Gamma_2/N_2$. Using that $\Gamma_i$ has MD relative to $\Lambda$, we may choose $N_i$ to be $\Lambda$-invariant, so that the kernel of $\Gamma_1 \Asterisk \Gamma_2 \to \Gamma_1 /N_1 \Asterisk \Gamma_2 / N_2$ is $\Lambda$-invariant. Now this is a finitely generated group, so we may replace $M$ with a finite index characteristic subgroup, which is in particular $\Lambda$-invariant. The kernel of $\Gamma_1 \Asterisk \Gamma_2 \to (\Gamma_1/N_1 \Asterisk \Gamma_2/N_2)/M$ is therefore $\Lambda$-invariant.
\end{proof}

This leads to the desired analog of Corollary \ref{fd am}, with the same proof.

\begin{corollary}
\label{MD amalgam}
Let $\Gamma_1$ and $\Gamma_2$ be MD groups, with a common amenable retract $\Lambda$.  Then $\Gamma_1\Asterisk_\Lambda \Gamma_2$ is MD. \qed
\end{corollary}

\subsection{New examples of MD groups}

We can finally prove Theorem \ref{intro thm md} from the introduction.

\begin{theorem}
\label{MD list}
    If $\Gamma$ is finitely generated and satisfies one of the following properties, then it has property (E)MD, and in particular it is stable in finite actions.
    \begin{enumerate}[(i)]
        \item\label{MD list manifold} $\Gamma = \pi_1(M)$, where $M$ is a connected manifold of dimension at most $3$;
        \item\label{MD list vfbc} $\Gamma$ is virtually free-by-cyclic;
        \item\label{MD list 1rel} $\Gamma$ is a one-relator group, and $\Gamma$ has either torsion, negative immersions, non-trivial center or a small cancellation relation;
        \item\label{MD list limit} $\Gamma$ is a limit group;
        \item\label{MD list raag} $\Gamma$ is a right-angled Artin group on a chordal graph.
    \end{enumerate}
\end{theorem}

We note that the result for limit groups was already known \cite[Theorem 4.4]{treeable}, and proved in the same way, using \cite[Proposition 4.3]{treeable} in place of our Theorem \ref{MD LS}.

\begin{proof}
We prove everything from MD, which directly implies stability in finite actions (Remark \ref{stability in finite actions}). EMD follows from Remark \ref{MD vs EMD}, since none of these groups have property (T) when they are infinite, as in that case they all virtually map onto $\Z$. This is by definition for virtually free-by-cyclic groups, while (infinite) one-relator groups, limit groups and RAAGs all have infinite abelianization. For (infinite) $3$-manifold fundamental groups, this follows from the arguments below.

Using Lemma \ref{MD virtually} and Theorems \ref{MD free amenable} and \ref{MD amenableext iff}, the same argument as Proposition \ref{vfbc fd} gives Item \eqref{MD list vfbc}. Item \eqref{MD list 1rel} is a special case. Item \eqref{MD list limit} follows from the same argument as Proposition \ref{limit fd}, using Theorems \ref{MD free amenable} and \ref{MD LS}. Item \eqref{MD list raag} follows from the same argument as Proposition \ref{raag fd}, using Theorems \ref{MD freeprod} and \ref{MD amalgam}.

Let now $M$ be a $3$-manifold such that $\pi_1(M)$ is finitely generated. As in the proof of Theorem \ref{manifold fd}, using Theorem \ref{MD freeprod}, we reduce to the case in wich $M$ is a compact, aspherical and has incompressible boundary. Since virtually free-by-cyclic groups are MD, \cite[Theorem 1.1]{kielak:linton:manifolds} covers the case of non-empty boundary. If $M$ is closed and not a graph manifold, then it is virtually fibered (see the references before Theorem \ref{nongraph fd}) and so it is MD: see \cite[Proposition 3.10]{ifsv}, although it follows also from Lemma \ref{MD virtually} and Theorem \ref{MD amenableext iff}.

So we may assume that $M$ is a closed graph manifold. The analog of Lemma \ref{graph fd easy} follows again from Lemma \ref{MD virtually} and Theorems \ref{MD free amenable}, \ref{MD LS} and \ref{MD amenableext iff}, as Seifert fibered manifolds have subgroup separable fundamental group and virtually map to $\Z$ with kernel a direct product of a free group and $\Z$; and virtual torus bundles follow more directly from the fact that amenable residually finite groups are MD (Theorem \ref{MD free amenable}). In the remaining cases, as in the proof of Proposition \ref{graph fd}, thanks to Lemma \ref{MD virtually} we pass to a finite index subgroup $\Gamma < \pi_1(M)$ such that $\Gamma = K \rtimes \Z$, where $K$ is a directed union of subgroups of $\Gamma$ that are fundamental groups of compact graph manifolds with non-empty boundary, hence MD, and efficient in $\Gamma$. Theorem \ref{MD LS} applies and we conclude.
\end{proof}

\footnotesize

\addcontentsline{toc}{section}{References}

\bibliography{Generalbib}

\vspace{0.5cm}

\normalsize

\noindent{\textsc{Department of Pure Mathematics and Mathematical Statistics, University of Cambridge, UK}}

\noindent{\textit{E-mail address:} \texttt{ff373@cam.ac.uk}}

\vspace{0.2cm}

\noindent{\textsc{Department of Mathematics,
University of {Hawai\kern.05em`\kern.05em\relax i} at M{\=a}noa}}

\noindent{\textit{E-mail address:} \texttt{rwillett@hawaii.edu}}

\end{document}